\documentclass[USenglish]{article}

\usepackage{fullpage}
\usepackage{amsmath,amsfonts,amsthm,amssymb}
\usepackage{color}
\usepackage[dvipsnames]{xcolor}
\usepackage{ifpdf}
\usepackage{psfrag}
\usepackage{graphicx,graphics}
\usepackage{hhtensor}
\usepackage{comment}
\usepackage[small]{caption}
\usepackage{subcaption}
\usepackage{tabularx}

\usepackage{xparse}
\usepackage{bigints}

\usepackage{url}
\usepackage{hyperref}
\usepackage{todonotes}
\usepackage{cleveref}
\newtheorem{theorem}{Theorem}[section]
\newtheorem{corollary}[theorem]{Corollary}
\newtheorem{lemma}[theorem]{Lemma}
\newtheorem{proposition}[theorem]{Proposition}
\newtheorem{prop}[theorem]{Proposition}
\newtheorem{remark}[theorem]{Remark}

\usepackage{cancel}
\usepackage{algorithm2e}
\usepackage{enumitem}

\usepackage{pgfplots}
\usepackage{pgfplotstable}
\usepackage{tikz}

\usetikzlibrary{fillbetween}
\usepgfplotslibrary{fillbetween}

\usetikzlibrary{shapes.misc}
\usetikzlibrary{math} 

\usetikzlibrary{arrows}

\newcommand{\R}{\mathbb R}

\newcommand\norm[1]{\left\lVert#1\right\rVert}

\NewDocumentCommand{\mat}{mo}{%
	\IfValueTF{#2}{%
		\underline{\underline{#1}}{#2}
	}{%
		\underline{\underline{#1}}\,
	}%
}

\usepackage{bbm}

\def\bu{\boldsymbol{u}}

\definecolor{darkspringgreen}{rgb}{0., 0.55, 0.3}
\definecolor{dartmouthgreen}{rgb}{0.05, 0.5, 0.06}
\definecolor{etonblue}{rgb}{0.59, 0.78, 0.64}
\definecolor{airforceblue}{rgb}{0., 0.4, 0.66}
\definecolor{arylideyellow}{rgb}{0.91, 0.84, 0.42}
\definecolor{emerald}{rgb}{0.31, 0.78, 0.47}
\definecolor{uclagold}{rgb}{1.0, 0.7, 0.0}
\definecolor{cadmiumorange}{rgb}{0.93, 0.53, 0.18}

\newcommand{\lopd}[0]{\mathcal{L}_\Delta}
\newcommand{\lopdt}[0]{\mathcal{L}_{\Delta}}
\newcommand{\usol}[0]{\underline{\uvec{u}}_\Delta}

\newcommand{\uapp}[0]{\uvec{u}_h}

\newcommand{\tess}[0]{\mathcal{T}_h}

\newcommand{\uvec}[2][3]{\boldsymbol{#2\mkern-#1mu}\mkern#1mu}

\newcommand{\dt}{\Delta t}

\newcommand{\bbu}{\underline{\uvec{u}}}

\newcommand{\by}{\uvec{y}}

\newcommand{\bG}{{\uvec{G}}}

\newcommand{\vecz}{\underline{0}}
\newcommand{\matz}{\underline{\underline{0}}}
\newcommand{\vecbeta}{\underline{\beta}}

\newcommand{\undu}{\underline{\uvec{u}}}

\newcommand{\undv}{\underline{\uvec{v}}}
\newcommand{\undw}{\underline{\uvec{w}}}
\newcommand{\undr}{\underline{\uvec{r}}}

\newcommand{\usolp}[0]{\underline{\uvec{u}}_{\Delta}^{(p)}}
\newcommand{\Xp}[0]{X^{(p)}}
\newcommand{\Yp}[0]{Y^{(p)}}
\newcommand{\uex}[0]{\underline{\uvec{u}}_{ex}}
\newcommand{\uexp}[0]{\underline{\uvec{u}}_{ex}^{(p)}}

\newcommand{\embep}[0]{\mathcal{E}^{(p)}}
\newcommand{\projp}[0]{\Pi^{(p)}}

\newcommand{\undG}[0]{\underline{\uvec{G}}}

\newcommand{\hphi}{\widehat{\phi}}
\newcommand{\uhphi}{\underline{\widehat{\phi}}}

\newcommand{\ADERIWF}{ADER-IWF}
\newcommand{\ADERRK}{ADER-IWF-RK} 

\newcommand{\Bp}{\mathcal{B}} 
\newcommand{\Cp}{\mathcal{C}} 
\newcommand{\Dp}{\mathcal{D}} 

\newcommand{\IM}{\mathcal{H}} 
\newcommand{\co}{u} 
\newcommand{\va}{v} 

\newcommand{\SRK}{S} 

\newcommand{\ADERdu}{ADERdu}
\newcommand{\ADERu}{ADERu} 
\newcommand{\cADER}{cADER} 

\newcommand{\CC}{C5}
\begin{document}
	\title{On improving the efficiency of ADER methods}

	\author{Maria Han Veiga\thanks{Department of Mathematics,  Ohio State University,  231 W 18th Ave,  Columbus,  OH 43210,  United States of America. Email: hanveiga.1@osu.edu },
	Lorenzo Micalizzi\thanks{Institute of Mathematics, University of Zurich, Winterthurerstrasse 190, Zurich, 8057, Switzerland. Email: lorenzo.micalizzi@math.uzh.ch},  
	Davide Torlo\thanks{SISSA mathLab, SISSA, via Bonomea 265, Trieste, 34136, Italy. Email: davide.torlo@sissa.it}
	}
	\date{}
	\maketitle
	
	\begin{abstract}
		The (modern) arbitrary derivative (ADER) approach is a popular technique for the numerical solution of differential problems based on iteratively solving an implicit discretization of their weak formulation. 
		In this work, focusing on an ODE context, we investigate several strategies to improve this approach. 
		Our initial emphasis is on the order of accuracy of the method in connection with the polynomial discretization of the weak formulation.  We demonstrate that precise choices lead to higher-order convergences in comparison to the existing literature.
		Then, we put ADER methods into a Deferred Correction (DeC) formalism. This allows to determine the optimal number of iterations, which is equal to the formal order of accuracy of the method, and to introduce efficient $p$-adaptive modifications. These are defined by matching the order of accuracy achieved and the degree of the polynomial reconstruction at each iteration.
		We provide analytical and numerical results, including the stability analysis of the new modified methods, the investigation of the computational efficiency, an application to adaptivity and an application to hyperbolic PDEs with a Spectral Difference (SD) space discretization.
	\end{abstract}

\section{Introduction}
\label{sec:Introduction}
Differential problems play a crucial role in the world of modeling of natural and technological processes. 
In fact, countless systems in many applications are described through ODEs and PDEs. 
Due to the variety of existing phenomena to be modeled, no unified theory exists and very different system-dependent differential problems are reported in literature. 
Despite such heterogeneity,  the vast majority of the mentioned mathematical models share a common feature: the impossibility to find analytical solutions, apart from very basic exceptions. 
At the moment, the only possible way to quantitatively cope with such models is to rely on numerical methods. 
Nevertheless, numerical approaches come with their own set of challenges. 
One of the biggest issues in such context is given by the computational cost associated to very accurate approximations of the sought analytical solutions, 
which require very refined discretizations, with consequent need for huge computational resources, in terms of computational time and memory consumption.

A classical expedient to reduce the computational cost, making hence the simulations more accessible, consists in the adoption of high order methods. 
In fact, such methods are well-known to require, at least on smooth problems, less computational resources for a fixed discretization error. 
This justifies the large interest, lately shown by the scientific community, in the construction of (arbitrary) high order frameworks for the numerical solution of differential problems, both for ODEs, with Runge-Kutta (RK) methods  \cite{butcher2016numerical,hairer1987solving,wanner1996solving}, predictor-corrector methods \cite{gautschi2011numerical} and multistep methods \cite{hairer1987solving}, and PDEs, with 
finite difference methods \cite{leveque2007finite,godlewski2013numerical}, finite volume methods \cite{leveque2002finite,eymard2000finite,versteeg2007introduction,godlewski2013numerical}, finite element methods \cite{ern2004theory, hesthaven2007nodal,cockburn2000development,gottlieb2001spectral}, SD methods \cite{jameson2010,veiga2021arbitrary,velasco2022spectral}, Residual Distribution methods \cite{RemiMarioRD,Decremi,abgrall2019high,abgrall2020high} and ADER methods \cite{balsara2009efficient,gaburro2021unified,boscheri2022continuous}.

Compared to the other frameworks, ADER is relatively new. In fact, originally introduced by Titarev and Toro \cite{toro2001towards,titarev2002ader} in 2001 for hyperbolic problems, as a time integrator for finite volume formulations, based on the Cauchy--Kovalevskaya theorem, it became popular in its modern formulation presented in 2008 \cite{dumbser2008unified}. 
The method consists in constructing nonlinear systems of algebraic equations, associated to high order discretizations of the weak formulation of the analytical problem under investigation, which are solved through an iterative procedure.
Despite its recent definition, the approach has been proven to be robust and reliable through numerous applications to many different fields such as ODEs contexts \cite{han2021dec}, Eulerian-Lagrangian formulations on unstructured moving meshes \cite{Lagrange2D,ALEMOOD1,ALEMOOD2,boscheri2019high,ArepoTN}, structure-preserving \cite{gaburro2022high,DOOM}, magnetohydrodynamics \cite{veiga2021arbitrary,balsara2009efficient,gaburro2021posteriori,dumbser2009very,zanotti2015space}, solid mechanics \cite{Busto2020,DUMBSER2017298}, compressible fluid-dynamics \cite{dumbser2007FVStiff,boscheri2022continuous,titarevtoro,velasco2022spectral,dumbser2010arbitrary,toro3,hidalgo2011ader}, aeroacustics \cite{dumbser2005ader,schwartzkopff2004fast,schwartzkopff2002ader}, adaptivity \cite{DOOM,zanotti2015space}, with parallel implementations used in the context of large scale simulations for real test cases \cite{dumbser2018efficient,rio2021massively}.
The mentioned literature does not pretend to be complete, however, it does indeed provide a clear proof of the level of maturity, of flexibility and of robustness of the ADER approach.


The goal of this paper is to introduce several strategies to drastically reduce the computational cost of ADER methods because, even if they have shown good performances in the context of large scale simulations, still they leave the door open for great improvements under many points of view, in particular at the level of the computational efficiency.
More in detail, in an ODE context, we discuss the following main contributions:

\begin{itemize}
	\item We start by investigating the role of the polynomial reconstruction of the numerical solution adopted in the context of the discretization of the weak formulation of the differential problem, leading to the definition of the ADER nonlinear system, later solved iteratively. 
	In particular,  we prove that careful choices on bases and quadrature points lead to schemes with order of accuracy higher than expected according to classical literature.  This is done by reinterpreting the ADER nonlinear system as an implicit high order RK, which is analyzed in depth.
	The computational advantage achieved in this context is far from being negligible. In fact, usually in the context of ADER methods, order $P$ is achieved by selecting polynomial discretizations of degree $P-1$ in time \cite{han2021dec,veiga2021arbitrary,velasco2022spectral,zanotti2015solving,fernandez2022arbitrary,rannabauer2018ader,dumbser2018efficient}. We show that the same accuracy can be obtained with a polynomial degree which is approximately $\frac{P}{2}$, with related saving of computational resources both in terms of memory and time. 
	A further result, in this context, concerns the link between ADER methods and Lobatto IIIC RK methods \cite{wanner1996solving}.
	\item We characterize the ADER methods obtained for general polynomial bases.  
	As ADER methods do not require particular constraints with respect to the basis functions adopted for the discretization of the weak formulation of the differential problem,  one could wonder whether the adoption of particular bases (e.g., modal bases) could lead to better schemes. We show that there is a strong link between ADER methods obtained with arbitrary bases and ADER methods with nodal bases, 
	which are proven to be equivalent under the assumption that the adopted quadratures coincide.
	\item By using the fact that ADER methods can be put in a DeC framework \cite{Decremi,loredavide,DOOM}, we exploit the DeC formulation of ADER methods, to construct efficient modifications of such schemes, following the idea introduced in \cite{loredavide} and generalized in \cite{DOOM}. 
	The new modified schemes are based on increasing the degree of the polynomial reconstruction of the numerical solution along the iterative procedure in such a way that the achieved accuracy matches the discretization accuracy of the ADER iterations.
	Further, we show how to recast the new schemes as explicit RK methods, by defining the related Butcher tableaux, and we study their linear stability, proving that the efficient modifications do not affect stability. 
	\item Lastly,  we introduce a natural way to perform $p$-adaptivity. 
	This aspect is particularly interesting, in fact, if on the one hand we have remarked how the adoption of high order methods results in a considerable computational advantage for a fixed discretization error, on the other hand one must notice that the error is not known a priori. 
	In principle, users are not strictly interested in the order of the method adopted, but rather on the final error being smaller than a given tolerance. 
	Therefore, in practical applications, high order methods should be used in combination with adaptive strategies able to estimate the error and to select the discretization and/or the order accordingly. 
	Unfortunately, such adaptive strategies are in general not easy to design but the particular structure of the proposed modified ADER methods offers a natural way to do it.
\end{itemize}

A rich variety of numerical tests is provided to show the computational advantages of the proposed modifications, with large registered speed-ups.

The structure of this work is the following: we introduce the ADER methods for ODEs in Section~\ref{sec:ADER}.
In Section~\ref{sec:analytical_results}, we study the accuracy of such methods with respect to the discretization of the weak formulation, showing that the adoption of Gauss--Lobatto (GLB) and Gauss--Legendre (GLG) bases leads to accuracy higher than expected in Section~\ref{app:proof_L2_orders} and characterizing the ADER methods obtained with general polynomial bases in Section~\ref{app:beyond_nodal}.
In Section~\ref{sec:DeC}, we introduce the DeC framework and show how the ADER methods can be reinterpreted as DeC methods, hence fixing the optimal number of iterations.
We present thus efficient modifications of ADER methods in Section~\ref{sec:ADERNEW} and their application to $p$-adaptivity. 
Then, in Section~\ref{sec:RK}, we show how the modified methods can be written in RK form, defining their Butcher tableaux, and we study their linear stability. 
In Section~\ref{sec:SD}, we describe an application of the ADER methods to hyperbolic PDEs with an SD space discretization via the method of lines.
The methods are validated in Section~\ref{sec:Numerics} and, finally, Section~\ref{sec:Conclusions} is left for conclusions and further developments.

Moreover, \ref{app:proofs} is dedicated to the proofs of several ``minor'' results presented along the paper. To maintain the flow of information smooth, we have placed these proofs in an appendix, so that readers can focus on the main concepts and conclusions, referring to the proofs at their convenience.
In \ref{app:GLB_Lobatto_IIIC}, the proof of an equivalence theorem between particular ADER methods and Lobatto IIIC schemes is reported, while, an overview of all symbols used in the paper can be found in \ref{app:notation}.

\section{ADER}
\label{sec:ADER}
In this section, we will present the original ADER method \cite{dumbser2008unified} in its simplified version for systems of ODEs, firstly described in \cite{han2021dec}.
Let us consider the following system
\begin{equation}
	\label{eq:ODE}
	\begin{cases}
		\frac{d}{dt}\uvec{u}(t) = \uvec{G}(t,\uvec{u}(t)),\quad t\in[0,T], \\
		\uvec{u}(0)=\uvec{z},
	\end{cases}
\end{equation}
where $\uvec{u} :\R^+_0 \to \mathbb{R}^Q$ is the sought solution, $\uvec{z} \in \R^Q$ the initial condition and $\uvec{G}: \R^+_0 \times \R^Q \to \R^Q$ a function that is continuous and Lipschitz-continuous with respect to $\uvec{u}$ uniformly with respect to $t$ with a Lipschitz constant $C_{Lip}$.
As usual in the context of one-step methods, we focus on a generic time interval $[t_n,t_{n+1}]$, with $t_{n+1}=t_{n}+\Delta t$, and we look for a recipe for $\uvec{u}_{n+1}\approx \uvec{u}(t_{n+1})$ from $\uvec{u}_{n}\approx \uvec{u}(t_{n})$.
In particular, as commonly done in the context of consistency analyses, we assume $\uvec{u}_{n}= \uvec{u}(t_{n})$.

\subsection{Method}
We consider the weak formulation of \eqref{eq:ODE} over $[t_n,t_{n+1}]$, obtained by multiplying the equation by a smooth test function $\psi(t)$ and applying integration by parts on the time derivative term 
\begin{equation}
	\psi(t_{n+1})\uvec{u}(t_{n+1})-\psi(t_n)\uvec{u}_n-\int_{t_n}^{t_{n+1}}\left(\frac{d}{dt}\psi(t)\right) \uvec{u}(t)dt - \int_{t_n}^{t_{n+1}} \psi(t)\uvec{G}(t,\uvec{u}(t))  dt = \uvec{0},
	\label{eq:weakproblem}
\end{equation}
where we remark that $\uvec{u}_{n}= \uvec{u}(t_{n})$ is known. 
We introduce now $M+1$ subtimenodes $t^m$ for $m=0,1,\dots,M$ in the interval $[t_n,t_{n+1}]$ such that 
\begin{equation}
	t_{n}\leq t^0 < t^1 < \dots < t^M\leq t_{n+1} 
\end{equation}
and we denote by $\uvec{u}^m$ an approximation of the exact solution $\uvec{u}(t)$ in $t^m$.
We pass to a discrete setting by projecting \eqref{eq:weakproblem} onto a finite dimensional functional space.
In particular, we replace $\uvec{u}(t)$ and $\uvec{G}(t,\uvec{u}(t))$ by their Lagrange interpolating polynomials of degree $M$ associated to the $M+1$ subtimenodes
\begin{align}
	\uapp(t):=\sum_{m=0}^M \uvec{u}^{m} \psi^{m}(t), \label{eq:uh} \qquad 
	\uvec{G}_h(t):=\sum_{m=0}^M \uvec{G}(t^m, \uvec{u}^{m}) \psi^{m}(t),
\end{align}
leading to the ADER implicit weak form (\ADERIWF)
\begin{align}
	\begin{split}
		\sum_{m=0}^M \Bigg[ \psi^{\ell}(t_{n+1})\psi^m(t_{n+1}) & -\int_{t_n}^{t_{n+1}} \left(\frac{d}{dt}\psi^\ell(t)\right)\psi^m(t)dt   \Bigg]\uvec{u}^{m}-\psi^\ell(t_n) \uvec{u}_n \\
		&- \sum_{m=0}^M \left( \int_{t_n}^{t_{n+1}} \psi^\ell(t)\psi^m(t) dt \right) \uvec{G}(t^m,\uvec{u}^m) =\uvec{0}, \quad \forall \ell=0,\dots,M,
	\end{split}
	\label{eq:weakproblemdiscrete}
\end{align}
which is a nonlinear system in the unknowns $\uvec{u}^{m}$. 

As we are going to see in the next subsection, such system can be put in a compact matrix formulation and, under time step restrictions, it can be solved through an explicit iterative strategy. 
Its solution consists of the coefficients of the continuous representation \eqref{eq:uh} of the numerical solution $\uvec{u}_h$, through which we get $\uvec{u}_{n+1}:=\uapp(t_{n+1})$. 
The order of accuracy of $\uvec{u}_{n+1}$ with respect to the exact solution $\uvec{u}(t_{n+1})$,  denoted by $N$, depends on the choice of the subtimenodes and on the quadrature rule used for integrals.
Summarizing, the ADER method consists in solving iteratively the \ADERIWF~\eqref{eq:weakproblemdiscrete} with respect to the coefficients $\uvec{u}^{m}$, which are later used to compute $\uvec{u}_{n+1}$.

\begin{remark}[On the order of accuracy]
	Several choices of subtimenodes and quadrature rules are possible.
	A generic distribution of $M+1$ subtimenodes $t^m$ in the interval $[t_{n},t_{n+1}]$, e.g., equispaced, guarantees, in general, an order of accuracy of the resulting ADER method equal to $N=M+1$. 
	However, particular choices yield higher accuracy, for example if we choose, both for the basis function definitions and for the quadrature rule, to use $M+1$ GLB subtimenodes we obtain order of accuracy equal to $N=2M$, while with $M+1$ GLG nodes we get an accuracy of $N=2M+1$. The proofs of the accuracy are based on theoretical results in the context of the RK methods and are presented in Section~\ref{sec:analytical_results}, in Theorems \ref{th:order_GLB} and \ref{th:order_GLG}.
\end{remark}

\begin{remark}[On the computation of $\uvec{u}_{n+1}$]
	The \ADERIWF~\eqref{eq:weakproblemdiscrete}, together with the reconstruction \eqref{eq:uh} of $\uvec{u}_h$ in $t_{n+1}$, could be interpreted as an implicit high order RK, referred as \ADERRK, as we are going to show in Theorem \ref{th:l2isRK}.
	In particular, in the context of the computation of $\uvec{u}_{n+1}$ after the solution of the nonlinear system, the final interpolation could be equivalently replaced by the following integration
	\begin{align}
		\label{eq:integration}
		\begin{split}
			\uvec{u}_{n+1}&=\uvec{u}_{n}+\int_{t_n}^{t_{n+1}} \uvec{G}_h(t)dt				   =\uvec{u}_{n}+\sum_{m=0}^M \left( \int_{t_n}^{t_{n+1}} \psi^m(t) dt \right)\uvec{G}(t^{m}, \uvec{u}^{m}).
		\end{split}
	\end{align}		
	
	In the ADER formulation for hyperbolic PDEs introduced in \cite{dumbser2008unified}, the final integration allows communications between neighboring cells after an iterative procedure used to solve local weak formulations in space-time control volumes.

\end{remark}

\begin{remark}[On the difference between ADER and \ADERRK~and non-suitability for stiff problems]
	The ADER method is based on an explicit iterative procedure for the solution of the nonlinear system \eqref{eq:weakproblemdiscrete} and it has, hence, an explicit character, not suited for stiff problems. On the other hand, the \ADERRK~is an implicit method. Therefore, an ADER method and its associated \ADERRK~method do not share the same properties in terms of stability.
	
	Let us notice that implicit versions of ADER, based on an implicit treatment of $\uvec{G}$ in the context of the iterative procedure and, hence, suitable for stiff problems, exist in literature \cite{han2021dec} but they will not be investigated in this work.
	The application of the efficient modifications proposed in this work to implicit ADER methods is currently under investigation.
\end{remark}

\subsection{Matrix formulation and explicit iterative solution}
The \ADERIWF~\eqref{eq:weakproblemdiscrete} can be recast in the following matrix formulation
\begin{equation}
	B \undu - \undr - \Delta t \Lambda  \undG(\undu) =\uvec{0},
	\label{eq:ADER_system}
\end{equation}
where $B$ and $\Lambda$ are matrices, while, $\undu$, $\undr$ and $\undG(\undu)$ are vectors, defined by
\begin{align}
	\begin{split}
		B_{\ell,m}:&= \psi^\ell(t_{n+1})\psi^m(t_{n+1})- \int_{t_{n}}^{t_{n+1}} \left(\frac{d}{dt}\psi^\ell(t)\right)\psi^m(t)dt =\widehat{\psi}^\ell(1) \widehat{\psi}^m(1)- \int_{0}^{1} \left(\frac{d}{d\xi}\widehat{\psi}^\ell(\xi)\right)\widehat{\psi}^m(\xi)d\xi ,\\
		\Lambda_{\ell,m}:&= \frac{1}{\Delta t}\int_{t_{n}}^{t_{n+1}} \psi^\ell(t)\psi^m(t) dt=\int_{0}^{1} \widehat{\psi}^\ell(\xi)\widehat{\psi}^m(\xi)d\xi\\
		\undu:&=\begin{pmatrix}
			\uvec{u}^0\\
			\vdots\\
			\uvec{u}^M
		\end{pmatrix},\quad
		\undr:=\begin{pmatrix}
			\psi^0(t_n)\uvec{u}_n\\
			\vdots\\
			\psi^M(t_n)\uvec{u}_n
		\end{pmatrix},\quad
		\undG(\undu):=\begin{pmatrix}
			\uvec{G}(t^0,\uvec{u}^0)\\
			\vdots\\
			\uvec{G}(t^M,\uvec{u}^M)
		\end{pmatrix}
	\end{split}
	\label{eq:ADER_structure}
\end{align}
with $\widehat{\psi}^m(\xi):=\psi^m(t_{n}+\Delta t \xi)$ being the Lagrange basis functions remapped onto the interval $[0,1].$ 
For simplicity, the matrices $B$ and $\Lambda$ were defined for a scalar problem, they need to be block expanded for a vectorial problem.
Furthermore, inverting the matrix $B$, from \eqref{eq:ADER_system} we get
\begin{equation}
	\undu-\undu_n-\Delta t B^{-1}\Lambda  \underline{\uvec{G}}(\underline{\uvec{u}})=\uvec{0}, \quad \text{with} \quad \undu_n:=\begin{pmatrix}
		\uvec{u}_n\\
		\vdots\\
		\uvec{u}_n
	\end{pmatrix},
	\label{eq:ADER_system_final}
\end{equation}
thanks to the following proposition, whose proof can be found in \ref{app:r}.
\begin{proposition}\label{prop:r}
	If $B$ and $\uvec{r}$ are defined as in \eqref{eq:ADER_structure} and $\undu_n$ as in \eqref{eq:ADER_system_final}, then
	$B^{-1}\uvec{r}=\undu_n.$
\end{proposition}
For $\Delta t$ small enough, system \eqref{eq:ADER_system_final} has a unique solution, which can be obtained as the limit of the following explicit iterative procedure
\begin{equation}
	\undu^{(p)}:=\undu_n+\Delta t B^{-1}\Lambda  \underline{\uvec{G}}(\undu^{(p-1)}), \quad p\geq 1,
	\label{eq:ADER_Picard}
\end{equation}
which converges independently of the chosen initial guess $\undu^{(0)}$, due to the following proposition, whose proof can be found in \ref{app:iterative_procedure}.
\begin{proposition}[Convergence of the iterative procedure]\label{prop:iterative_procedure}
	For $\Delta t$ small enough, a unique solution to \eqref{eq:ADER_system_final} exists and it coincides with the limit of the iterative procedure \eqref{eq:ADER_Picard}. 
	In particular, the convergence is ensured for $\Delta t< \frac{1}{\norm{B^{-1}\Lambda}_\infty C_{Lip}}$.
\end{proposition}
The time step restriction resulting from Proposition \ref{prop:iterative_procedure} amounts to a classical time step restriction for explicit methods. In fact, the step size $\Delta t$ is constrained by the inverse of the Lipschitz constant $C_{Lip}$ of $\uvec{G}$ up to a constant, $\norm{B^{-1}\Lambda}_\infty$, independent of $\Delta t$. Let us notice that such estimate for the upper bound of $\Delta t$ might not be optimal and could be improved, for example, by choosing other norms in the context of the proof.

Let us notice that, since the expected discretization accuracy of $\uvec{u}_{n+1}:=\uapp(t_{n+1})$ is $N$, there is no need to  solve the nonlinear system \eqref{eq:ADER_system_final} up to a tolerance that is stricter than the discretization error corresponding to such accuracy. In general, obtaining an $N$-th order accurate approximation of the solution of \eqref{eq:ADER_system_final} is sufficient and this is possible by performing exactly $N$ iterations of \eqref{eq:ADER_Picard}, as we are going to see in Section~\ref{sec:DeC}, by putting the ADER method in a DeC formalism and showing that each iteration increases the order of accuracy by one.

Before continuing, we state here a result concerning the invertibility of the matrix $B$. 
\begin{theorem}[Invertibility of $B$]\label{th:invB}
	The ADER matrix $B$, defined for any generic basis $\left\lbrace\hphi^{m}(\xi)\right\rbrace_{m=0,\dots,M}$ of the space of the polynomials of degree $M$ over $[0,1]$ as 
	\begin{align}
		B_{\ell,m}:=\widehat{\phi}^\ell(1) \widehat{\phi}^m(1)- \int_{0}^{1} \left(\frac{d}{d\xi}\widehat{\phi}^\ell(\xi)\right)\widehat{\phi}^m(\xi)d\xi,
	\end{align}
	is invertible.
\end{theorem}
The construction of ADER methods is strongly based on the assumption that the matrix $B$ is nonsingular, however, the problem of the existence of its inverse has been, up to authors' knowledge, never investigated in literature. 
In fact, the proof of Theorem~\ref{th:invB} is less immediate than the proofs of the invertibility of classical mass and stiffness matrices of standard finite element formulations and can be found in \ref{app:invertB}.

\section{Accuracy of \ADERIWF}
\label{sec:analytical_results}
This section is divided into two subsections. In the first one, we will discuss the order of accuracy of ADER methods for equispaced, GLB and GLG subtimenodes. In the second one, we will present an original result concerning ADER methods with arbitrary bases.

Both subsections address important aspects in the context of the computational efficiency of ADER methods. In fact, in several works \cite{han2021dec,veiga2021arbitrary,velasco2022spectral,zanotti2015solving,fernandez2022arbitrary,rannabauer2018ader,dumbser2018efficient}, $(M+1)$-th order of accuracy has been obtained via $M+1$ GLG subtimenodes leading to a polynomial reconstruction in time of degree $M$; however, as we are going to show, nearly half of the subtimenodes (with consequent shorter computational times and smaller memory consumption) can guarantee the same accuracy order.
In particular, in Theorem \ref{th:order_GLB} and Theorem \ref{th:order_GLG}, we show that $M+1$ subtimenodes correspond to an accuracy order equal to $2M$ and $2M+1$, respectively for GLB and GLG subtimenodes if the associated quadrature formulas are adopted to compute the integrals.

Moreover, we show in Theorem \ref{th:link} that any general set of polynomial basis functions (not necessarily Lagrangian) lead to the same schemes defined above, 
under mild assumptions.

The quadrature formulas used to compute the integrals in the ADER terms will play a crucial role in this section. Whenever not specified, we assume an exact integration.

\subsection{Accuracy of ADER with nodal bases}\label{app:proof_L2_orders}
In this subsection, we will show that the order of accuracy of the ADER methods with $M+1$ subtimenodes is $N=M+1$ for equispaced subtimenodes with exact quadrature and, respectively, $N=2M$ and $N=2M+1$ for GLB and GLG subtimenodes, when adopting the associated quadrature formulas.
We remark that the accuracy of an ADER method is referred to the approximation $\uvec{u}_{n+1}$, obtained via the reconstruction \eqref{eq:uh}, after the solution of the \ADERIWF~\eqref{eq:weakproblemdiscrete}. 
In the following, we will show how it is possible to associate an implicit high order RK scheme with $S=M+1$ stages to a general ADER method.

Let us recall the structure of a RK scheme  \cite{butcher1964implicit} with $S$ stages 
\begin{equation}\label{eq:RK_ACCURACY_NODAL}
	\begin{cases}
		\by^s = \uvec{u}_n + \dt \sum_{r = 0}^{S-1} a_{s,r} \bG(t_n+c_r \dt, \by^r), \quad \text{for }s=0,\dots,S-1,\\
		\uvec{u}_{n+1}  = \uvec{u}_n + \dt \sum_{r = 0}^{S-1} b_{r} \bG(t_n+c_r \dt, \by^r),
	\end{cases}	 
\end{equation}
where $a_{s,r}$, $b_r$ and  $c_r$ for $s,r=0,\dots ,S-1$ are some coefficients that characterize the method, usually stored in matrix $A$ and vectors $\uvec{b}$ and $\uvec{c}$ respectively. The associated Butcher tableaux defines the method as well 
$$\begin{array}{c|c}
	\uvec{c}& A  \\
	\hline
	& \uvec{b} 
\end{array}.
$$

In the following, we will denote by $\xi^{m}:=\frac{t^{m}-t_n}{\Delta t}$ the ADER subtimenodes rescaled into $[0,1]$, and by $w_m:=\int_0^1 \widehat{\psi}^m d\xi$ the quadrature weights of the induced quadrature formula.
Further, we collect the $\xi^{m}$ values in the vector $\underline{\beta}:=\left( \xi^0 , \dots, \xi^M\right)^T$.
We can now present the first result of this subsection concerning the link between \ADERIWF~and implicit RK methods.

\begin{theorem}[The \ADERIWF~\eqref{eq:weakproblemdiscrete} is an implicit RK]\label{th:l2isRK}
	Solving the \ADERIWF~\eqref{eq:weakproblemdiscrete} and then using reconstruction \eqref{eq:uh} to get $\uvec{u}_{n+1}$ with $M\geq 1$ is equivalent to the implicit RK method characterized by $A:=B^{-1}\Lambda$, $\uvec{c}:=\vecbeta$ and $\uvec{b}$ defined by $b_m:=w_m$. 
\end{theorem}
The proof can be found in \ref{app:l2isRK}.	
Given an ADER method, we will refer to the associated RK method as \ADERRK.
It is possible to show that such RK methods fulfill classical properties of high order RK schemes, for example the following proposition holds. 
\begin{proposition}\label{prop:HORK_condition}
	A general \ADERRK~method with $M\geq 1$ satisfies condition 
	\begin{equation}
		\sum_{j=0}^{S-1}a_{i,j}=c_i, \quad i=0,\dots, S-1.
		\label{eq:semirequiredforRK}
	\end{equation}	
\end{proposition}	
The proof can be found in \ref{app:HORK_condition}.
Further, thanks to basic interpolation properties, it is easy to show that an \ADERRK~method with $M+1$ subtimenodes is at least of order $N=M+1$, as summarized in the next proposition.
\begin{proposition}\label{prop:ADERRK_Mp1}
	\ADERRK~with $M+1$ subtimenodes is at least of order $M+1$.
\end{proposition}
The proof can be found in \ref{app:ADERRK_Mp1}.
Therefore, a general distribution of subtimenodes, e.g., equispaced, yields accuracy $M+1$.
Nonetheless,  we will show that particular choices of subtimenodes lead to higher convergence rates. 
The order of accuracy of the \ADERRK~schemes can be verified with the help of some conditions on the coefficients $A$, $\uvec{b}$ and $\uvec{c}$. 
Let us, hence, define, for $p,\eta,\zeta \in \mathbb{N}$, the following conditions \cite{wanner1996solving}
\begin{subequations}
	\begin{align}
		\Bp(p):\qquad  & \sum_{i=0}^{S-1} b_i c_i^{z-1}=\frac1z,\qquad & z=1,\dots,p; \label{eq:condRKB}\\
		\Cp(\eta):\qquad  & \sum_{j=0}^{S-1} a_{i,j} c_j^{z-1}=\frac{c_i^z}{z},\qquad & i=0,\dots,S-1,\,z=1,\dots,\eta;\label{eq:condRKC}\\
		\Dp(\zeta):\qquad  & \sum_{i=0}^{S-1} b_i c_i^{z-1}a_{i,j}=\frac{b_j}{z}(1-c_j^z),\qquad &j=0,\dots,S-1,\, z=1,\dots,\zeta, 			\label{eq:condRKD}
	\end{align}
\end{subequations}
which allow to easily verify the order of accuracy of implicit RK schemes through the following theorem.
\begin{theorem}[Butcher 1964 \cite{butcher1964implicit}]\label{th:butcher_order}
	If the coefficients $a_{i,j},b_i,c_i$ of a RK scheme satisfy $\Bp(p)$, $\Cp(\eta)$ and $\Dp(\zeta)$ with $p\leq \eta +\zeta +1$ and $p\leq 2\eta +2$, then the method is of order $p$.
\end{theorem}
We recall that the condition $\mathcal{B}(p)$ is necessary to reach order $p$, while $\mathcal{C}(\eta)$ and $\mathcal{D}(\zeta)$ are only sufficient conditions \cite{hairer1987solving}. Typically, explicit methods do not fulfill them.

We are now going to prove two preliminary results concerning the \ADERRK~schemes, which will be later used to show their accuracy for GLB and GLG subtimenodes.

\begin{lemma}\label{lem:conditionsRKC_ADER}
	The \ADERRK~methods, with integral terms evaluated through a quadrature formula $\mathcal{Q}\left\lbrace \cdot \right\rbrace$ with degree of exactness at least $2S-3$, satisfy $\Cp(S-1)$.
\end{lemma}
\begin{proof}
	Condition $\Cp(S-1)$ can be verified by explicit computations in matricial form. Indeed, since $A=B^{-1}\Lambda$, then $\Cp(S-1)$ can be rewritten as
	\begin{subequations}
		
		\begin{equation}
			A\, \underline{\uvec{c}^{z-1}}  = \frac{1}{z}\underline{\uvec{c}^{z}} \Longleftrightarrow \Lambda \, \underline{\uvec{c}^{z-1}}= \frac{1}{z} B\,\underline{\uvec{c}^z}, \qquad z=1,\dots, S-1,
		\end{equation}
		where $\underline{\uvec{c}^{\alpha}}$ denotes the vector $\uvec{c}$ with each entry to the power of $\alpha$.
		
		Let us compute the general $\ell$-th component of both sides of the last equality.
		The right-hand side is
		\begin{align}
			\begin{split}
				\frac{1}{z} \left(B\,\underline{\uvec{c}^z}\right)_\ell&=\frac{1}{z}\sum_{m=0}^M\Bigg[ \widehat{\psi}^\ell(1)\widehat{\psi}^m(1)- \int_{0}^{1} \left(\frac{d}{d\xi}\widehat{\psi}^\ell(\xi)\right)\widehat{\psi}^m(\xi)d\xi \Bigg] \left( \xi^m \right)^z\\
				&=\frac{1}{z}\left\lbrace\widehat{\psi}^\ell(1)\Bigg[ \sum_{m=0}^M \widehat{\psi}^m(1)\left( \xi^m \right)^z \Bigg] - \int_{0}^{1} \left(\frac{d}{d\xi}\widehat{\psi}^\ell(\xi)\right)\Bigg[ \sum_{m=0}^M \widehat{\psi}^m(\xi) \left( \xi^m \right)^z \Bigg] d\xi\right\rbrace. 
			\end{split}
		\end{align}
		Since $z$ is at most $S-1$ and the Lagrange basis functions $\widehat{\psi}^m$ associated to $M+1=S$ points allow to exactly interpolate polynomials up to degree $M=S-1$, we have that the terms in square brackets are nothing but the exact interpolation of the function $\xi^z$.  Therefore, due to the exactness of the quadrature formula for polynomials of degree $2S-3$, we can integrate by parts and obtain
		\begin{align}
			\begin{split}
				\frac{1}{z} \left(B\,\underline{\uvec{c}^z}\right)_\ell&=\frac{1}{z}\left\lbrace   \widehat{\psi}^\ell(1)\cdot 1^z- \int_{0}^{1}  \left(  \frac{d}{d\xi}\widehat{\psi}^\ell(\xi) \right)\xi^z d\xi\right\rbrace\\
				&=\frac{1}{z}\left\lbrace  \widehat{\psi}^\ell(1)\cdot 1^z - \widehat{\psi}^\ell(1)\cdot 1^z+\widehat{\psi}^\ell(0)\cdot 0^z + z \int_{0}^{1}   \widehat{\psi}^\ell(\xi)\xi^{z-1}d\xi\right\rbrace=\int_{0}^{1}  \widehat{\psi}^\ell(\xi) \xi^{z-1}d\xi.\label{eq:Bcz}
			\end{split}
		\end{align}
		The left-hand side is slightly more delicate to handle, as the adopted quadrature formula does not allow to compute exactly the terms of the matrix $\Lambda$. We have
		\begin{align}
			\left(\Lambda \, \underline{\uvec{c}^{z-1}}\right)_{\ell}&=\sum_{m=0}^M\mathcal{Q}\left\lbrace  \widehat{\psi}^\ell(\xi)\widehat{\psi}^m(\xi) \right\rbrace\left( \xi^m \right)^{z-1}  =\mathcal{Q}\left\lbrace  \widehat{\psi}^\ell(\xi)\left[ \sum_{m=0}^M\widehat{\psi}^m(\xi)\left( \xi^m \right)^{z-1} \right]\right\rbrace \label{eq:quad_linear}\\
			&=\mathcal{Q}\left\lbrace \widehat{\psi}^\ell(\xi)\xi^{z-1}\right\rbrace=\int_{0}^{1}   \widehat{\psi}^\ell(\xi)\xi^{z-1}d\xi=\frac{1}{z} \left(B\,\underline{\uvec{c}^z}\right)_\ell, \label{eq:use_acc_quad}
		\end{align}
		where in \eqref{eq:quad_linear} we have used the linearity of the quadrature operator and in  \eqref{eq:use_acc_quad} the fact that the subtimenodes exactly interpolate polynomials up to degree $M=S-1$ and \eqref{eq:Bcz}.
		
		Note that the condition $\Cp(S-1)$ is sharp. In fact, since the interpolation is not anymore exact for $z=S$, then $\Cp(S)$ is not satisfied.
	\end{subequations}
\end{proof}
The condition of the previous lemma is therefore satisfied by \ADERRK~methods with GLB and GLG subtimenodes if the same subtimenodes are adopted as quadrature points, since they induce quadrature formulas respectively characterized by degree of exactness equal to $2S-3$ and $2S-1$. 
We refer to such schemes, characterized by having the subtimenodes as quadrature points, as \ADERRK-GLB and \ADERRK-GLG.
They satisfy another important property stated in the next lemma.

\begin{lemma}\label{lem:conditionsRKD_ADER}
	The \ADERRK-GLB and \ADERRK-GLG methods satisfy $\Dp(S-1)$.
\end{lemma}
\begin{proof}
	\begin{subequations}
		Let us observe that, for both \ADERRK-GLB and \ADERRK-GLG methods, the subtimenodes and the quadrature points coincide. Hence, in these particular cases, we have a diagonal matrix $\Lambda$. In fact, its general entry is
		\begin{equation}
			\Lambda_{\ell,m}:=\mathcal{Q}\left\lbrace \widehat{\psi}^\ell (\xi) \widehat{\psi}^m (\xi) \right\rbrace= w_\ell \delta_{\ell,m}, 
		\end{equation}
		where $w_\ell=\int_0^1 \widehat{\psi}^\ell d\xi=b_\ell$ is the quadrature weight associated to $\xi^\ell$. This fact will be useful in the following.

		We write explicitly condition $\Dp(S-1)$ in matricial form, for all $z=1,\dots, S-1$
		\begin{align}
			A^T \underline{\uvec{b}\uvec{c}^{z-1}}= \frac{1}{z} \underline{\uvec{b}(\uvec{1}-\uvec{c}^z)} \Longleftrightarrow \left[\underline{\uvec{b}\uvec{c}^{z-1}}\right]^T A= \frac{1}{z} \left[\underline{\uvec{b}(\uvec{1}-\uvec{c}^z)}\right]^T \Longleftrightarrow\, 
			\left[\underline{\uvec{b}\uvec{c}^{z-1}}\right]^T = \frac{1}{z} \left[\underline{\uvec{b}(\uvec{1}-\uvec{c}^z)}\right]^T \Lambda^{-1}B,
		\end{align}
		where the general $i$-th entries of the vectors	$\underline{\uvec{b}\uvec{c}^{z-1}}$ and $\underline{\uvec{b}(\uvec{1}-\uvec{c}^z)}$ are respectively $b_ic_i^{z-1}$ and $b_i(1-c_i^z)$.
		
		Recalling that $\Lambda_{i,j}=w_i \delta_{ij}=b_i\delta_{ij}$, the following expression holds $\left[\underline{\uvec{b}(\uvec{1}-\uvec{c}^z)}\right]^T\Lambda^{-1}=\underline{(\uvec{1}-\uvec{c}^z)}^T$, with general $i$-th entry equal to $1-c_i^z$. 
		Hence, it is left to prove that $\underline{\uvec{b}\uvec{c}^{z-1}} = \frac{1}{z}B^T \underline{\uvec{1}-\uvec{c}^z}.$
		
		Expanding the $\ell$-th component of the right-hand side, we get
		\begin{align}
			\begin{split}
				\frac{1}{z}\left( B^T \underline{\uvec{1}-\uvec{c}^z}\right)_\ell&=\frac{1}{z}\sum_{m=0}^M\Bigg[ \widehat{\psi}^m(1)\widehat{\psi}^\ell(1)- \int_{0}^{1} \left(\frac{d}{d\xi}\widehat{\psi}^m(\xi)\right)\widehat{\psi}^\ell(\xi)d\xi \Bigg]\left[1- \left( \xi^m \right)^z\right]\\
				&=\frac{1}{z}\left\lbrace \left(\sum_{m=0}^M\widehat{\psi}^m(1)\left[1- \left( \xi^m \right)^z\right]\right)\widehat{\psi}^\ell(1)- \int_{0}^{1} \left(\sum_{m=0}^M\frac{d}{d\xi}\widehat{\psi}^m(\xi)\left[1- \left( \xi^m \right)^z\right]\right)\widehat{\psi}^\ell(\xi) d\xi \right\rbrace.
			\end{split}
		\end{align}
		Since $(1-\xi^z)$ is a polynomial of degree at most $S-1$, its Lagrange interpolation is exact, hence
		\begin{align}
			\begin{split}
				\frac{1}{z}\left( B^T \underline{\uvec{1}-\uvec{c}^z}\right)_\ell&=\frac{1}{z}\left\lbrace \left(1-  1^z\right)\cdot \widehat{\psi}^\ell(1) - \int_{0}^{1} \left[\frac{d}{d\xi}\left(1- \xi ^z\right)\right]\widehat{\psi}^\ell(\xi) d\xi \right\rbrace\\
				&=\int_0^1  \widehat{\psi}^\ell(\xi)\xi^{z-1} d \xi=w_\ell \cdot \left(\xi^\ell \right)^{z-1}=\left(\underline{\uvec{b}\uvec{c}^{z-1}}\right)_\ell.
			\end{split}
		\end{align}
		In the last step, we have exploited the exactness of the quadrature rule  in the GLB or GLG subtimenodes for the considered integral.
		Hence, condition $\Dp(S-1)$ holds. Note that also in this case the condition is sharp: the exact interpolation of polynomials of degree $S-1$, guaranteed for $z\leq S-1$, was necessary and therefore $\Dp(S)$ is not satisfied.
	\end{subequations}
\end{proof}

\begin{remark}[\ADERRK-GLG is not a collocation method]
	From the proof of Lemma \ref{lem:conditionsRKC_ADER}, we can observe that \ADERRK-GLG methods do not satisfy $\Cp(S)$, despite having all the $c_i$ coefficients distinct, hence, the methods are not collocation methods and they do not coincide with Gauss methods  \cite{hairer1987solving}.
\end{remark}
To finally obtain the accuracy of \ADERRK-GLB and \ADERRK-GLG, we use Theorem \ref{th:butcher_order}. 
\begin{theorem}\label{th:order_GLB}
	\ADERRK-GLB is of order $2S-2$.
\end{theorem}
\begin{proof}
	The condition $\Bp(2S-2)$ holds, because the vectors $\uvec{c}$ and $\uvec{b}$ are the quadrature points and weights of the GLB quadrature formula characterized by degree of exactness $2S-3$.
	Lemmas \ref{lem:conditionsRKC_ADER} and \ref{lem:conditionsRKD_ADER} prove that \ADERRK-GLB satisfies $\Cp(S-1)$ and $\Dp(S-1)$, so Theorem \ref{th:butcher_order} is satisfied for order $p=2S-2$ and $\eta = \zeta=S-1$. Moreover, since the quadrature is of order exactly $2S-2$, $\Bp(2S-1)$ does not hold and, hence, the method is not of order $2S-1$. 
	This observation is based on the theory of order conditions of RK methods involving trees presented in \cite{hairer1987solving}. If $\Bp(p)$ is not satisfied, not all the order conditions of \cite[Theorem 2.13]{hairer1987solving} hold for all trees of order $\leq p$ and the method cannot be of order $p$.
\end{proof}

In particular, the \ADERRK-GLB methods coincide with Lobatto IIIC RK methods \cite{wanner1996solving} and we prove it in \ref{app:GLB_Lobatto_IIIC}.

\begin{theorem}\label{th:order_GLG}
	\ADERRK-GLG is of order $2S-1$.
\end{theorem}
\begin{proof}
	\ADERRK-GLG satisfies $\Bp(2S)$, because the vectors $\uvec{c}$ and $\uvec{b}$ are the quadrature points and weights of the GLG quadrature formula characterized by degree of exactness $2S-1$. Hence, also $\Bp(2S-1)$ holds. For Lemmas \ref{lem:conditionsRKC_ADER} and \ref{lem:conditionsRKD_ADER}, it also satisfies $\Cp(S-1)$ and $\Dp(S-1)$. Hence, Theorem \ref{th:butcher_order} guarantees that the method is of order $2S-1$, since it is satisfied with $p=2S-1$ and $\eta=\zeta=S-1$.
	Also here, it is possible to prove that \ADERRK-GLG is not of order $2S$, but sharply of order $2S-1$. Indeed, a method of order $2S$ must verify $\Cp(S)$, see \cite[Theorem 342C]{butcher2016numerical}. So, by contradiction, it must be that \ADERRK-GLG is not of order $2S$.
\end{proof}

\subsection{Beyond (and within) nodal bases}\label{app:beyond_nodal}
In this subsection, we prove an interesting result concerning the equivalence between ADER schemes with arbitrary bases and ADER schemes with Lagrangian basis functions defined in the quadrature points of the former schemes. 
Let us start by introducing the ADER formulation for an arbitrary basis $\left\lbrace\phi^{m}\right\rbrace_{m=0,\dots,M}$ of the space of polynomials with degree $M$ over $[t_n,t_{n+1}]$. Again, moving from the weak formulation of the ODEs system \eqref{eq:ODE} and projecting it onto a finite dimensional functional space, we get the nonlinear system
\begin{align}
	\begin{split}
		\sum_{m=0}^M \Bigg[ \phi^{\ell}(t_{n+1})\phi^m(t_{n+1}) & -\int_{t_n}^{t_{n+1}} \left(\frac{d}{dt}\phi^\ell(t)\right)\phi^m(t)dt   \Bigg]\uvec{\co}^{m}-\phi^\ell(t_n)\uvec{u}_n \\
		&-  \int_{t_n}^{t_{n+1}} \phi^\ell(t)\uvec{G}_h(t) dt =\uvec{0}, \quad \ell=0,\dots,M,
	\end{split}
	\label{eq:weakproblemdiscrete_arbitrary}
\end{align}
where $\uvec{G}_h(t):=\uvec{G}(t,\uvec{u}_h(t))$ with the reconstruction $\uvec{u}_h(t):=\sum_{m=0}^M \uvec{\co}^m \phi^m(t)$. The unknown coefficients $\uvec{\co}^m$ are general representation coefficients of the numerical solution in the polynomial space spanned by the basis functions, not anymore nodal values. In literature, one can find different choices of basis functions alternative to nodal ones, for example modal Taylor basis functions \cite{Busto2020,boscheri2019high,ArepoTN,DOOM}.

In order to get a fully discrete version, we need to specify the quadrature formula that we are going to use in the integral terms: $\lbrace(\xi^q,w_q)\rbrace_{q=0,\dots,M}$ with $\xi^q \in [0,1]$ nodes and $w_q = \int_0^1 \widehat\psi^q(\xi)d \xi$ weights of the quadrature, where $\widehat\psi^q$ are the Lagrangian basis functions associated to the quadrature points $\xi^q$. 
Then, \eqref{eq:weakproblemdiscrete_arbitrary} reads
\begin{align}
	\begin{split}
		\sum_{m=0}^M \Bigg[ \widehat\phi^{\ell}(1)\widehat\phi^m(1) & -\sum_{q=0}^M \left(\frac{d}{d\xi} \widehat\phi^\ell(\xi^q)\right)\widehat\phi^m(\xi^q) w_q   \Bigg]\uvec{\co}^{m}-\widehat\phi^\ell(0)\uvec{u}_n \\
		&-  \dt \sum_{q=0}^M \widehat\phi^\ell(\xi^q)\uvec{G}\left(\xi^q, \sum_{m=0}^M \uvec{\co}^m \widehat\phi^m(\xi^q)\right) w_q =\uvec{0}, \quad \ell=0,\dots,M,
	\end{split}
	\label{eq:weakproblem_fully_discrete_arbitrary}
\end{align}
which is the \ADERIWF~of the method defined by the basis $\left\lbrace\phi^{m}\right\rbrace_{m=0,\dots,M}$ and the quadrature $\lbrace(\xi^q,w_q)\rbrace_{q=0,\dots,M}$. 
The system can be solved iteratively as in \eqref{eq:ADER_Picard} to obtain $\uvec{u}_{n+1}:=\uvec{u}_h(t_{n+1})$. Again, we denote with $\widehat{\left(\cdot\right)}$ the quantities rescaled onto $[0,1]$.

\begin{theorem}[Link between ADER schemes with arbitrary and nodal bases]\label{th:link}
	Consider a basis $\left\lbrace\phi^{m}\right\rbrace_{m=0,\dots,M}$ of 
	the space of the polynomials of degree $M$ over $[t_n,t_{n+1}]$ and let us denote by $\lbrace(\xi^q,w_q)\rbrace_{q=0,\dots,M}$  the $(M+1)$-points GLB or GLG quadrature rule, which we will denote by GL*, and by $\left\lbrace\psi^{m}\right\rbrace_{m=0,\dots,M}$ the respective Lagrange polynomials. Then, the \ADERIWF~\eqref{eq:weakproblem_fully_discrete_arbitrary} defined by $\left\lbrace\phi^{m}\right\rbrace_{m=0,\dots,M}$ and by the quadrature $\lbrace(\xi^q,w_q)\rbrace_{q=0,\dots,M}$  is equivalent to the \ADERIWF~\eqref{eq:weakproblemdiscrete} defined by $\left\lbrace\psi^{m}\right\rbrace_{m=0,\dots,M}$ with integrals computed through the same quadrature. 
\end{theorem}
\begin{proof}
	\begin{subequations}
		Let us observe that, being $\lbrace\hphi^{m}\rbrace_{m=0,\dots,M}$ a basis of the space of the polynomials of degree $M$ over $[0,1]$, there exists a unique vector of coefficients $\underline{\gamma}=(\gamma_0,\dots,\gamma_M)^T$ such that 
		\begin{equation}
			\sum_{m=0}^M \gamma_m \hphi^{m} \equiv 1.
			\label{eq:gamma}
		\end{equation}
		
		We aim at rewriting \eqref{eq:weakproblem_fully_discrete_arbitrary} into a matrix formulation, as in \eqref{eq:ADER_system_final}, and at comparing the resulting systems. We first introduce 
		$\uvec{\va}^q=\uvec{u}_h(t_n+\Delta t\xi^q)=\sum_{m=0}^M  \uvec{\co}^m \widehat\phi^m(\xi^q)$, the reconstructed solution value in the GL* quadrature point $\xi^q$, so that $\uvec{u}_h(t)=\sum_{m=0}^M\uvec{\co}^m\phi^m(t)=\sum_{q=0}^M\uvec{\va}^q\psi^q(t)$. 
		We also define the change of basis matrix $\IM$ such that $\underline{\uvec{\va}}=\IM\underline{\uvec{\co}}$, where the general element of such matrix is defined as $\IM_{\ell,m}=\hphi^m(\xi^\ell)$.

		Hence, system \eqref{eq:weakproblem_fully_discrete_arbitrary} can be equivalently recast as
		\begin{equation}
			B \undu - \undr - \Delta t \Lambda  \undG(\undv) =\uvec{0},
			\label{eq:ADER_system_arbitrary}
		\end{equation}
		where 
		\begin{align}
			\begin{split}
				&B_{\ell,m}:=  \hphi^\ell(1)\hphi^m(1)- \sum_{q=0}^M \left(\frac{d}{d\xi}\hphi^\ell(\xi^q)\right)\hphi^m(\xi^q)w_q , \qquad \Lambda_{\ell,m}:=\hphi^\ell(\xi^m) w_m,  \\ 
				&\undu:=\begin{pmatrix}
					\uvec{\co}^0\\
					\vdots\\
					\uvec{\co}^M
				\end{pmatrix},\quad
				\undr:=\begin{pmatrix}
					\phi^0(t_n)\uvec{u}_n\\
					\vdots\\
					\phi^M(t_n)\uvec{u}_n
				\end{pmatrix},\quad
				\undG(\undv):=\begin{pmatrix}
					\uvec{G}(t^0,\uvec{\va}^0)\\
					\vdots\\
					\uvec{G}(t^M,\uvec{\va}^M)
				\end{pmatrix}, \quad \underline{\uvec{\va}}=\IM\underline{\uvec{\co}}.
			\end{split}
			\label{eq:ADER_structure_arbitrary}
		\end{align}

		Again, for compactness, the structures are referred to a scalar problem and the matrices should be block expanded for a vectorial one. 
		We can rewrite $\Lambda$ as $\Lambda=\IM^T W$ with $W$ being a diagonal matrix having as entries the GL* quadrature weights $w_m$.

		Then, inverting $B$ and multiplying by $\IM$, we have that system \eqref{eq:ADER_system_arbitrary} is equivalent to
		\begin{equation}
			\undv-\IM B^{-1}\undr-\Delta t \IM B^{-1}\IM^T W  \underline{\uvec{G}}(\underline{\uvec{v}})=\uvec{0}.
			\label{eq:ADER_system_final2}
		\end{equation}

		Now, we can compare such system with the ADER system \eqref{eq:ADER_system_final} obtained for GL* subtimenodes and quadrature.
		In particular, the two formulations coincide if and only if
		\begin{equation}\label{eq:equivalence_equi_GLG}
			\begin{cases}
				\IM B^{-1}\undr = \undu_n  \\
				\IM B^{-1} \IM^T W= B_{GL*}^{-1}\Lambda_{GL*},
			\end{cases}
		\end{equation}

		where $B_{GL*}$ and $\Lambda_{GL*}$ are the ADER structures given in \eqref{eq:ADER_structure} with GL* subtimenodes, also assumed as quadrature points for the computation of the integrals.
		
		Let us start by the first equivalence of \eqref{eq:equivalence_equi_GLG}. Let us simplify the notation of $\undr$ writing it as  
		$
		\undr=\uhphi(0)\uvec{u}_n$, 
		where $\uvec{u}_n$ is meant to be multiplied to each component of $\uhphi(0)$. With this, we just need to prove that
		\begin{equation}
			\IM B^{-1} \uhphi(0)=\uvec{1}.
		\end{equation}
		Recalling that $\sum_{m=0}^M \gamma_m\hphi^m(\xi) \equiv 1$, we can instead prove that 
		\begin{equation}
			\uhphi(0)=B\underline{\gamma}, 
			\label{eq:intermidiatex}
		\end{equation}
		as we would have $\left(\IM B^{-1} \uhphi(0)\right)_{\ell}=\left(\IM \underline{\gamma}\right)_{\ell}=\sum_{m=0}^M \hphi^m(\xi^\ell) \gamma_m=1$.
		Equality \eqref{eq:intermidiatex} can be proven expanding the right hand side:
		\begin{align}
			\begin{split}
				\left(B\underline{\gamma}\right)_\ell&=\sum_{m=0}^M \left[ \hphi^\ell(1)\hphi^m(1)- \int_{0}^{1} \left(\frac{d}{d\xi}\hphi^\ell(\xi)\right)\hphi^m(\xi)d\xi \right]\gamma_m\\
				&=\hphi^\ell(1)\left(\sum_{m=0}^M \hphi^m(1)\gamma_m\right) - \int_{0}^{1}\left( \frac{d}{d\xi}\hphi^\ell(\xi)\right) \left(\sum_{m=0}^M\hphi^m(\xi)\gamma_m \right)d\xi \\
				&= \hphi^\ell(1)\cdot 1- \int_{0}^{1} \left(\frac{d}{d\xi}\hphi^\ell(\xi)\right) \cdot 1 d\xi=\hphi^\ell(1)-\hphi^\ell(1)+\hphi^\ell(0)=\hphi^\ell(0),
			\end{split}
		\end{align}
		thanks again to \eqref{eq:gamma} 
		and the exactness of the quadrature formula for polynomials of degree $M-1$.\\
		For the second equality of \eqref{eq:equivalence_equi_GLG}, let us first notice that, due to the assumption on the quadrature, we have $\Lambda_{GL*}=W$.  
		Thus, we suffice to prove
		\begin{equation}\label{eq:equivalence_B_matrices}
			\IM B^{-1} \IM^T = B_{GL*}^{-1}.
		\end{equation}
		Taking the inverse of the previous equality and inverting the $\IM$ matrices, we get
		\begin{align*}
			\IM B^{-1} \IM^T = B_{GL*}^{-1}  \Longleftrightarrow 
			(\IM^T)^{-1} B(\IM)^{-1}  = B_{GL*}  \Longleftrightarrow 
			B=\IM^T B_{GL*}\IM.
		\end{align*}
		Let us compute the general entry of the matrix at the right-hand side
		\begin{align}
			\begin{split}
				\left(\IM^T B_{GL*}\IM\right)_{\ell,m}&=\sum_{i=0}^M  \left(\IM^T\right)_{\ell,i} \sum_{j=0}^M \left(B_{GL*}\right)_{i,j}  \IM_{j,m}\\
				&=\sum_{i=0}^M  \hphi^\ell(\xi^i) \sum_{j=0}^M \left[  \widehat{\psi}^i(1)\widehat{\psi}^j(1)- \int_{0}^{1} \left(\frac{d}{d\xi}\widehat{\psi}^i(\xi)\right) \widehat{\psi}^j(\xi) d\xi \right]  \hphi^m(\xi^j)\\
				&=\sum_{i=0}^M  \hphi^\ell(\xi^i)\widehat{\psi}^i(1) \sum_{j=0}^M   \widehat{\psi}^j(1)\hphi^m(\xi^j) - \int_{0}^{1} \frac{d}{d\xi} \left(\sum_{i=0}^M\hphi^\ell(\xi^i) \widehat{\psi}^i(\xi)\right)\sum_{j=0}^M \widehat{\psi}^j(\xi)\hphi^m(\xi^j) d\xi.
			\end{split}
		\end{align}
		
		Now, observing that $\sum_{i=0}^M  \hphi^\ell(\xi^i)\widehat{\psi}^i$ and $\sum_{j=0}^M  \hphi^m(\xi^j)\widehat{\psi}^j$ are the exact interpolations of $\hphi^\ell$ and $\hphi^m$ respectively, we finally have
		\begin{align}\label{eq:equivalence_B_matrices2}
			\begin{split}
				\left(\IM^T B_{GL*}\IM\right)_{\ell,m}&=  \hphi^\ell(1) \hphi^m(1) - \int_{0}^{1} \left(\frac{d}{d\xi} \hphi^\ell(\xi)\right) \hphi^m(\xi) d\xi=B_{\ell,m}.
			\end{split}
		\end{align}
		Clearly, since the nodal values define a unique polynomial, also the final interpolation step will be the same for the two ADER methods involved in this theorem. 
	\end{subequations}
\end{proof}
The previous result is interesting under several points of view. 
Thanks to Theorem \ref{th:link}, we can give a precise characterization of the ADER methods for arbitrary bases. 
If, for some reason, a user wants to adopt another basis other than the GL* polynomials for the definition of the ADER method, for example a Taylor basis, then the properties of the resulting scheme do not change, provided that the function $\uvec{G}$ is directly evaluated in the GL* quadrature points. Therefore, also the accuracy of the resulting scheme is known and it is characterized by the one of the \ADERRK-GL* methods in the previous section.

Notice that Theorem~\ref{th:link} applies also in the context of nodal bases: if one chooses $M+1$ equispaced subtimenodes for the definition of the ADER method but $\uvec{G}$ is evaluated in the GL* quadrature points, instead of the basis nodal values, then the resulting accuracy will not be $M+1$, but rather $2M$ and $2M+1$ for GLB and GLG, respectively.

\section{ADER as a Deferred Correction method}
\label{sec:DeC}

This section aims to determine the optimal number of iterations for an ADER method using the DeC formalism, an abstract framework used to approximate arbitrarily well the solution of implicit (nonlinear) discretizations of analytical problems, through an easy iterative procedure. 
Originally presented in 1949 \cite{fox1949some}, the DeC was applied in different flavors to ODEs \cite{minion2003semi,layton2005implications,huang2006accelerating,minion2011hybrid,mPDeC,loredavide,han2021dec} and PDEs \cite{minion2004semi,speck2015multi,Decremi,ciallella2022arbitrary,DOOM,abgrall2020high,bacigaluppi2019posteriori,abgrall2019high} contexts.
Abgrall \cite{Decremi} proposed a formalization using two operators $\lopd^1,\lopd^2: X \rightarrow Y$, depending on a same parameter $\Delta$, corresponding to two different discretizations of the same problem: $\lopd^2$ is a \textit{difficult-to-solve}  high order nonlinear implicit discretization of the problem and $\lopd^1$ is an \textit{easy-to-solve}  low order (explicit) one.
Aiming at $\undu_\Delta\in X$, solution of $\lopd^2(\undu_\Delta)=\uvec{0}_Y$, we iteratively approximate it arbitrarily well through an \textit{easy-to-solve} (explicit) iteration process as prescribed in the next theorem.
\begin{theorem}[DeC; Abgrall \cite{Decremi}]\label{th:DeC}
	For a fixed $\underline{\uvec{u}}^{(0)}\in X$, let us define the sequence of vectors $\underline{\uvec{u}}^{(p)}$ as the solution of
	\begin{equation}
		\label{eq:DeC_iteration}
		\lopd^1(\underline{\uvec{u}}^{(p)}):=\lopd^1(\underline{\uvec{u}}^{(p-1)})-\lopd^2(\underline{\uvec{u}}^{(p-1)}), \quad p\geq 1.
	\end{equation}
	If the following conditions on the operators $\lopd^1$ and $\lopd^2$ hold
	\begin{enumerate}
		\item $\exists ! \,\usol \in X$ solution of $\lopd^2$ such that $\lopd^2(\usol)=\uvec{0}_Y$ (Existence of a unique solution to $\lopd^2$);
		\item $\exists \,\alpha_1 \geq 0$ independent of $\Delta$ such that \begin{equation}
			\norm{\lopd^1(\underline{\uvec{v}})-\lopd^1(\underline{\uvec{w}})}_Y\geq \alpha_1\norm{\underline{\uvec{v}}-\underline{\uvec{w}}}_X, ~ \forall \underline{\uvec{v}},\underline{\uvec{w}}\in X \text{(Coercivity-like property of $\lopd^1$)};
			\label{eq:DeC_coercivity}
		\end{equation}
		\item $\exists\, \alpha_2 \geq 0$ independent of $\Delta$ such that 
		\begin{equation}
			\begin{split}
				\norm{\left[\lopd^1(\underline{\uvec{v}})\!-\!\lopd^2(\underline{\uvec{v}})\right]\!-\!\left[\lopd^1(\underline{\uvec{w}})\!-\!\lopd^2(\underline{\uvec{w}})\right]}_Y\!\leq & \alpha_2 \Delta \!\norm{\underline{\uvec{v}}-\underline{\uvec{w}}}_X ,~\forall \underline{\uvec{v}},\underline{\uvec{w}}\in X;\\
				&\text{(Lipschitz-continuity-like property of $\lopd^1-\lopd^2$)}
			\end{split}
			\label{eq:DeC_lipschitz}
		\end{equation}
	\end{enumerate}
	then, we can prove the following error estimate
	\begin{equation}
		\label{eq:DeC_accuracy}
		\norm{\underline{\uvec{u}}^{(p)}-\usol}_X \leq \left( \Delta \frac{\alpha_2}{\alpha_1} \right)^p\norm{\underline{\uvec{u}}^{(0)}-\usol}_X, \quad \forall p\in \mathbb{N}. \end{equation}
\end{theorem}
The proof can be found in \cite{Decremi,loredavide,DOOM} and it uses induction on the iterations.
Estimate \eqref{eq:DeC_accuracy} tells that $\undu_\Delta$ can be approximated with arbitrarily high precision as $p \rightarrow +\infty$. 
However, $\undu_\Delta$ is itself an approximation of the exact solution $\undu_{ex}$ of the original analytical problem to which the operators are associated.
If its order of accuracy is $R$, it suffices to find an $R$-th order accurate approximation of $\undu_\Delta$ to get the same formal order of accuracy.
By triangular inequality, if $\undu^{(0)}$ is $O(\Delta)$-accurate, the order of accuracy of $\undu^{(p)}$ with respect to $\undu_{ex}$ is $\min{(p,R)}$, hence, the optimal choice for the final number of iterations $P$ is $P=R$. Extra iterations may (slightly) improve the accuracy but they do not increase the order of accuracy and are essentially a waste of computational resources.

\subsection{Link ADER-DeC}
We will show here how the ADER method presented in Section~\ref{sec:ADER} can be put in a DeC formalism with $\Delta=\Delta t$. We start by defining the $N$-th order accurate operator $\lopd^2:\mathbb{R}^{(M+1)\times Q} \rightarrow \mathbb{R}^{(M+1)\times Q}$ as
\begin{align}
	\label{eq:ADER_l2}
	\lopdt^2(\undu):=\undu-\undu_n-\Delta t B^{-1}\Lambda  \underline{\uvec{G}}(\underline{\uvec{u}}).
\end{align}
Let us notice that the problem of finding $\usol$ such that $\lopdt^2(\usol)=\uvec{0}$ is indeed equivalent to solving the nonlinear system \eqref{eq:ADER_system_final}. 
We introduce the low order operator $\lopd^1:\mathbb{R}^{(M+1)\times Q} \rightarrow \mathbb{R}^{(M+1)\times Q}$ as a first order explicit approximation of \eqref{eq:ADER_system_final}
\begin{align}
	\label{eq:ADER_l1}
	\lopdt^1(\undu):=\undu-\undu_n-\Delta t B^{-1}\Lambda  \underline{\uvec{G}}(\undu_n).
\end{align}
\begin{remark}[On the accuracy of the operators]
	We remark that the mentioned order of accuracy of the operators $\lopd^1$ and $\lopd^2$ is referred, in this context, to the accuracy of the approximation $\uvec{u}_{n+1}=\uapp(t_{n+1})$ that the related solution coefficients $\undu$ induce via \eqref{eq:uh}.
\end{remark}
We can easily prove that the operators that we have defined satisfy the assumptions of Theorem \ref{th:DeC}.
\begin{theorem}[Properties of $\lopd^1$ and $\lopd^2$; Han Veiga, \"Offner, Torlo \cite{han2021dec}]
	The operators $\lopd^1$ and $\lopd^2$, given by \eqref{eq:ADER_l1} and \eqref{eq:ADER_l2}, satisfy the hypotheses of Theorem \ref{th:DeC}.
\end{theorem}
The proof can be found in \cite[Propositions 4.3 and 4.4]{han2021dec}.
Finally, the resulting DeC iteration \eqref{eq:DeC_iteration}, in this particular case, after an easy direct computation, reads
\begin{equation}
	\undu^{(p)}=\undu_n+\Delta t B^{-1}\Lambda  \underline{\uvec{G}}(\undu^{(p-1)}),
	\label{eq:ADER_DeC_iteration}
\end{equation}
which coincides with the iteration defined in \eqref{eq:ADER_Picard}. However, the DeC formalism allows us to select the optimal number of iterations, which is equal to $N$, i.e., the order of $\lopd^2=0$, provided that $\undu^{(0)}$ is $O(\Delta t)$-accurate. Therefore, we can set $\undu^{(0)}:=\undu_n$ and $P=N$ to get the formal order of accuracy.

\section{Novel modified ADER methods}
\label{sec:ADERNEW}

In this section, we describe some efficient versions of ADER methods
based on a modification of the original approach,
following the strategy proposed in \cite{loredavide}, in a DeC context for ODEs, and then generalized and applied to an ADER-DG framework in \cite{DOOM}.
Actually, the basic idea was firstly introduced by Minion in \cite{minion2003semi} as ladder DeC methods for ODEs, counting several follow-ups \cite{speck2015multi,hamon2019multi,franco2018multigrid,minion2015interweaving,benedusi2021experimental}. However, the approach was very specific for DeC time-integration methods for ODEs.
The more general formulation presented in \cite{DOOM} allows for applications to many other contexts, e.g., to ADER schemes.
In particular, the modification consists in redesigning the whole iterative process in such a way that the discretization accuracy increases accordingly to the order of accuracy of the numerical solution at each specific iteration. In practice, we will look for a solution in a different approximation space at each iteration $p$, i.e., the time reconstruction is such that $\uapp^{(p)}(t) \in \mathbb P_p$ up to a maximum reconstruction degree, where $\mathbb P_p$ is the space of polynomials of degree $p$ over $[t_n,t_{n+1}].$ 
This is beneficial under many points of view, among which: 
\begin{itemize}
	\item we save computations in the early iterations, as we work with smaller vectors and matrices;
	\item $p$-adaptivity can be naturally embedded in the new methods, as there is no formal upper bound on the order of accuracy. 
\end{itemize}
In order to change iteration structures along the iterative process, some embeddings between different spaces $\mathcal{E}^{(p-1)}:\mathbb{P}^{p-1} \hookrightarrow \mathbb P^{p}$, e.g., interpolations or $L^2$-projections, are needed to pass from $\undu^{(p-1)}$ to some $\undu^{*(p-1)}$, which is a suitable input for the $p$-th iteration to compute $\undu^{(p)}$.  
For the specific context of DeC methods, the following theorem holds.
\begin{theorem}[Micalizzi, Torlo, Boscheri \cite{DOOM}]\label{th:NEWDEC}
	Let us consider a problem with exact solution $\uex \in Z$ and the normed vector spaces $(\Xp,\norm{\cdot}_{\Xp})$ and $(\Yp,\norm{\cdot}_{\Yp})$  for $p\in \mathbb{N}$,  $p\geq 1$. 
	Further, let us assume that some operators $\lopdt^{1,(p)},\lopdt^{2,(p)}:\Xp \rightarrow \Yp$ are defined for $p \geq 1 $, dependent on the same parameter $\Delta$ and satisfying the assumptions in Theorem \ref{th:DeC} for $\alpha_{1}^{(p)}, \alpha_{2}^{(p)}>0$ and $\usolp\in \Xp$. 
	Let us also assume that $ \forall p\in \mathbb{N}$ there exists an embedding operator $\embep:\Xp\rightarrow X^{(p+1)}$ and a projection $\projp:Z\rightarrow \Xp$.
	We define $\undu^{*(p)}:=\embep(\undu^{(p)}) \in X^{(p+1)}$ and $\uexp:=\projp(\uex) \in \Xp$.
	Let us consider the modified DeC method whose general $p$-th iteration is given by
	\begin{equation}
		\begin{cases}
			\undu^{*(p-1)}:=\mathcal{E}^{(p-1)}(\undu^{(p-1)}),\\
			\lopd^{1,(p)}(\underline{\uvec{u}}^{(p)}):=\lopd^{1,(p)}(\underline{\uvec{u}}^{*(p-1)})-\lopd^{2,(p)}(\underline{\uvec{u}}^{*(p-1)}),
		\end{cases}
		\label{eq:NEWDEC_p_iteration}
	\end{equation}
	for some $\underline{\uvec{u}}^{(0)}\in X^{(0)}$.
	Moreover, assume that $\norm{\usolp-\uexp}_{\Xp}=O(\Delta^{p+1}), $ for $p \geq 1$, that there exists $C>0$ independent of $\Delta$ such that $\norm{\undu^{*(p)}-\undu_{ex}^{(p+1)}}_{X^{(p+1)}}\leq C\norm{\undu^{(p)}-\uexp}_{\Xp},$  for $p\geq 0$ and that $\norm{\undu^{(0)}-\undu_{ex}^{(0)}}_{X^{(0)}}=O(\Delta)$.
	Then, the following error estimate holds
	\begin{equation}
		\norm{\undu^{(p)}-\uexp}_{\Xp}=O(\Delta^{p+1}), \quad \forall p\in \mathbb{N}.
		\label{eq:NEWDEC_accuracy}
	\end{equation}
\end{theorem}

The proof can be found in \cite{DOOM}.
A fundamental difference between the framework of Theorem \ref{th:DeC} and the one of Theorem  \ref{th:NEWDEC} is the fact that the former deals with converge towards the solution of a fixed operator $\lopd^2$, see \eqref{eq:DeC_accuracy}, while in the latter the error estimate \eqref{eq:NEWDEC_accuracy} is referred, at each iteration, to a new and more accurate projection of the exact solution $\undu_{ex}$.

In the previous section, we have proved how the ADER methods can be put in a DeC framework. Now, we will see how the presented modification can be applied in this specific case. 
In particular, we propose three modifications, which differ in the way the embeddings between the different iterations are achieved, even though we will later prove that two of them are actually equivalent in the considered framework. 

We assume equispaced subtimenodes at the beginning and we generalize for other types of subtimenodes later on.

We introduce here, for each $p$, the space $X^{(p)}:= \mathbb{R}^{(M^{(p)}+1)\times Q}$ of the representation coefficients of vectorial polynomial functions in $(\mathbb P_{M^{(p)}})^Q$, where
the used basis functions $\psi^{m,(p)}$ for  $m=0,\dots,M^{(p)}$ are the Lagrange polynomials of degree $M^{(p)}$ associated to $M^{(p)}+1$ equispaced subtimenodes in the interval $[t_{n},t_{n+1}]$, collected in the vector $\underline{t}^{(p)}:=\left(t^{0,(p)},\dots,t^{{M^{(p)},(p)}} \right)^T$, with $M^{(p)}=p$ for all $p\neq 0$ and $M^{(0)}=1$.
It is also useful to introduce here some structures associated to such basis functions and in particular the matrices $B^{(p)}$ and $\Lambda^{(p)}$, defined as in \eqref{eq:ADER_structure} but considering the functions $\left\lbrace \psi^{m,(p)}\right\rbrace_{m=0,\dots,M^{(p)}}$ in place of $\left\lbrace \psi^{m}\right\rbrace_{m=0,\dots,M}$.

In the following, we will explain in detail the procedure to pass from $\undu^{(p-1)}\in X^{(p-1)}$ to $\undu^{(p)}\in X^{(p)}$ with the three approaches.

\subsection{\ADERu}
In this case, the embeddings consist in interpolations of the reconstructed numerical solution $\uvec{u}(t)$ in $[t_n,t_{n+1}]$ between one iteration and the next one.

We start by $\undu^{(0)}:=(\uvec{u}_n,\uvec{u}_n)^T\in X^{(0)}=\mathbb{R}^{2\times Q}$ associated to two subtimenodes, yielding $O(\Delta t)$-accuracy, and we perform the standard update \eqref{eq:ADER_DeC_iteration} with structures associated to two subtimenodes
\begin{equation}
	\undu^{(1)}=\undu_n^{(1)}+\Delta t \left(B^{(1)} \right)^{-1}\Lambda^{(1)}  \underline{\uvec{G}}(\undu^{(0)})\quad \text{with}\quad \undu_n^{(1)}:=\begin{pmatrix}
		\undu_n\\
		\undu_n
	\end{pmatrix}\in \mathbb{R}^{2\times Q}.
	\label{eq:ADERu1}
\end{equation} 
We get $\undu^{(1)}\in X^{(1)}=\mathbb{R}^{2\times Q}$, corresponding to two subtimenodes and first order accurate, which allows to get an $O(\Delta t^2)$-accurate linear reconstruction of the numerical solution in $[t_n,t_{n+1}].$ 
Then, we perform the embedding and, by a simple interpolation of the linear reconstruction, via a suitable interpolation matrix $H^{(1)}$, we get 
\begin{equation}
	\undu^{*(1)}:=H^{(1)}\undu^{(1)}\in X^{(2)}=\mathbb{R}^{3\times Q},
\end{equation}
corresponding to three equispaced subtimenodes, still $O(\Delta t^2)$-accurate. For the next iteration, $\undu^{*(1)}$ will be the starting point to compute $\undu^{(2)}$ with structures associated to three subtimenodes. 
Iteratively,  
at the generic iteration $p>1$, we pass from $\undu^{(p-1)}\in X^{(p-1)}=\mathbb{R}^{p\times Q}$ to $\undu^{(p)}\in X^{(p)}=\mathbb{R}^{(p+1)\times Q}$ through an interpolation and a standard ADER iteration
\begin{align}
	&\begin{cases}
		\undu^{*(p-1)}:=H^{(p-1)}\undu^{(p-1)}\in X^{(p)}=\mathbb{R}^{(p+1)\times Q},\\
		\undu^{(p)}=\undu_n^{(p)}+\Delta t \left(B^{(p)} \right)^{-1}\Lambda^{(p)}  \underline{\uvec{G}}(\undu^{*(p-1)})=\undu_n^{(p)}+\Delta t \left(B^{(p)} \right)^{-1}\Lambda^{(p)}  \underline{\uvec{G}}(H^{(p-1)}\undu^{(p-1)}),
	\end{cases}
	\label{eq:ADERu}
\end{align}
with $\undu^{(p)}\in X^{(p)}=\mathbb{R}^{(p+1)\times Q}$ being $p$-th order accurate and associated to $p+1$ equispaced subtimenodes. The interpolation matrices $H^{(p-1)}$ are defined by 
$
H^{(p-1)}_{\ell,m}:=\psi^{m,(p-1)}(t^{\ell,(p)})
$
and the vector $\undu_n^{(p)}$ has $p+1$ components equal to $\uvec{u}_n$. Clearly, $
H^{(p-1)}$ must be block-expanded in the context of a vectorial problem.
\begin{remark}[On the optimal number of iterations]\label{rmk:last_iteration}
	In such context, it is worth observing that $p+1$ subtimenodes could guarantee $(p+1)$-th order of accuracy. Therefore, if the final number of subtimenodes is fixed to be $M+1$, we perform $M$ iterations, getting $\undu^{(M)} \in X^{(M)}= \R^{(M+1)\times Q}$, plus one final iteration without interpolation, obtaining $\undu^{(M+1)}\in X^{(M+1)}=X^{(M)}$, to reach the maximal accuracy associated to such subtimenodes. This holds for equispaced subtimenodes, while in Section~\ref{sec:other_nodes} we will generalize this idea to other subtimenodes types. 
\end{remark}

\subsection{\ADERdu}
Contrarily to the previous case, here the embedding is performed on the evolution operator $\uvec{G}(t,\uvec{u}(t))$, rather than on $\uvec{u}(t)$. The name of the method is given to remark that the embedding is performed on the time derivative of the variable $\uvec{u}$. 

The iteration process is very similar to before and, starting again by two subtimenodes and a first iteration without interpolation, at the general iteration $p>1$ we have
\begin{equation}
	\begin{cases}
		\underline{\uvec{G}}^{*(p-1)}:=H^{(p-1)}\underline{\uvec{G}}(\undu^{(p-1)})\in \mathbb{R}^{(p+1)\times Q},\\
		\undu^{(p)}=\undu_n^{(p)}+\Delta t \left(B^{(p)} \right)^{-1}\Lambda^{(p)}  \underline{\uvec{G}}^{*(p-1)} = \undu_n^{(p)}+\Delta t \left(B^{(p)} \right)^{-1}\Lambda^{(p)}  H^{(p-1)}\underline{\uvec{G}}(\undu^{(p-1)}),
	\end{cases}
	\label{eq:ADERdu}
\end{equation}
where $\undu^{(p)}$ is again $p$-th order accurate.

Also in this case, Remark~\ref{rmk:last_iteration} holds and, if the final number of subtimenodes is fixed to be $M+1$, then $P=M+1$ iterations are recommended without interpolation at the last iteration. In particular, the final iteration without interpolation saturates the $(M+1)$-th order of accuracy associated to $M+1$ equispaced subtimenodes.

\subsection{ADER-$L^2$}
Also for this efficient ADER, the embedding is not performed on $\uvec{u}(t)$, but on $\uvec{G}(t,\uvec{u}(t))$. The difference with \ADERdu~is that in ADER-$L^2$ the embedding consists in a Galerkin projection of the evolution operator from the polynomial space corresponding to $X^{(p-1)}$ to the polynomial space associated to $X^{(p)}$.

For the general iteration $p>1$, we directly modify the discretization of the \ADERIWF~\eqref{eq:weakproblemdiscrete} by looking for the new solution $\uvec{u}_h^{(p)}(t)$ into the polynomial space of $X^{(p)}$, with test functions belonging as well to this space, while $\uvec{G}(t,\uvec{u}(t))$ is still represented in the space $X^{(p-1)}$,
leading to
\begin{align}
	\begin{split}
		\sum_{m=0}^{M^{(p)}} \Bigg[ \psi^{\ell,(p)}(t_{n+1}) & \psi^{m,(p)}(t_{n+1}) -\int_{t_n}^{t_{n+1}} \left(\frac{d}{dt}\psi^{\ell,(p)}(t)\right) \psi^{m,(p)}(t) dt   \Bigg]\uvec{u}^{m,(p)}-\psi^{\ell,(p)}(t_n)\uvec{u}_n \\
		&- \sum_{m=0}^{M^{(p-1)}} \left( \int_{t_n}^{t_{n+1}}  \psi^{\ell,(p)}(t)\psi^{m,(p-1)}(t)dt \right) \uvec{G}(t^{m,(p-1)},\uvec{u}^{m,(p-1)}) =\uvec{0}, \quad \ell=0,\dots,M^{(p)}.
	\end{split}
	\label{eq:weakproblemdiscreteL2}
\end{align}

System \eqref{eq:weakproblemdiscreteL2} can be written in matricial form as
\begin{equation}
	B^{(p)} \undu^{(p)} - \undr^{(p)} - \Delta t \Lambda^{(p,p-1)}  \undG(\undu^{(p-1)}) =\uvec{0},
	\label{eq:ADER_systemL2}
\end{equation}
where the matrix $\Lambda^{(p,p-1)}$ and the vector $\undr^{(p)}$ are defined by
\begin{align}
	\begin{split}
		\Lambda^{(p,p-1)}_{\ell,m}:&= \frac{1}{\Delta t}\int_{t_{n}}^{t_{n+1}} \psi^{\ell,(p)}(t)\psi^{m,(p-1)}(t) dt=\int_{0}^{1} \widehat{\psi}^{\ell,(p)}(\xi)\widehat{\psi}^{m,(p-1)}(\xi)d\xi,\quad \undr^{(p)}_\ell:=\psi^{\ell,(p)}(t_n)\uvec{u}_n.
	\end{split}
	\label{eq:ADER_structureL2}
\end{align}
Inverting $B^{(p)}$ and making use of Proposition \ref{prop:r}, we get
\begin{equation}
	\undu^{(p)}:=\undu_n^{(p)}+\Delta t \left( B^{(p)} \right)^{-1}\Lambda^{(p,p-1)}  \underline{\uvec{G}}(\undu^{(p-1)}).
	\label{eq:ADERL2}
\end{equation}
In this formulation, the embedding is naturally realized by the mismatch between the spaces of $\uvec{G}(t,\uvec{u}(t))$ and of the test functions, without any need for additional structures.

Again, following Remark \ref{rmk:last_iteration}, we underline that a final extra iteration of the standard method without embedding is used for a fixed final number of subtimenodes to get the optimal accuracy.

We show the equivalence of \ADERdu~and ADER-$L^2$ in the next theorem. However, despite being equivalent in the context of ADER methods for ODEs, the two modifications have a deeply different nature and, hence, their generalization to the context of ADER methods for hyperbolic PDEs leads to different families of schemes. This is why both the approaches have been described.

\begin{prop}[Equivalence between \ADERdu~and ADER-$L^2$]
	The methods \ADERdu~\eqref{eq:ADERdu} and ADER-$L^2$ \eqref{eq:ADERL2} are equivalent.
\end{prop}
\begin{proof}
	In both cases the first iteration is the one of the standard method with two subtimenodes, therefore, let us focus on the generic iteration $p>1$.
	The updates of the two methods, respectively given in \eqref{eq:ADERdu} and \eqref{eq:ADERL2}, are clearly equivalent if $\Lambda^{(p)}  H^{(p-1)}=\Lambda^{(p,p-1)}$.
	Exploiting the definition of $\Lambda^{(p)}  H^{(p-1)}$, we get
	\begin{align}
		\begin{split}
			\left(\Lambda^{(p)}  H^{(p-1)}\right)_{\ell,m}&=\sum_{k=1}^{M^{(p)}} \int_{0}^{1} \widehat{\psi}^{\ell,(p)}(\xi)\widehat\psi^{k,(p)}(\xi) d\xi \, \widehat\psi^{m,(p-1)}(\xi^{k,(p)})\\
			&=\int_{0}^{1}\widehat\psi^{\ell,(p)}(\xi) \left( \sum_{k=1}^{M^{(p)}}  \widehat\psi^{k,(p)}(\xi)\widehat\psi^{m,(p-1)}(\xi^{k,(p)}) \right)d\xi =\int_{0}^{1}\widehat\psi^{\ell,(p)}(\xi) \widehat\psi^{m,(p-1)}(\xi) d\xi,
		\end{split}
		\label{eq:intermediateequivalence}
	\end{align}
	which is the definition of $\Lambda^{(p,p-1)}_{\ell,m}$.
	In the previous computations, we have used the fact that the term in parenthesis is the exact interpolation of $\psi^{m,(p-1)}(t) \in \mathbb P _{p-1}$ into the polynomial space $\mathbb P_p$ defined by the Lagrange basis functions $\psi^{k,(p)}$ (modulo a remapping into $[0,1]$). 
	Clearly, a final iteration of the standard method in the spirit of Remark \ref{rmk:last_iteration}, to achieve the optimal accuracy, does not spoil the equivalence.
\end{proof}

\subsection{\ADERu, \ADERdu~and ADER-$L^2$ for other choices of subtimenodes}\label{sec:other_nodes}
As already said, the order $N$ of a standard ADER method for ODEs depends on the distribution of the adopted subtimenodes. 
The modified methods for a general distribution, e.g., GLB or GLG, with a fixed number of subtimenodes equal to $M+1$, are hence constructed as follows. 
We start with two subtimenodes and we proceed like described in the equispaced case: we perform the first iteration of the standard method and we continue with iterations of the modified method (\eqref{eq:ADERu}, \eqref{eq:ADERdu} or \eqref{eq:ADERL2}) increasing the number of subtimenodes until the iteration $p=M$ corresponding to $M+1$ subtimenodes and order $M$. At this point, we continue with $N-M$ iterations of the standard method to saturate the accuracy of the adopted distribution.

On the other hand, in order to reach a specific order $P$ of accuracy in the most efficient way, we select $M$ as the minimal integer such that $M+1$ subtimenodes guarantee the order of the ADER method to be $N\geq P$. 
This can be done for example for GLB nodes with $M=\left \lceil \frac{P}{2} \right \rceil$ and $M=\max{\left(\left \lceil \frac{P-1}{2} \right \rceil,1\right) }$ for GLG. 
Then, we perform $M$ iterations to reach $M+1$ subtimenodes and, further, we perform $P-M$ final iterations of the standard method to reach the desired accuracy.

We conclude this little section with some useful considerations.
First, it is mandatory to start with at least two subtimenodes because a single node (in principle admissible for GLG) would not be enough to guarantee, after the first iteration, a first order reconstruction of the numerical solution, hence, spoiling the accuracy.
Secondly, the embeddings must be performed starting already from the first iterations and it is not possible to postpone them after the saturation of the accuracy associated to an intermediate set of subtimenodes used along the iteration process.
The reason is given by the fact that the $p$-th embedding is only $p$-th order accurate, i.e., the interpolation with $p+1$ subtimenodes is associated to a local truncation error $O(\Delta t^{p+1})$, and using it at later iterations would spoil the accuracy.

\subsection{New adaptive ADER methods}
\label{sec:ADAPT}
The novel ADER methods can be easily further modified to design $p$-adaptive schemes.

From Theorem \ref{th:NEWDEC}, we have that the iteration process yields a numerical solution $\undu^{(p)}$, which gains one order of accuracy at each iteration, with no saturation due to a fixed operator $\lopdt^2$. Under smoothness assumptions on the exact solution, we can therefore perform the iteration process, increasing the number of subtimenodes and hence the order of accuracy, until convergence up to a user defined tolerance $\varepsilon$. If we define the approximation of $\uvec{u}_{n+1}$ obtained at the $p$-th iteration as $\uvec{u}_{n+1}^{(p)}:=\sum_{m=0}^{M^{(p)}} \uvec{u}^{m,(p)}\psi^{m,(p)}(t_{n+1})$, we can set as stopping criterion
\begin{equation}
	\frac{\norm{\uvec{u}_{n+1}^{(p)} - \uvec{u}_{n+1}^{(p-1)}}}{\norm{ \uvec{u}_{n+1}^{(p)}}}\leq \varepsilon.
	\label{eq:tolerance}
\end{equation} 
A remarkable aspect is that $p$-adaptivity is naturally embedded in the described formulation and that no effort, other than a simple check at each iteration, is required to implement such feature. 
This is an interesting aspect of the novel methods and of the general approach introduced in \cite{loredavide,DOOM}, which is particularly desirable in the context of real-world applications. 
Indeed, the final users want the error to be smaller than a predefined tolerance and to reduce the need for computational resources as much as possible. 
Per se, the adoption of high order methods does not guarantee the optimal balance. 
In fact, depending on the level of mesh refinement and on the prescribed tolerance, low order methods may actually be sufficient and cheaper than high order methods.
The proposed approach is able to automatically detect the required order of accuracy.
Furthermore, some of the authors are involved in other projects whose goal is to investigate such strategy in the context of non-smooth problems.

Let us finally remark that different criteria, other than convergence, can be chosen to halt the iterative process.  For example, in \cite{DOOM}, in a PDE context, the ADER iterations are stopped if the obtained numerical solution does not fulfill some physical constraints.

\section{ADER as Runge--Kutta methods}
\label{sec:RK}
In this section, we will show how the described ADER methods can be rewritten as explicit RK methods, i.e., in the form \eqref{eq:RK_ACCURACY_NODAL} with a strictly lower-triangular matrix $A$ of the coefficients $a_{s,r}$. We will define their Butcher tableaux and we will study their linear stability.

We introduce here the iteration-dependent vectors of the subtimenodes in the reference inteval $[0,1]$ as $\underline{\beta}^{(p)}:=\left( \xi^{0,(p)} , \dots, \xi^{M(p),{(p)}} \right)^T$ for the new modified methods, with $\xi^{m,{(p)}}:=\frac{t^{m,{(p)}}-t_n}{\Delta t}$.

In all cases, for the sake of efficiency, the first iteration is replaced by a simple Euler step: this reduces the number of stages associated to the first iteration and does not spoil the accuracy. In fact, the first iteration is supposed to provide a first order accurate approximation.

We construct the tableaux for a general distribution of subtimenodes. For a fixed accuracy order $P$, we consider $P$ iterations, assuming a final number of subtimenodes, which is chosen to be the minimal (and so the optimal) allowing to reach such accuracy, according to the theoretical analysis presented in Section~\ref{sec:analytical_results}, i.e., $M=P-1$ for equispaced subtimenodes, $M=\left \lceil \frac{P}{2} \right \rceil$ for GLB subtimenodes and $M=\max{\left(\left \lceil \frac{P-1}{2} \right \rceil,1\right) }$ for GLG subtimenodes.

The Butcher tableaux associated to the described ADER methods are reported in Tables \ref{tab:RKADER}, \ref{tab:RK_ADERu} and \ref{tab:RK_ADERdu}, while, the number of RK stages is reported in Tables \ref{tab:S_equispaced}, \ref{tab:S_GLB} and \ref{tab:S_GLG}.
The RK matrices $A$ of the ADER methods under investigation have a block-diagonal structure, with each block being associated to an ADER iteration. Some of the computed rows of the RK $A$ matrices have only zero elements, leading to ``ghost'' stages that do not contribute to the method. These rows have not been considered in the computation of the number of RK stages in Tables~\ref{tab:S_equispaced}, \ref{tab:S_GLB} and \ref{tab:S_GLG}.

In tables~\ref{tab:S_equispaced}, \ref{tab:S_GLB} and \ref{tab:S_GLG}, we report the \textit{theoretical speed-ups} of ADERu and ADERdu (equivalent to ADER-$L^2$) with respect to the original ADER method without interpolations, simply denoted as ``ADER''.
Further, for GLB and GLG subtimenodes, we also report the \textit{theoretical speed-ups} of ADER, ADERu and ADERdu with respect to the non-optimal ADER method using a number of subtimenodes equal to the desired order of accuracy, as in \cite{han2021dec,veiga2021arbitrary,velasco2022spectral,zanotti2015solving,fernandez2022arbitrary,rannabauer2018ader,dumbser2018efficient}, referring to such method as ``classical ADER'' (\cADER). Such method makes use of an unnecessarily large number of subtimenodes at each iteration, indeed, as proven in Section~\ref{sec:analytical_results}, GLB and GLG subtimenodes can achieve the same order with nearly half of the subtimenodes.

The \textit{theoretical speed-ups} are computed as the ratios between the number of RK stages of the reference methods over the ones of the novel methods. Therefore, a \textit{theoretical speed-up} larger than one implies that the investigated scheme is, in theory, faster than the reference one. 
In Section~\ref{sec:Numerics}, the \textit{theoretical speed-ups} will be compared with the \textit{numerical speed-ups}, defined as the ratios between the wall-clock computational times  of the numerical simulations. Even if other factors might come into play, e.g., memory ordering or floating point operations of matrix multiplications between ADER structures, the two quantities should be highly correlated.
This is also due to the fact that, for all the methods, all the matrices can be efficiently precomputed at the beginning of the simulation, as they do not depend on the specific time iteration.

\subsection{ADER}
We start by recalling the definition of the general ADER iteration \eqref{eq:ADER_DeC_iteration}
\begin{equation}
	\undu^{(p)}:=\undu_n+\Delta t B^{-1}\Lambda  \underline{\uvec{G}}(\undu^{(p-1)}).
	\label{eq:ADER_iteration_RK}
\end{equation}
Collecting all the iterations and identifying each state $\undu^{m,(p)}$, associated to the subtimenode $t^m$, of $\undu^{(p)}$ as a RK stage $\by^s$, we get the RK formulation \eqref{eq:RK_ACCURACY_NODAL} characterized by the Butcher tableau reported in Table \ref{tab:RKADER}. The vector $\widetilde{\uvec{b}}$ is defined as $\widetilde{\uvec{b}}^T:=\left(\underline{\widehat{\psi}}(1)\right)^T B^{-1}\Lambda$, with $\underline{\widehat{\psi}}(\xi):=(\widehat{\psi}^0(\xi),\dots,\widehat{\psi}^M(\xi))^T$, and keeps into account the final ADER iteration and the interpolation.
In fact, recalling that $\sum_{m=0}^M \widehat{\psi}^m \equiv 1$, we have
\begin{align}
	\uvec{u}_{n+1}&
	=\left(\underline{\widehat{\psi}}(1)\right)^T\undu^{(P)}
	=\left(\underline{\widehat{\psi}}(1)\right)^T\left[\undu_n+\Delta t B^{-1}\Lambda  \underline{\uvec{G}}(\undu^{(P-1)}) \right]
	=\uvec{u}_n+\Delta t \left(\underline{\widehat{\psi}}(1)\right)^T B^{-1}\Lambda  \underline{\uvec{G}}(\undu^{(P-1)}).
\end{align}
The number of stages is equal to $S=1+(P-1)(M+1)$ and is reported, for equispaced, GLB and GLG subtimenodes, in Tables \ref{tab:S_equispaced}, \ref{tab:S_GLB} and \ref{tab:S_GLG}, neglecting the ``ghost'' stages.
For \cADER~the same relation holds, but the number of subtimenodes is always $M+1$ with $M=P-1$ independently of the chosen subtimenodes.

We remark that, for equispaced subtimenodes, ADER and \cADER~coincide.
Instead, for GLB and GLG, the adoption of the optimal number of subtimenodes for a given order, i.e., the choice of ADER over cADER, determines a substantial decrease in the number of stages, even without extra interpolation processes.

\begin{table}
	\centering
	\begin{tabular}{|c||ccccccc|l|}
		\hline
		$\uvec{c}$& $\uvec{u}_n$ & $\bbu^{(1)}$& $\bbu^{(2)}$&$\bbu^{(3)}$& $\cdots$& $\bbu^{(P-2)}$& $\bbu^{(P-1)}$&A\\
		\hline
		$0$ & 0 & &&&&&&$\uvec{u}_n$\\
		$\vecbeta$& $\vecbeta$ & $\matz$ &&&&&&$\bbu^{(1)}$ \\
		$\vecbeta$& $\underline{0}$ & $B^{-1}\Lambda$ & $\matz$&&&&&$\bbu^{(2)}$\\
		$\vecbeta$& $\underline{0}$&  $\matz$& $B^{-1}\Lambda$ & $\matz$&&&&$\bbu^{(3)}$\\
		&$\vdots$   &$\vdots$ &&$\ddots$& $\ddots$&&& $\vdots$\\
		&$\vdots$   &$\vdots$ &&&$\ddots$& $\ddots$&& $\vdots$\\
		$\vecbeta$ &$\underline{0}$ &  $\matz$&$\cdots$&$\cdots$&$\matz$&$B^{-1}\Lambda$&$\matz$&$\bbu^{(P-1)}$\\ \hline \hline
		$\uvec{b}$& $0$ &  $\vecz$&$\cdots$&$\cdots$&$\cdots$&$\vecz$&$\widetilde{\uvec{b}}^T$&$\uvec{u}_{n+1}$\\ \hline
	\end{tabular}
	\caption{Butcher tableau of the original ADER method, $\uvec{c}$ at the left, $\uvec{b}$ at the bottom, $A$ in the middle. References to the stages are reported on top and on the right side \label{tab:RKADER}}
\end{table}

\begin{table}
	\centering
	\begin{tabular}{|c|c||c|c|c||c|c|} \hline 
		\multicolumn{2}{|c||}{Param} &\multicolumn{3}{c||}{RK Stages}&\multicolumn{2}{c|}{\cADER/ADER-speed-up}\\ \hline
		$P$ & $M$  & \cADER/ADER & \ADERu & \ADERdu  & \ADERu & \ADERdu   \\ \hline 
		2 & 1 & 2  & 2  & 2  & 1.000 & 1.000\\ \hline 
		3 & 2 & 6  & 6  & 4  & 1.000 & 1.500\\ \hline 
		4 & 3 & 12  & 11  & 7  & 1.091 & 1.714\\ \hline 
		5 & 4 & 20  & 17  & 11  & 1.176 & 1.818\\ \hline 
		6 & 5 & 30  & 24  & 16  & 1.250 & 1.875\\ \hline 
		7 & 6 & 42  & 32  & 22  & 1.312 & 1.909\\ \hline 
		8 & 7 & 56  & 41  & 29  & 1.366 & 1.931\\ \hline 
		9 & 8 & 73  & 51  & 37  & 1.431 & 1.973\\ \hline 
		10 & 9 & 90  & 62  & 46  & 1.452 & 1.957\\ \hline 
		11 & 10 & 111  & 74  & 56  & 1.500 & 1.982\\ \hline 
		12 & 11 & 133  & 87  & 67  & 1.529 & 1.985\\ \hline 
		13 & 12 & 156  & 101  & 79  & 1.545 & 1.975\\ \hline 
		14 & 13 & 183  & 116  & 92  & 1.578 & 1.989\\ \hline   
	\end{tabular}
	\caption{Number of stages $\SRK$ for various ADER with equispaced subtimenodes and \textit{theoretical speed-ups} with respect to ADER/\cADER~method  computed as the ratio of the number of RK stages of ADER method over the number of RK stages of the method of interest \label{tab:S_equispaced}}
\end{table}

\begin{table}
	\centering
	\resizebox{\columnwidth}{!}{
		\begin{tabular}{|c|c||c|c|c|c||c|c|c||c|c|} \hline 
			\multicolumn{2}{|c||}{Param} &\multicolumn{4}{c||}{RK Stages}&\multicolumn{3}{c||}{\cADER-speed-up}&\multicolumn{2}{c|}{ADER-speed-up}\\ \hline
			$P$ & $M$  & \cADER & ADER & \ADERu & \ADERdu& ADER  & \ADERu & \ADERdu  & \ADERu & \ADERdu   \\ \hline 
			2 & 1 & 2  & 2  & 2  & 2  & 1.000 & 1.000 & 1.000 & 1.000 & 1.000\\ \hline 
			3 & 2 & 6  & 6  & 6  & 4  & 1.000 & 1.000 & 1.500 & 1.000 & 1.500\\ \hline 
			4 & 2 & 12  & 9  & 9  & 7  & 1.333 & 1.333 & 1.714 & 1.000 & 1.286\\ \hline 
			5 & 3 & 20  & 16  & 15  & 11  & 1.250 & 1.333 & 1.818 & 1.067 & 1.455\\ \hline 
			6 & 3 & 30  & 20  & 19  & 15  & 1.500 & 1.579 & 2.000 & 1.053 & 1.333\\ \hline 
			7 & 4 & 43  & 30  & 27  & 21  & 1.433 & 1.593 & 2.048 & 1.111 & 1.429\\ \hline 
			8 & 4 & 56  & 35  & 32  & 26  & 1.600 & 1.750 & 2.154 & 1.094 & 1.346\\ \hline 
			9 & 5 & 73  & 48  & 42  & 34  & 1.521 & 1.738 & 2.147 & 1.143 & 1.412\\ \hline 
			10 & 5 & 91  & 54  & 48  & 40  & 1.685 & 1.896 & 2.275 & 1.125 & 1.350\\ \hline 
			11 & 6 & 110  & 71  & 60  & 50  & 1.549 & 1.833 & 2.200 & 1.183 & 1.420\\ \hline 
			12 & 6 & 133  & 78  & 67  & 57  & 1.705 & 1.985 & 2.333 & 1.164 & 1.368\\ \hline 
			13 & 7 & 156  & 96  & 81  & 69  & 1.625 & 1.926 & 2.261 & 1.185 & 1.391\\ \hline 
			14 & 7 & 183  & 104  & 89  & 77  & 1.760 & 2.056 & 2.377 & 1.169 & 1.351\\ \hline 
		\end{tabular}
	}
	\caption{Number of stages $\SRK$ for various ADER with GLB subtimenodes and \textit{theoretical speed-ups} with respect to \cADER~and to ADER computed as the ratio of the number of RK stages of the method of \cADER~or ADER, respectively, over the number of RK stages of the method of interest \label{tab:S_GLB}}
\end{table}

\begin{table}
	\centering
	\resizebox{\columnwidth}{!}{
		\begin{tabular}{|c|c||c|c|c|c||c|c|c||c|c|} \hline 
			\multicolumn{2}{|c||}{Param} &\multicolumn{4}{c||}{RK Stages}&\multicolumn{3}{c||}{\cADER-speed-up}&\multicolumn{2}{c|}{ADER-speed-up}\\ \hline
			$P$ & $M$  & \cADER & ADER & \ADERu & \ADERdu& ADER  & \ADERu & \ADERdu  & \ADERu & \ADERdu   \\ \hline 
			2 & 1 & 3  & 3  & 3  & 3  & 1.000 & 1.000 & 1.000 & 1.000 & 1.000\\ \hline 
			3 & 1 & 7  & 5  & 5  & 5  & 1.400 & 1.400 & 1.400 & 1.000 & 1.000\\ \hline 
			4 & 2 & 13  & 10  & 10  & 9  & 1.300 & 1.300 & 1.444 & 1.000 & 1.111\\ \hline 
			5 & 2 & 21  & 13  & 13  & 12  & 1.615 & 1.615 & 1.750 & 1.000 & 1.083\\ \hline 
			6 & 3 & 31  & 21  & 20  & 18  & 1.476 & 1.550 & 1.722 & 1.050 & 1.167\\ \hline 
			7 & 3 & 43  & 25  & 24  & 22  & 1.720 & 1.792 & 1.955 & 1.042 & 1.136\\ \hline 
			8 & 4 & 57  & 36  & 33  & 30  & 1.583 & 1.727 & 1.900 & 1.091 & 1.200\\ \hline 
			9 & 4 & 73  & 41  & 38  & 35  & 1.780 & 1.921 & 2.086 & 1.079 & 1.171\\ \hline 
			10 & 5 & 91  & 55  & 49  & 45  & 1.655 & 1.857 & 2.022 & 1.122 & 1.222\\ \hline 
			11 & 5 & 111  & 61  & 55  & 51  & 1.820 & 2.018 & 2.176 & 1.109 & 1.196\\ \hline 
			12 & 6 & 133  & 78  & 68  & 63  & 1.705 & 1.956 & 2.111 & 1.147 & 1.238\\ \hline 
			13 & 6 & 157  & 85  & 75  & 70  & 1.847 & 2.093 & 2.243 & 1.133 & 1.214\\ \hline 
			14 & 7 & 183  & 105  & 90  & 84  & 1.743 & 2.033 & 2.179 & 1.167 & 1.250\\ \hline 
		\end{tabular}
	}
	\caption{Number of stages $\SRK$ for various ADER with GLG subtimenodes and \textit{theoretical speed-ups} with respect to \cADER~and to ADER computed as the ratio of the number of RK stages of \cADER~or ADER, respectively, over the number of RK stages of   the method of interest  \label{tab:S_GLG}}
\end{table}

\subsection{\ADERu}
Again, we recall the update at the generic iteration $p>1$ characterized by interpolation of $\uvec{u}(t)$
\begin{align}
	&\begin{cases}
		\undu^{*(p-1)}:=H^{(p-1)}\undu^{(p-1)},\\
		\undu^{(p)}=\undu_n^{(p)}+\Delta t \left(B^{(p)} \right)^{-1}\Lambda^{(p)}  \underline{\uvec{G}}(\undu^{*(p-1)}).
	\end{cases}
	\label{eq:ADERu_RK}
\end{align}
It is possible to formulate the previous update in terms of the interpolated states by multiplying the second equation by $H^{(p)}$
\begin{align}
	\begin{split}
		\undu^{*(p)}&=H^{(p)}\undu^{(p)}=
		\undu_n^{(p+1)}+\Delta t H^{(p)}\left(B^{(p)} \right)^{-1}\Lambda^{(p)}  \underline{\uvec{G}}(\undu^{*(p-1)}),
	\end{split}
\end{align}
where we have exploited the fact that the sum of the elements in each row of the interpolation matrix $H^{(p)}$ is equal to $1$.
The Butcher tableau can be therefore constructed as in Table \ref{tab:RK_ADERu}.
The matrices $W^{(p)}$ are defined as
\begin{align}
	W^{(p)} : = \begin{cases}
		H^{(p)}\left(B^{(p)} \right)^{-1}\Lambda^{(p)} \in \R^{(p+2)\times(p+1)}, &\text{if } p =2, \dots, M-1,\\
		\left(B^{(M)} \right)^{-1}\Lambda^{(M)} \in \R^{(M+1)\times (M+1)}, & \text{if } p = M.
	\end{cases}
\end{align}

Notice that we do not have interpolation in the final iterations. We perform the interpolations in the early iterations in order to reach the needed number of subtimenodes that allows to get the desired order of accuracy $P$, then, from $\uvec{u}^{(M)}$ on, we keep iterating using the structures of the standard ADER method with $M+1$ subtimenodes. 
Vector $\widetilde{\uvec{b}}^{(M)}$ is, in fact, defined as $\widetilde{\uvec{b}}$ considering the final number of subtimenodes.
The number of stages is equal to $S=1+(P-1)(M+1) -\frac{(M-1)(M-2)}{2} $, i.e., $\frac{(M-1)(M-2)}{2}$ less with respect to ADER,  and is reported, for equispaced, GLB and GLG subtimenodes, in Tables \ref{tab:S_equispaced}, \ref{tab:S_GLB} and \ref{tab:S_GLG}. We report also the \textit{theoretical speed-up} factor with respect to ADER and \cADER.

\begin{table}
	\centering
	\begin{tabular}{|c||cccccccccccc|l|}
		\hline
		$\uvec{c}$& $\uvec{u}_n$ & $\!\!\bbu^{*(1)}\!\!$& $\!\!\bbu^{*(2)}\!\!$&$\!\!\bbu^{*(3)}\!\!$& $\!\!\cdots\!\!$& $\!\!\bbu^{*(M-2)}\!\!$& $\!\!\bbu^{*(M-1)}\!\!$& $\!\!\bbu^{(M)}\!\!$& $\!\!\bbu^{(M+1)}\!\!$& $\!\!\cdots\!\!$& $\!\!\bbu^{(P-2)}\!\!$& $\!\!\bbu^{(P-1)}\!\!$&A\\
		\hline
		$0$ & 0 & &&&&&&&&&&&$\bu_{n}$\\
		$\vecbeta^{(2)}$& $\vecbeta^{(2)}$ & $\matz$ &&&&&&&&&&&$\bbu^{*(1)}$ \\
		$\vecbeta^{(3)}$& $\!\!\underline{0}\!\!$ & $\!\!\!\!W^{(2)}\!\!\!\!$ & $\matz$&&&&&&&&&&$\bbu^{*(2)}$\\
		$\vecbeta^{(4)}$& $\underline{0}$&  $\matz$& $\!\!\!\!W^{(3)}\!\!\!\!$ & $\matz$&&&&&&&&&$\bbu^{*(3)}$\\
		&$\vdots$   &$\vdots$ &$\ddots$&$\ddots$& $\ddots$&&&&&&&&$\vdots$\\
		&$\vdots$   &$\vdots$ &&$\ddots$&$\ddots$& $\ddots$&&&&&&&$\vdots$\\
		$\vecbeta^{(M)}\!\!$ &$\!\!\underline{0}\!\!\!\!$ &  $\matz$&$\cdots$&$\cdots$&$\matz$ &$\!\!\!\!W^{(M-1)}\!\!\!\!$&$\matz$&&&&&&$\bbu^{*(M-1)}\!\!$\\
		$\vecbeta^{(M)}\!\!$&$\underline{0}$ &   $\matz$&$\cdots$&$\cdots$&$\cdots$&$\matz$&$\!\!\!\!W^{(M)}\!\!\!\!$&$\matz$&&&&&$\bbu^{(M)}$\\
		$\vecbeta^{(M)}\!\!$&$\underline{0}$ &   $\matz$&$\cdots$&$\cdots$&$\cdots$&$\cdots$&$\matz$&$\!\!\!\!W^{(M)}\!\!\!\!$&$\matz$&&&&$\bbu^{(M+1)}$\\
		$\vdots\!\!$&$\vdots$ &   $\vdots$&&&&&&$\ddots$&$\ddots$&$\ddots$&&&$\vdots$\\
		$\vdots\!\!$&$\vdots$ &   $\vdots$&&&&&&&$\ddots$&$\ddots$&$\ddots$&&$\vdots$\\
		$\vecbeta^{(M)}\!\!$&$\underline{0}$ &   $\matz$&$\cdots$&$\cdots$&$\cdots$&$\cdots$&$\cdots$&$\cdots$&&$\matz$&$\!\!\!\!W^{(M)}\!\!\!\!$&$\matz$&$\bbu^{(P-1)}$\\ \hline\hline
		$\uvec{b}$&$\!\!0$ &  $\vecz$&$\cdots$&$\cdots$&$\cdots$&$\cdots$&$\cdots$&$\cdots$&$\cdots$&$\cdots$&$\vecz$&$\!\!\!\!(\widetilde{\uvec{b}}^{(M)})^T\!\!$&$\bu_{n+1}\!\!$\\ \hline
	\end{tabular}
	\caption{Butcher tableau of the \ADERu~method, $\uvec{c}$ at the left, $\uvec{b}$ at the bottom, $A$ in the middle. References to the stages are reported on top and on the right side \label{tab:RK_ADERu}}
\end{table}

\subsection{\ADERdu~and ADER-$L^2$}
Since \ADERdu~and ADER-$L^2$ coincide in an ODE context, they are characterized by the same RK form.
In order to obtain it and to construct the associated Butcher tableau, we recall the general update \eqref{eq:ADERdu}
\begin{align}
	\undu^{(p)}&=\undu_n^{(p)}+\Delta t \left(B^{(p)} \right)^{-1}\Lambda^{(p)}  H^{(p-1)}\underline{\uvec{G}}(\undu^{(p-1)}),
\end{align}
leading to the Butcher tableau reported in Table \ref{tab:RK_ADERdu}.
The matrices $Z^{(p)}$ are defined as
\begin{align}
	Z^{(p)} : = \begin{cases}
		\left(B^{(p)} \right)^{-1}\Lambda^{(p)}H^{(p-1)} \in \R^{(p+1)\times p}, &\text{if } p =2, \dots, M,\\
		\left(B^{(M)} \right)^{-1}\Lambda^{(M)}   \in \R^{(M+1)\times (M+1)}, & \text{if } p=M+1.
	\end{cases}
\end{align}

Same considerations as the ones made in the context of the \ADERu~methods apply here: the final iterations are performed without interpolations to reach the desired accuracy.
The number of stages is equal to $S=1+(P-1)(M+1) -\frac{M(M-1)}{2}$, i.e., $\frac{M(M-1)}{2}$ less with respect to ADER. Again, it can be found for equispaced, GLB and GLG subtimenodes, in Tables \ref{tab:S_equispaced}, \ref{tab:S_GLB} and \ref{tab:S_GLG}, together with the \textit{theoretical speed-up} with respect to ADER and \cADER.

\begin{table}
	\centering
	\begin{tabular}{|c||cccccccccccc|l|}
		\hline
		$\uvec{c}$& $\uvec{u}_n$ & $\!\!\bbu^{(1)}\!\!$& $\!\!\bbu^{(2)}\!\!$&$\!\!\bbu^{(3)}\!\!$& $\!\!\cdots\!\!$& $\!\!\bbu^{(M-2)}\!\!$& $\!\!\bbu^{(M-1)}\!\!$& $\!\!\bbu^{(M)}\!\!$& $\!\!\bbu^{(M+1)}\!\!$& $\!\!\cdots\!\!$& $\!\!\bbu^{(P-2)}\!\!$& $\!\!\bbu^{(P-1)}\!\!$&A\\
		\hline
		$0$ & 0 & &&&&&&&&&&&$\bu_{n}$\\
		$\vecbeta^{(1)}$& $\vecbeta^{(1)}$ & $\matz$ &&&&&&&&&&&$\bbu^{(1)}$ \\
		$\vecbeta^{(2)}$& $\!\!\underline{0}\!\!$ & $\!\!\!\!Z^{(2)}\!\!\!\!$ & $\matz$&&&&&&&&&&$\bbu^{(2)}$\\
		$\vecbeta^{(3)}$& $\underline{0}$&  $\matz$& $\!\!\!\!Z^{(3)}\!\!\!\!$ & $\matz$&&&&&&&&&$\bbu^{(3)}$\\
		&$\vdots$   &$\vdots$ &$\ddots$&$\ddots$& $\ddots$&&&&&&&&$\vdots$\\
		&$\vdots$   &$\vdots$ &&$\ddots$&$\ddots$& $\ddots$&&&&&&&$\vdots$\\
		$\vecbeta^{(M-1)}\!\!$ &$\!\!\underline{0}\!\!\!\!$ &  $\matz$&$\cdots$&$\cdots$&$\matz$ &$\!\!\!\!Z^{(M-1)}\!\!\!\!$&$\matz$&&&&&&$\bbu^{(M-1)}\!\!$\\
		$\vecbeta^{(M)}\!\!$&$\underline{0}$ &   $\matz$&$\cdots$&$\cdots$&$\cdots$&$\matz$&$\!\!\!\!Z^{(M)}\!\!\!\!$&$\matz$&&&&&$\bbu^{(M)}$\\
		$\vecbeta^{(M)}\!\!$&$\underline{0}$ &   $\matz$&$\cdots$&$\cdots$&$\cdots$&$\cdots$&$\matz$&$\!\!\!\!Z^{(M+1)}\!\!\!\!$&$\matz$&&&&$\bbu^{(M+1)}$\\
		$\vdots\!\!$&$\vdots$ &   $\vdots$&&&&&&$\ddots$&$\ddots$&$\ddots$&&&$\vdots$\\
		$\vdots\!\!$&$\vdots$ &   $\vdots$&&&&&&&$\ddots$&$\ddots$&$\ddots$&&$\vdots$\\
		$\vecbeta^{(M)}\!\!$&$\underline{0}$ &   $\matz$&$\cdots$&$\cdots$&$\cdots$&$\cdots$&$\cdots$&$\cdots$&&$\matz$&$\!\!\!\!Z^{(M+1)}\!\!\!\!$&$\matz$&$\bbu^{(P-1)}$\\ \hline\hline
		$\uvec{b}$&$\!\!0$ &  $\vecz$&$\cdots$&$\cdots$&$\cdots$&$\cdots$&$\cdots$&$\cdots$&$\cdots$&$\cdots$&$\vecz$&$\!\!\!\!(\widetilde{\uvec{b}}^{(M)})^T\!\!$&$\bu_{n+1}\!\!$\\ \hline
	\end{tabular} 
	\caption{Butcher tableau of the \ADERdu~and ADER-$L^2$ method, $\uvec{c}$ at the left, $\uvec{b}$ at the bottom, $A$ in the middle. References to the stages are reported on top and on the right side \label{tab:RK_ADERdu}}
\end{table}

Before studying the linear stability of the methods, as an example and for the sake of completeness, we report, in Figure~\ref{fig:sparsity}, the sparsity pattern of the $A$ matrices of the \cADER, ADER, ADERu and ADERdu (equivalent to ADER-$L^2$) methods for order 7 and GLB subtimenodes. The block-structure of all matrices is essentially identical, with each block corresponding to one iteration, however, the sizes of the blocks differ a lot.
All the blocks in the \cADER~matrix are larger than the ones in the ADER matrix, as the latter method makes use of the optimal number of subtimenodes. Moreover, for both \cADER~and ADER, the size of the blocks is always the same, instead, for ADERu and ADERdu, the size of the blocks increases along the iterations, as a result of the interpolations.
We remark that the non-zero entries in the first column of the ADERdu matrix are determined by the ghost stages, corresponding, in practice, to the first stage $\uvec{u}_n$.

\begin{figure}
	\centering
	\newcommand{\order}{7}
	\includegraphics[width=0.24\textwidth, trim ={70 0 90 0}, clip]{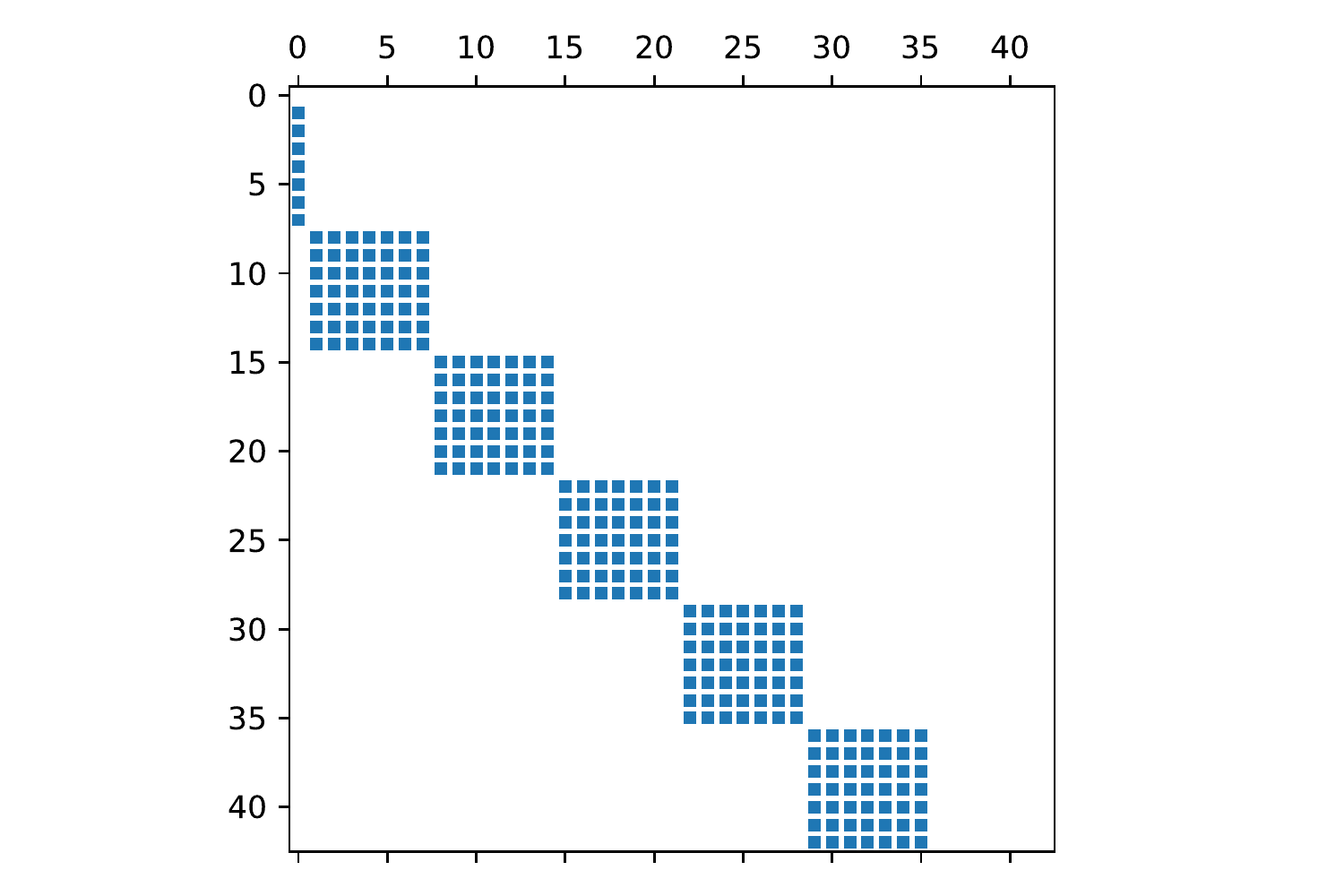}
	\includegraphics[width=0.24\textwidth, trim ={70 0 90 0}, clip]{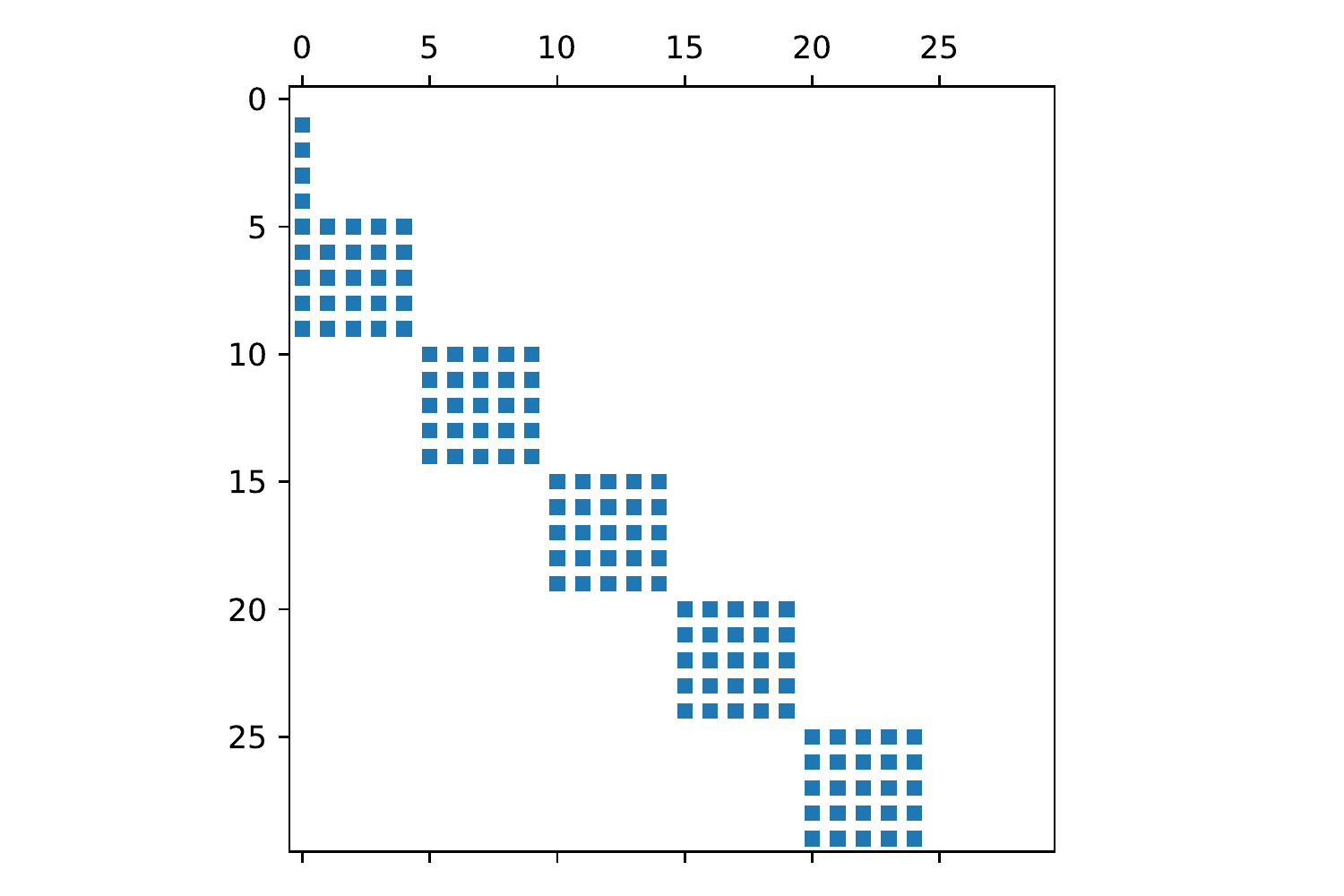}
	\includegraphics[width=0.24\textwidth, trim ={70 0 90 0}, clip]{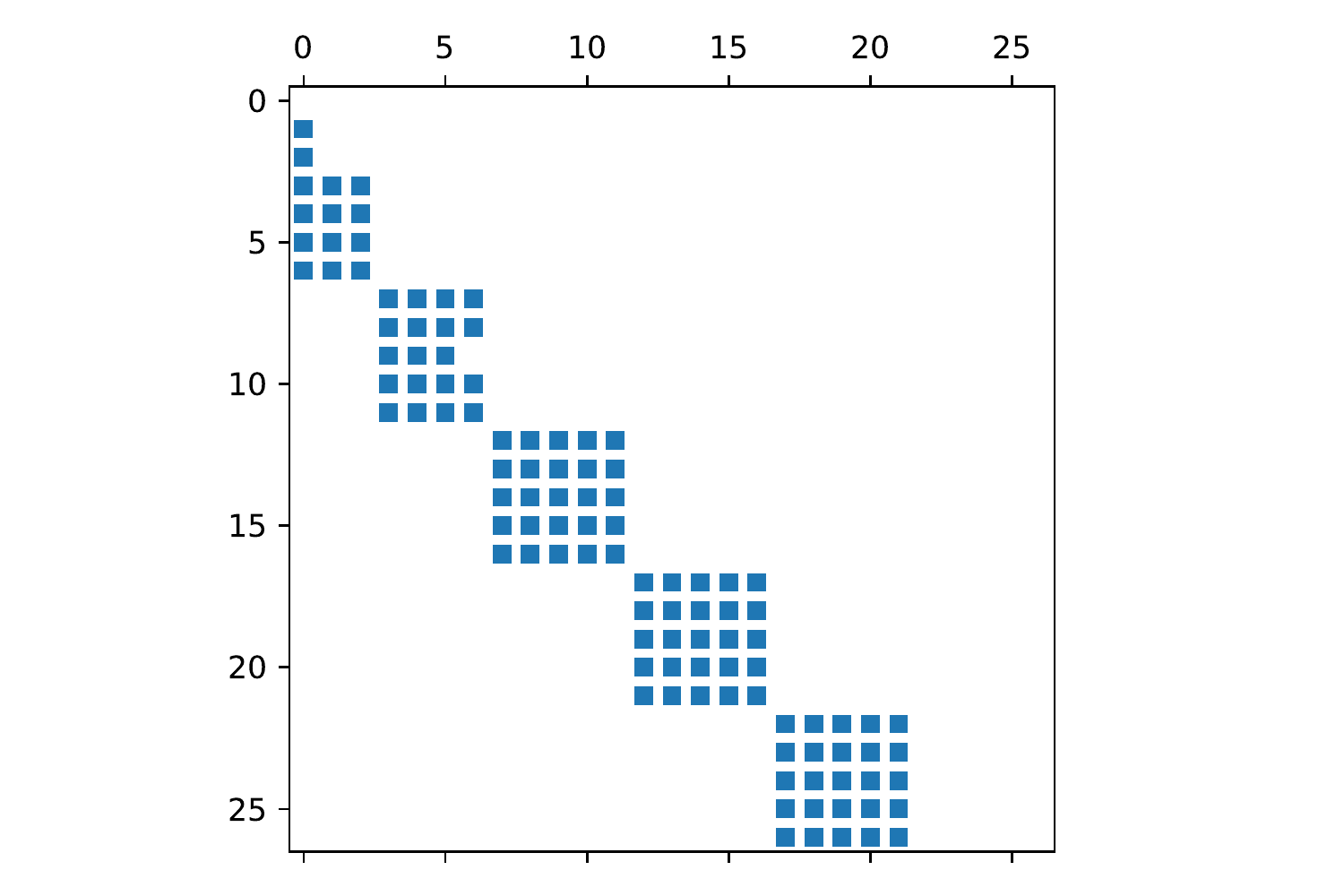}
	\includegraphics[width=0.24\textwidth, trim ={70 0 90 0}, clip]{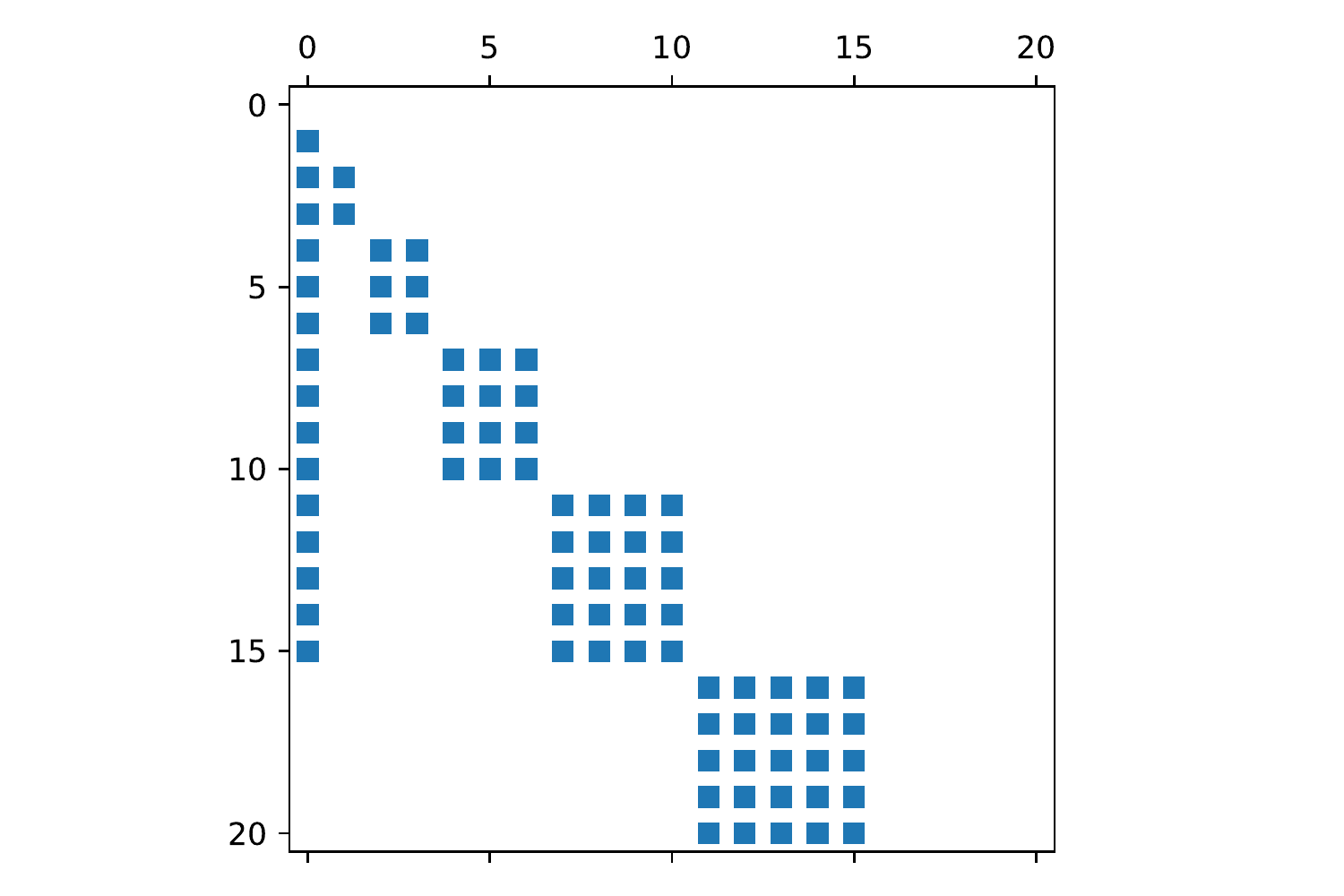}
	\caption{Sparsity pattern of the $A$ matrix with GLB subtimenodes and order \order. From left to right: \cADER, ADER, ADERu and ADERdu (equivalent to ADER-$L^2$). The references to the stages indices are reported on the left and on top. In \cADER~the blocks have size 7, in ADER they have size 5, in ADERu and ADERdu they have increasing sizes from 2 to 5}\label{fig:sparsity}
\end{figure}

\subsection{Linear stability}

The linear stability of a RK method is studied by analyzing the asymptotic behavior of the numerical solution of the method applied to Dahlquist's equation
$\frac{d}{dt}u(t)=\lambda u(t)$,
where $\lambda\in \mathbb{C}$ with $Re(\lambda)<0$. The linearity of the problem and of the method makes possible to express $u_{n+1}$ as
$
u_{n+1}=R(\lambda \dt) u_n,
$
where $R(\cdot)$ is the so-called stability function of the method.
In the context of the novel methods, the following result holds.
\begin{theorem}\label{th:bDeC_equivalence} 
	The stability function of any ADER, \ADERu~and \ADERdu~(and so ADER-$L^2$) method of order $P$ is
	\begin{equation}\label{eq:stab_func}
		R(z) = \sum_{r=0}^{P} \frac{z^r}{r!},
	\end{equation}
	independently of the distribution of the subtimenodes.
\end{theorem}  
The proof of the previous result is identical to the proof of Theorem 6.2 presented in \cite{loredavide} in the context of the study of the linear stability of the bDeC, bDeCu and bDeCdu methods and it is based on the particular block-structure of the RK matrix $A$ of the investigated methods. The reader is referred to \cite{loredavide} for further details.
An interesting consequence of Theorem \ref{th:bDeC_equivalence} is given by the following corollary.
\begin{corollary}\label{cor:bDeC_ADER_equivalence} 
	The bDeC, bDeCu, bDeCdu, ADER, \ADERu~and \ADERdu~(and so ADER-$L^2$) methods of order $P$ share the same stability function, independently of the distribution of the subtimenodes.
\end{corollary}  
The methods mentioned in the previous result are therefore equivalent on linear problems and characterized by the same stability region.

The explicit characterization of the stability function can be used to find the largest $\Delta t$ for which the methods are stable. From classical RK analysis \cite{hairer1987solving}, one should determine the stability region from the stability function and choose $\Delta t$ in such a way that all $\Delta t  \lambda_i (\uvec{u})$ lie inside the stability region, with  $\lambda_i (\uvec{u})$ being the generic eigenvalue of the Jacobian matrix $\frac{\partial \uvec{G}}{\partial \uvec{u}}(\uvec{u})$.
We recall that the set of complex numbers $\mathcal{S}:=\lbrace z\in \mathbb C: \vert R(z) \vert<1\rbrace$ constitutes the stability region of the scheme.
For the simple scalar problem $\frac{d}{dt}u=-u$, the bounds for $\dt$, reported in Table~\ref{tab:dt_bounds}, guarantee the stability of the methods. Indeed, for systems with real eigenvalues, the same bounds rescaled by a factor $\frac{1}{ \rho (\uvec{u})}$, with $\rho(\uvec{u})$ being the spectral radius of $\frac{\partial \uvec{G}}{\partial \uvec{u}}(\uvec{u})$, guarantee (linear) stability. For more general systems with complex eigenvalues, the choice of $\Delta t$ is more difficult and one has to take into account the whole stability region and the distribution of the eigenvalues in the complex plane.
We remark that, since the methods are equivalent on linear problems, they are subjected to the same linear stability bounds. 

\begin{table}
	\centering
	\begin{tabular}{|c|c|}\hline
		Order & $\Delta t$\\ \hline
		1&2\\
		2&2\\
		3&2.51\\
		\hline
	\end{tabular}
	\begin{tabular}{|c|c|}\hline
		Order & $\Delta t$\\ \hline
		4&2.79\\
		5&3.22\\
		6&3.55\\
		\hline
	\end{tabular}
	\begin{tabular}{|c|c|}\hline
		Order & $\Delta t$\\ \hline
		7&3.95\\
		8&4.31\\
		9&4.70\\
		\hline
	\end{tabular}	\begin{tabular}{|c|c|}\hline
		Order & $\Delta t$\\ \hline
		10&5.07\\
		11&5.45\\
		12&5.85\\
		\hline
	\end{tabular}
	\caption{Stability bounds on $\Delta t$ for the scalar problem $\frac{d}{dt}u=-u$ for different orders for bDeC, bDeCu and bDeCdu. Up to a simple rescaling by $\frac{1}{\rho (\uvec{u})}$, with $\rho (\uvec{u})$ being the spectral radius of the Jacobian matrix $\frac{\partial \uvec{G}}{\partial \uvec{u}}(\uvec{u})$, the bounds can be assumed for well--posed ODEs with real eigenvalues}\label{tab:dt_bounds}
\end{table}

\section{Application to PDEs with Spectral Difference schemes}
\label{sec:SD}
In this section, we show how to apply ADER methods in a PDE context with the SD space discretization.
Let us consider the monodimensional hyperbolic PDE
\begin{equation}
	\frac{\partial}{\partial t}\uvec{u}(x,t)+\frac{\partial}{\partial x}\uvec{F}(\uvec{u}(x,t))=\uvec{0}, \quad (x,t)\in [x_L,x_R] \times \mathbb{R}^+_0,
	\label{eq:sys}
\end{equation}
where $\uvec{u}:[x_L,x_R] \times \mathbb{R}^+_0\rightarrow \mathbb{R}^Q$ is the unknown solution and $\uvec{F}:\mathbb{R}^{Q}\rightarrow \mathbb{R}^{Q}$ is the flux. 
We introduce a tessellation $\tess$ of $[x_L,x_R]$ with non-overlapping segments $K$ and we adopt a classical nodal DG discretization of the numerical solution. 
Thus, globally the solution is approximated as a discontinuous piecewise polynomial function in $(V_M)^Q$ where $V_M:=\left\lbrace g\in L^2(\Omega)~s.t.~g\vert_K\in \mathbb{P}_M(K)\right\rbrace$, leading to accuracy $M+1$ for sufficiently smooth solutions. 
Locally,  in each element $K$, we represent the solution by interpolating it in $M+1$ solution points $x_i^s$
\begin{align}
	\uvec{u}_h(x,t):=\sum_{i=1}^{M+1} \uvec{u}_{i}^{s}(t)\varphi_{i}^{s}(x), \quad \forall x \in K,
	\label{eq:uh_PDE}
\end{align}
where the functions $\varphi_{i}^{s}$ are the Lagrange polynomials of degree $M$ associated to the solution points $x_i^s$ and $\uvec{u}_{i}^{s}$ the time-dependent values in the same points. 
For the sake of compactness, the label $K$ on the local coefficients and basis functions is omitted.
In each element, out of the approximation \eqref{eq:uh_PDE}, we can reconstruct the flux by interpolating it in $M+2$ flux points, $x_i^f$. In particular, we set $\uvec{F}_{i}^{f}(t):=\uvec{F}(\uvec{u}_h(x_i^f,t))$ and we define
\begin{align}
	\uvec{F}_h(x,t):=\sum_{i=1}^{M+2} \uvec{F}_{i}^{f}(t)\varphi_{i}^{f}(x), \quad \forall x \in K.
	\label{eq:Fh_PDE}
\end{align}

In order to guarantee a coupling between the elements $K$ and to avoid the instabilities of central schemes, we include the extrema of each segment $K$ among the flux points $x_i^f$ and we use a numerical flux $\uvec{F}^{num}$, rather than a direct evaluation, to define the flux value $\uvec{F}_{i}^{f}(t)$ in such extrema, keeping into account the trace of $\uvec{u}_h$ from the neighboring segments. Hence, the approximated flux is given by
\begin{align}
	\uvec{F}_h(x,t):=\uvec{F}^{num}(x_1^f,t)\varphi_{1}^{f}(x)+\sum_{i=2}^{M+1} \uvec{F}_{i}^{f}(t)\varphi_{i}^{f}(x)+\uvec{F}^{num}(x_{M+2}^f,t)\varphi_{M+2}^{f}(x), \quad \forall x \in K,
	\label{eq:Fh_PDE_flux_num}
\end{align}
where $x_1^f$ and $x_{M+2}^f$ are the extrema of the cell $K$.

The semidiscretization of the SD method is finally obtained by imposing that the discretizations of the solution and of the flux satisfy the PDE \eqref{eq:sys} in each solution point 
\begin{align}
	\frac{\partial}{\partial t}\uvec{u}_i^s(t)+\frac{\partial}{\partial x}\uvec{F}_h(x_i^s,t)=\uvec{0},\quad \forall x_i^s\in K, \quad \forall K\in \tess,
	\label{eq:semidiscretizazion_SD}
\end{align}
which is a system of ODEs, in the unknowns $\uvec{u}_i^s$, that must be solved in time.  
System \eqref{eq:semidiscretizazion_SD} is in the form \eqref{eq:ODE} and can be therefore solved with the described explicit ADER methods.  We use a Courant--Friedrichs--Lewy condition adapted to the SD scheme, taken from the linear stability analysis presented in \cite{sd_cfl}, and compute the time step as
\begin{equation}
	\label{eq:sd-cfl}
	\Delta t = \frac{C}{M+1}\frac{\Delta x}{v_{max}},
\end{equation} 
for some $C \in (0,1]$, with $v_{max}$ being the maximum wave speed of the system in absolute value, i.e., the spectral radius of the Jacobian of $\uvec{F}$.

\begin{remark}[On the spatial coupling in SD methods]
	If we identify the solution points as degrees of freedom of classical finite element formulations, we have that the spatial coupling, in the context of SD methods, does not differ much from the one of a standard DG formulation, with the exception that no mass matrix is present in SD methods. 
	Indeed, the local interpolation of the flux determines the coupling of the degrees of freedom within each cell, while, the coupling between neighboring cells is guaranteed, as already remarked, by the inclusion of the extrema of each segment among the flux points and the adoption of a numerical flux to define the flux values there.
\end{remark}

We presented the SD method in a 1-dimensional setting but the extension to the multidimensional case is straightforward on Cartesian grids, applying the same arguments dimension by dimension, as proposed, \textit{inter alia}, in \cite{veiga2021arbitrary,velasco2022spectral}. 

\begin{figure}
	\begin{center}
		\includegraphics[width=0.6\textwidth, trim={0 12.4cm 0 7cm},clip]{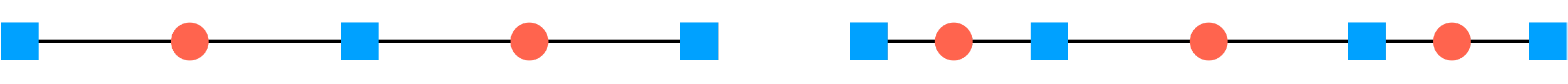}
		\includegraphics[width=0.18\textwidth, trim={17cm 9.2cm 5cm 10cm},clip]{figures/element_sd.pdf}
		\includegraphics[width=0.18\textwidth, trim={17cm 7.4cm 5cm 12.2cm},clip]{figures/element_sd.pdf}
	\end{center}
	\caption{SD element for second and third order in $1$-dimension}
	\label{fig:sd_element}
\end{figure}

The stability of SD methods has been shown to be independent of the choice of the solution and flux points under mild assumptions (i.e., solution points are located between flux points) in \cite{abeele2008}.
Further, it has been proven that there exist flux point placements for which the method is stable, both  for 1-dimensional intervals and Cartesian meshes.  
In \cite{jameson2010}, the stability of SD schemes is established when the interior flux collocation points are the zeros of the Legendre polynomials.  For further information on the SD-ADER scheme,  the reader is referred to \cite{veiga2021arbitrary,velasco2022spectral}.

\begin{remark}[CFL condition and convergence of the ADER iterative procedure]
	It is possible to notice a relation between the CFL condition \eqref{eq:sd-cfl} and the ADER convergence condition from Proposition \ref{prop:iterative_procedure}.
	Indeed, after linearization and after applying the Fourier transform, as done in \cite{sd_cfl}, we observe that the PDE can be rewritten as a system of ODEs like in \eqref{eq:ODE}.
	In particular, the Lipschitz-continuity constant of the resulting $\uvec{G}$ is proportional to $\frac{ v_{max} }{\Delta x}$, but it also depends on another coefficient, i.e., $M+1$ \cite{sd_cfl}.
	This is actually less restrictive with respect to classical finite element schemes, where the constant is proportional to $2M+1$.
	For practical purposes, we include also a safety factor $C \in \R^+$, keeping the $\Delta t$ a little lower than the theoretical bound, as $v_{max}$ is known only at time $t_n$ but not on the whole interval $[t_n,t_{n+1}]$.
\end{remark}

The treatment of shocks and oscillations is done through an \textit{a-posteriori} limiter described in \ref{app:limiter}.

\section{Numerical results}
\label{sec:Numerics}

In this section, we will numerically investigate the properties of the proposed improvements of ADER methods on several benchmarks both for ODEs and PDEs. 
Since the interpretation of ADER as DeC is not new \cite{han2021dec}, we do not investigate the \textit{numerical speed-up} coming from the adoption of the optimal number of iterations rather than solving iteratively the \ADERIWF~up to machine precision. Instead, we always assume a number of iterations equal to the desired order of accuracy and we focus on the impact of the modifications proposed in Sections \ref{sec:analytical_results} and \ref{sec:ADERNEW}, which are the main novelties of this work.
In particular, we study the \textit{numerical speed-up} of the methods with reduced number of GLB and GLG subtimenodes for a fixed order, according to the results presented in Section~\ref{sec:analytical_results}, simply indicated as ADER, with respect to cADER, characterized by a number of subtimenodes always equal to the desired order. 
Moreover, we investigate the \textit{numerical speed-ups} of the novel \ADERu~and \ADERdu~methods, and their adaptive versions, with respect to both ADER and cADER. We recall that \cADER~methods make use of a number of subtimenodes equal to desired order of accuracy. As proven in Section~\ref{sec:analytical_results}, such number is non-optimal for GLB and GLG subtimenodes.

The integral terms of the ADER structures are computed exactly for equispaced and GLG subtimenodes using the GLG quadrature formula, while for GLB subtimenodes we adopt the associated quadrature leading to an underintegrated diagonal matrix $\Lambda$.
This choices, for GLB and GLG subtimenodes, determine the high order implicit RK methods associated to the ADER methods to be respectively the \ADERRK-GLB and \ADERRK-GLG methods presented in Section~\ref{sec:analytical_results}. 
All the ADER matrices defined in the previous sections are precomputed at the beginning of the simulation, as they remain identical in every time step. This also holds for \ADERu~and \ADERdu, for which the structures change along the iterative procedure but do not depend on the specific time step.

Finally, since ADER, \ADERu~and \ADERdu~are equivalent for order 2 and since the focus of this work is on (arbitrary) high order,
we will investigate the methods from order 3 on.

\subsection{ODE tests}
In this section, we focus on two ODE benchmarks: a simple linear system, which allows to verify the equivalence result presented in Theorem \ref{th:bDeC_equivalence}, and a more involved problem, the \CC~test presented in \cite{enright1987two}, which allows to assess the performance of the methods in the context of real applications.

\subsubsection{Linear system}
\newcommand{\test}{linear_system2}
\newcommand{\testname}{Linear system}
The linear system under investigation reads
\begin{equation}
	\begin{cases}
		u' = -5u+v\\
		v' = 5u-v
	\end{cases}, \qquad \begin{pmatrix}
		u_0\\v_0
	\end{pmatrix} = \begin{pmatrix}
		0.9\\0.1
	\end{pmatrix},
\end{equation}
and the exact solution is given by $u(t)= u_0 + (1-e^{-6t}) (-5u_0+v_0)$ and $v(t)=1-u(t)$. 
The problem is indeed very simple and, in fact, it has been chosen for the main purpose of verifying the analytical result summarized in Theorem \ref{th:bDeC_equivalence}, i.e., the equivalence on linear systems of ADER, \ADERu~and \ADERdu~of a given order independently of the adopted subtimenodes. 
In particular, due to the independence of the choice of subtimenodes, the same result applies to the \cADER, from which we expect the same behavior as for the other schemes for any order.
The final time is set to be $T=1$.

The results of the convergence analysis are displayed in Figure~\ref{fig:convergence_linear_system}. The expected order of accuracy is obtained for all the methods and also the expected equivalence on linear problems is numerically confirmed.
The error against the computational time is reported in Figure~\ref{fig:conv_time_linear_system}. 
The adoption of the optimal number of subtimenodes is associated to a clear computational advantage,  with the errors of ADER being always much smaller than \cADER~ones for a given computational time.  Moreover, the novel \ADERu~and \ADERdu~methods are definitely faster than ADER in the context of equispaced subtimenodes. The computational performance of \ADERu~is slightly worse than ADER for GLB and GLG subtimenodes, while, \ADERdu~guarantees a computational advantage with respect to ADER for GLB subtimenodes and has the same computational performance as ADER for GLB subtimenodes.

The adaptive versions of \ADERu~and \ADERdu, respectively in black and gray, have been tested with the convergence criterion \eqref{eq:tolerance} and a tolerance $\varepsilon=10^{-8}$.
The methods are able to automatically choose the order of accuracy to fulfill the desired error condition, as it can be observed by the error lines that lie uniformly below the tolerance level in Figure~\ref{fig:convergence_linear_system}.
For the prescribed tolerance, their performances are similar to the ones of very high order schemes, as can be seen in Figure~\ref{fig:conv_time_linear_system}, with error lines approaching the Pareto front for big time steps.
In Figure~\ref{fig:p_adaptive_linear}, we observe how the average number of iterations required to achieve the expected accuracy changes with respect to $\Delta t$. 
As expected, larger time steps are associated to a higher number of iterations to reach the convergence tolerance.
The small standard deviation in the number of iterations means that, on average, fixed tolerance $\varepsilon$ and time step $\Delta t$ correspond to a fixed order.

\begin{figure}
	\newcommand{\quadrature}{equispaced}
	\includegraphics[width=0.28\textwidth, trim={24 0 120 0}, clip]	{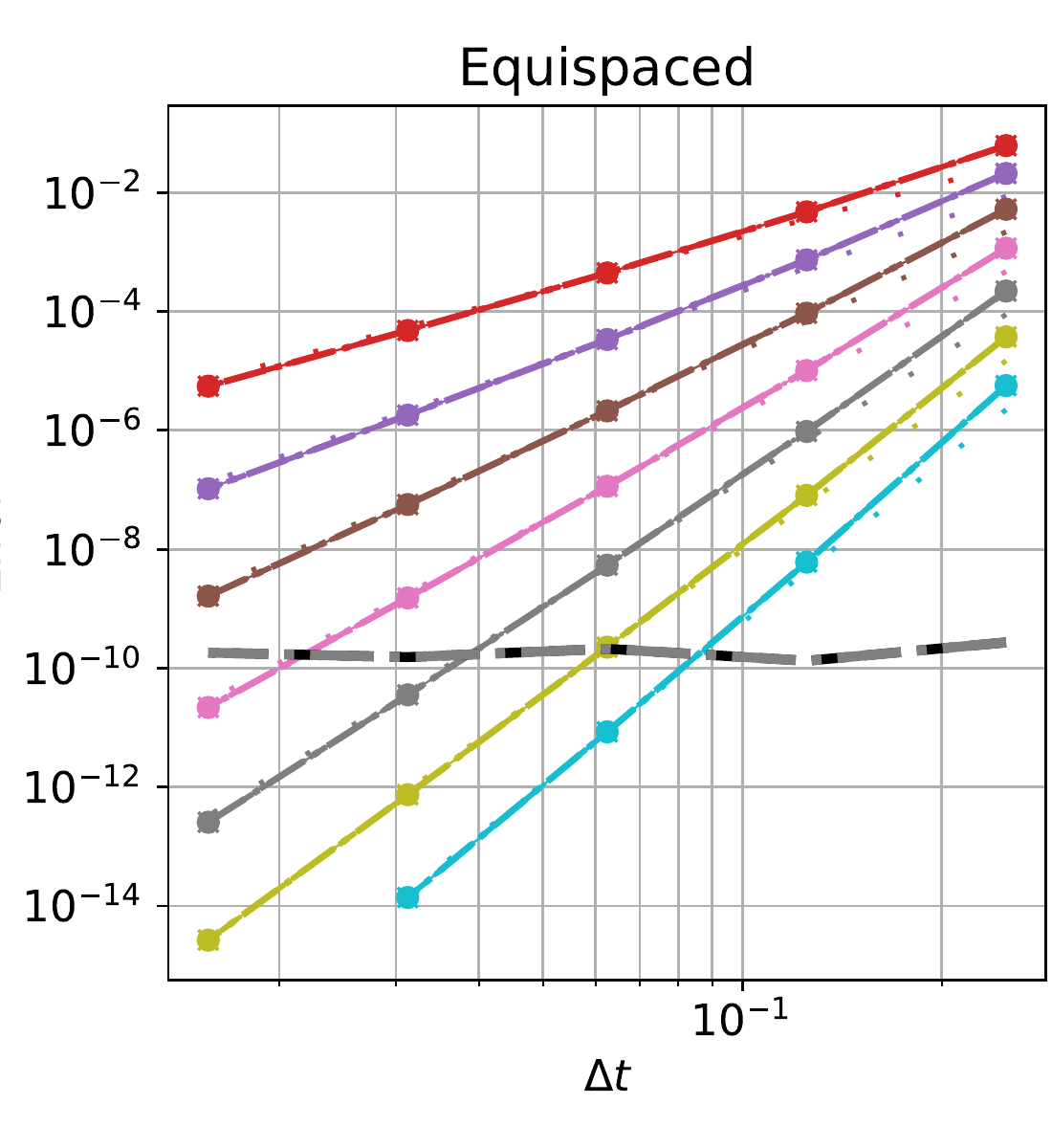}
	\renewcommand{\quadrature}{gaussLobatto}
	\includegraphics[width=0.28\textwidth, trim={24 0 120 0}, clip]{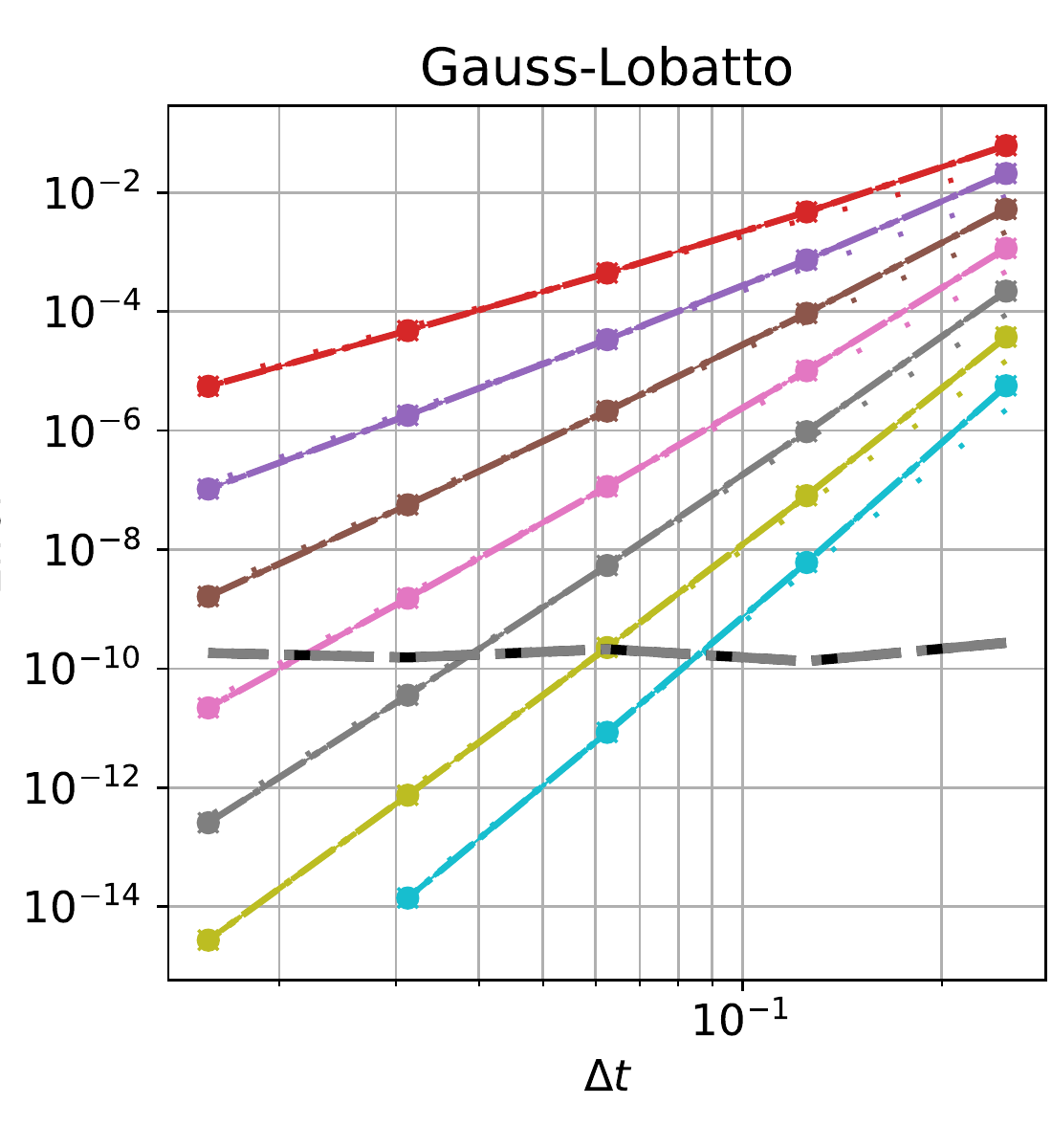}
	\renewcommand{\quadrature}{gaussLegendre}
	\includegraphics[width=0.28\textwidth, trim={24 0 120 0}, clip]{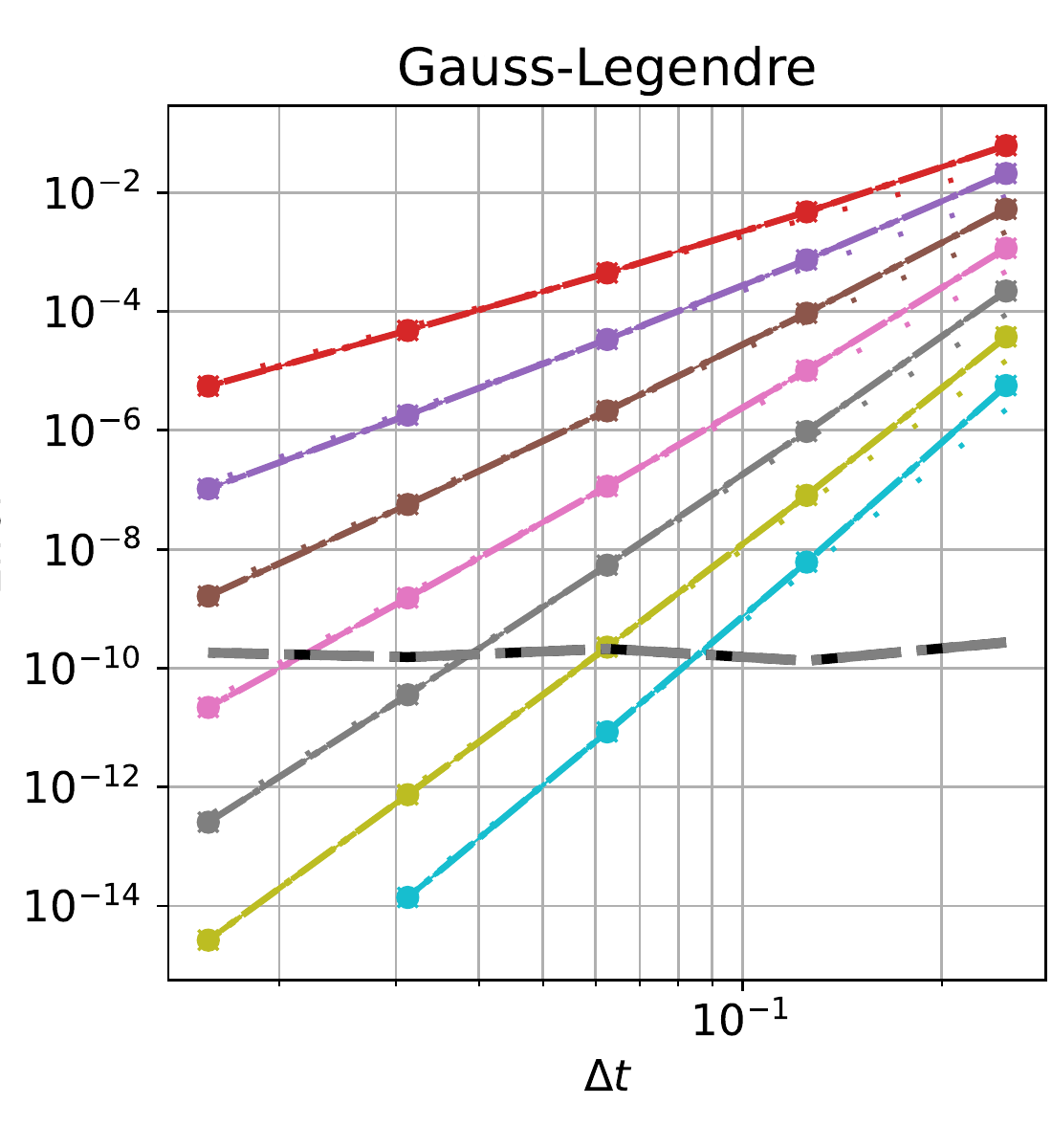}
	\includegraphics[width=0.12\textwidth, trim={340 20 0 40}, clip]{figures/adaptive/convergence_ADER_wc_adaptive_\quadrature_staggered_\test.pdf}
	\caption{\testname: Error decay for various methods and orders. The ``ref'' line is the reference for every order of accuracy
	}
	\label{fig:convergence_linear_system}
\end{figure}

\begin{figure}
	\newcommand{\quadrature}{equispaced}
	\includegraphics[width=0.28\textwidth, trim={24 0 120 0}, clip]	{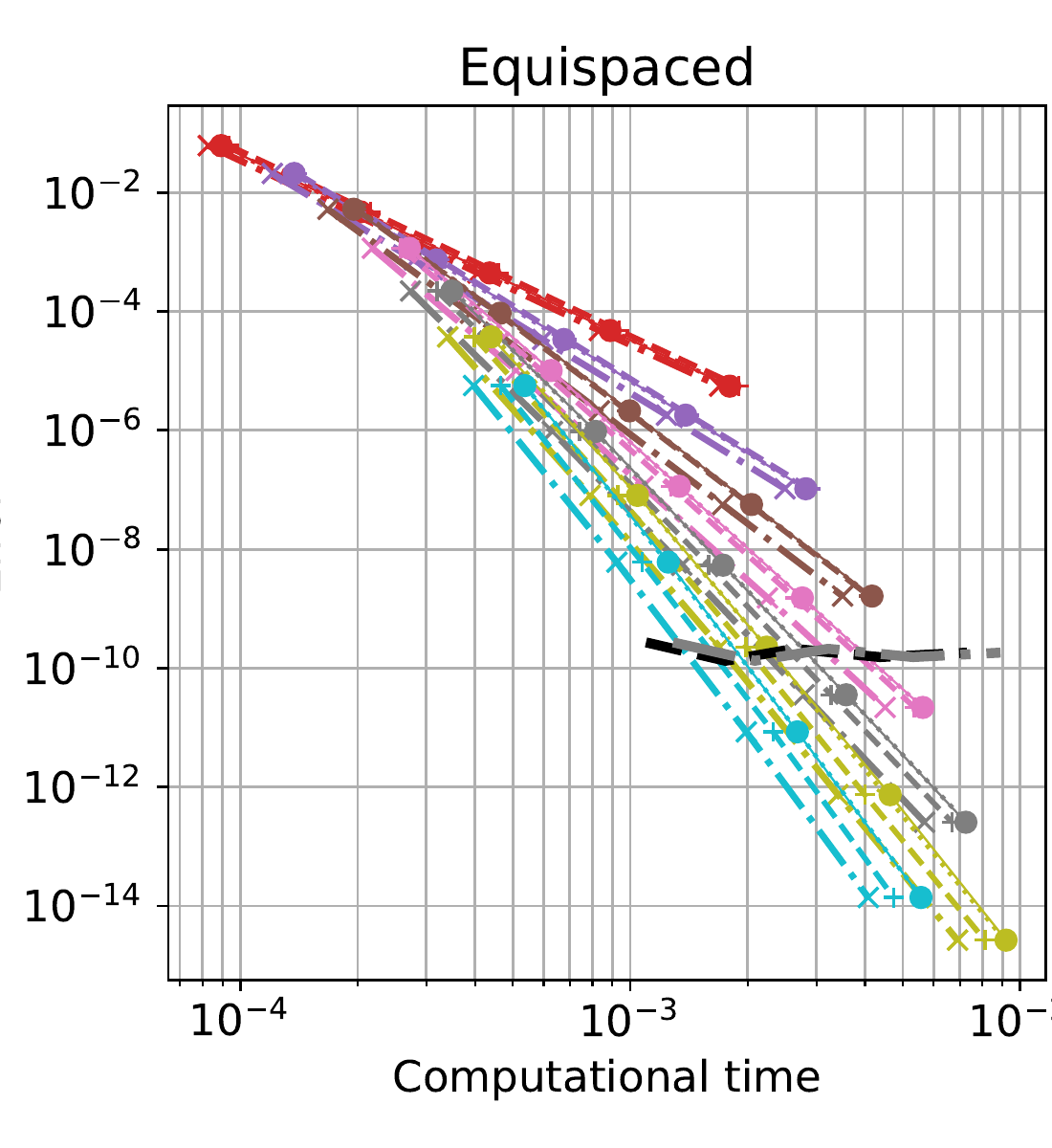}
	\renewcommand{\quadrature}{gaussLobatto}
	\includegraphics[width=0.28\textwidth, trim={24 0 120 0}, clip]{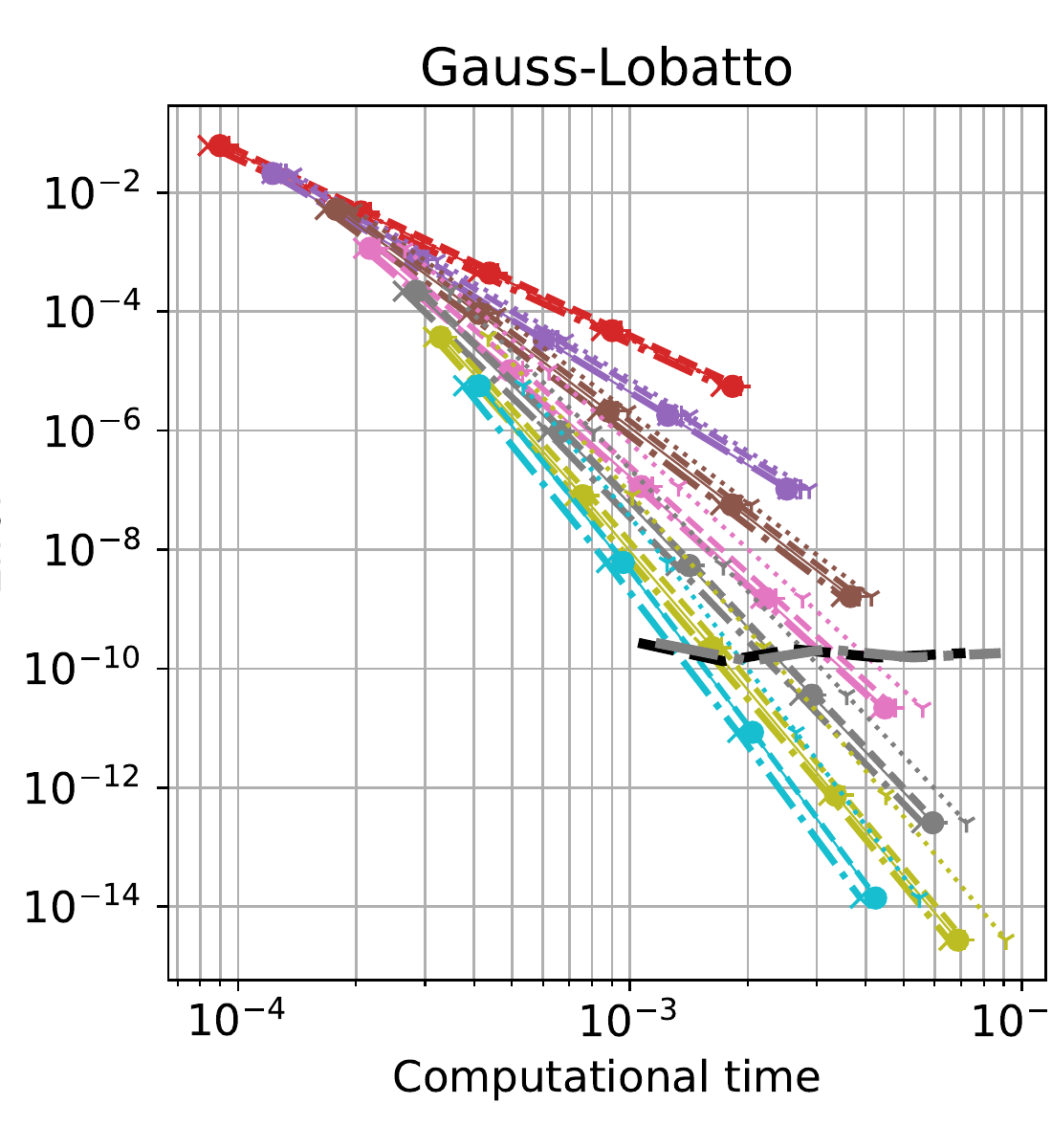}
	\renewcommand{\quadrature}{gaussLegendre}
	\includegraphics[width=0.28\textwidth, trim={24 0 120 0}, clip]{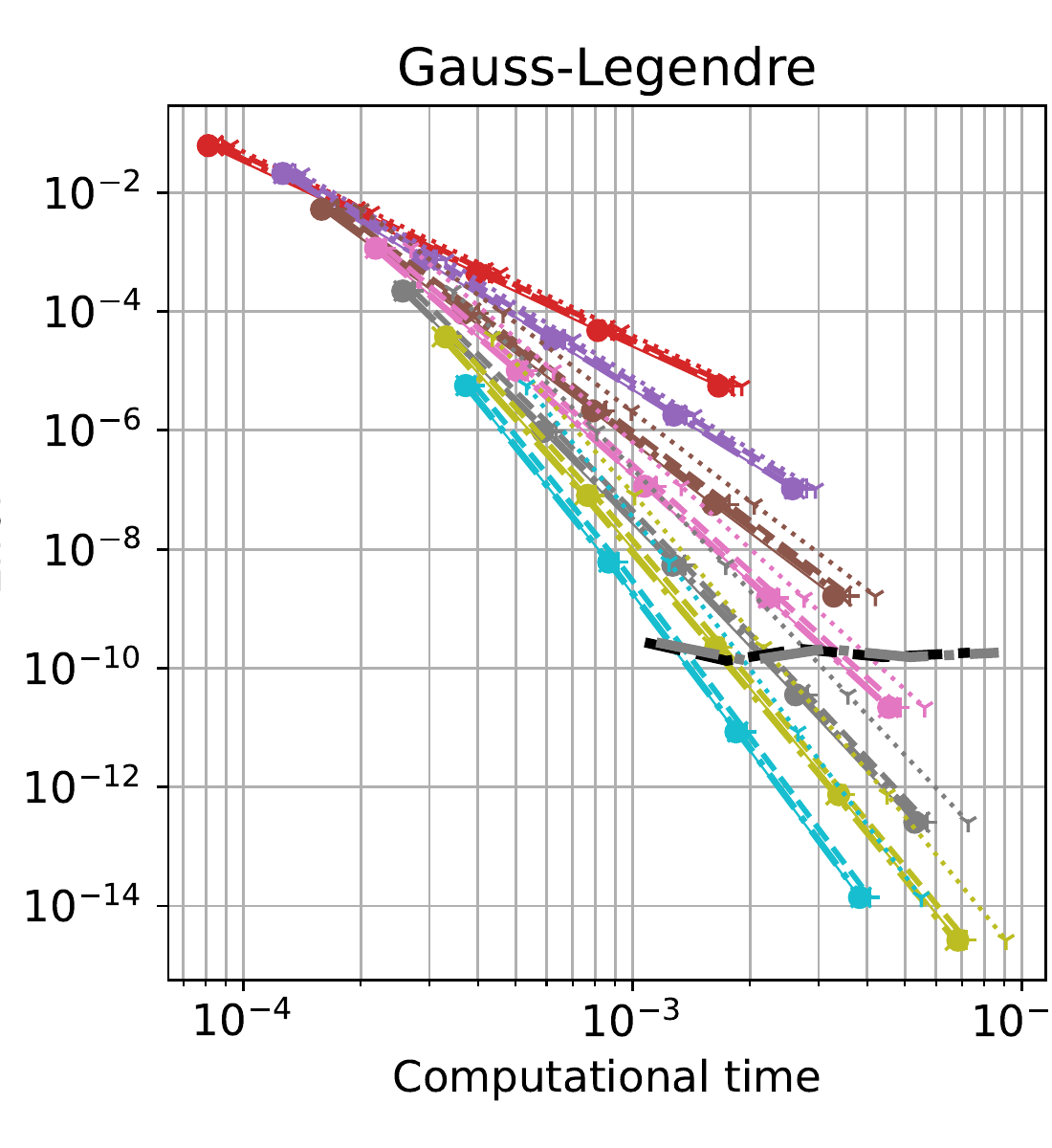}
	\includegraphics[width=0.12\textwidth, trim={343 30 0 40}, clip]	{figures/adaptive/convergence_vs_time_ADER_wc_adaptive_gaussLegendre_staggered_\test.pdf}
	\caption{\testname: Error with respect to computational time.
	}
	\label{fig:conv_time_linear_system}
\end{figure}

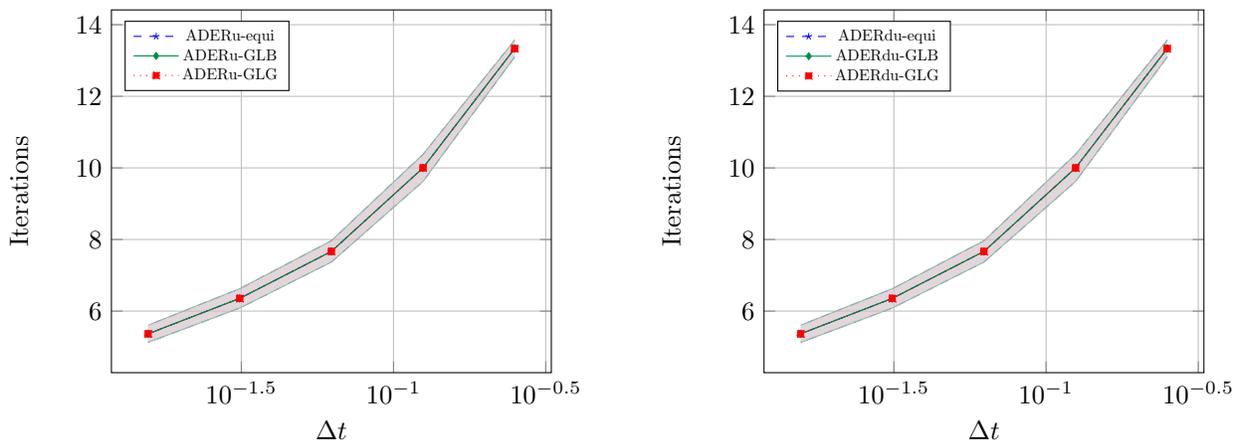
\begin{figure}
	\centering
	\begin{tikzpicture}
		\begin{axis}[
			xmode=log,
			grid=major,
			xlabel={$\dt$},
			ylabel={Iterations},
			legend pos=north west,
			legend style={nodes={scale=0.6, transform shape}},
			width=.45\textwidth
			]					
			\newcommand{\method}{ADER_u_equispaced}
			\addplot[mark=star,dashed, mark size=1.3pt,blue] table [x=dt, y=p average\method, col sep=comma] {figures/adaptive/convergence_ADER_adaptive_equispaced_\test.csv};
			\addlegendentry{ADERu-equi};

			\renewcommand{\method}{ADER_u_gaussLobatto}
			\addplot[mark=diamond*,solid,mark size=1.3pt,darkspringgreen] table [x=dt, y=p average\method, col sep=comma] {figures/adaptive/convergence_ADER_adaptive_gaussLobatto_\test.csv};
			\addlegendentry{ADERu-GLB};

			\renewcommand{\method}{ADER_u_gaussLegendre}
			\addplot[mark=square*,dotted,mark size=1.3pt,red] table [x=dt, y=p average\method, col sep=comma] {figures/adaptive/convergence_ADER_adaptive_gaussLegendre_\test.csv};
			\addlegendentry{ADERu-GLG};

			\renewcommand{\method}{ADER_u_equispaced}
			\addplot[name path=us_top,dashed,blue!50]   table [x=dt, y expr=\thisrow{p average\method}+0.5*\thisrow{p std\method}, col sep=comma  ]{figures/adaptive/convergence_ADER_adaptive_equispaced_\test.csv};
			\addplot[name path=us_bot,dashed,blue!50]   table [x=dt, y expr=\thisrow{p average\method}-0.5*\thisrow{p std\method}, col sep=comma  ]{figures/adaptive/convergence_ADER_adaptive_equispaced_\test.csv};	
			\addplot[blue!30,fill opacity=0.3] fill between[of=us_top and us_bot];	
			
			\renewcommand{\method}{ADER_u_gaussLobatto}
			\addplot[mark=diamond*,solid,mark size=1.3pt,darkspringgreen] table [x=dt, y=p average\method, col sep=comma] {figures/adaptive/convergence_ADER_adaptive_gaussLobatto_\test.csv};
			\addplot[name path=us_top,solid,darkspringgreen!50]   table [x=dt, y expr=\thisrow{p average\method}+0.5*\thisrow{p std\method}, col sep=comma  ]{figures/adaptive/convergence_ADER_adaptive_gaussLobatto_\test.csv};
			\addplot[name path=us_bot,solid,darkspringgreen!50]   table [x=dt, y expr=\thisrow{p average\method}-0.5*\thisrow{p std\method}, col sep=comma  ]{figures/adaptive/convergence_ADER_adaptive_gaussLobatto_\test.csv};
			\addplot[darkspringgreen!30,fill opacity=0.3] fill between[of=us_top and us_bot];

			\renewcommand{\method}{ADER_u_gaussLegendre}
			\addplot[mark=square*,dotted,mark size=1.3pt,red] table [x=dt, y=p average\method, col sep=comma] {figures/adaptive/convergence_ADER_adaptive_gaussLegendre_\test.csv};
			\addplot[name path=us_top,dotted,red!50]   table [x=dt, y expr=\thisrow{p average\method}+0.5*\thisrow{p std\method}, col sep=comma  ]{figures/adaptive/convergence_ADER_adaptive_gaussLegendre_\test.csv};
			\addplot[name path=us_bot,dotted,red!50]   table [x=dt, y expr=\thisrow{p average\method}-0.5*\thisrow{p std\method}, col sep=comma  ]{figures/adaptive/convergence_ADER_adaptive_gaussLegendre_\test.csv};
			\addplot[red!30,fill opacity=0.3] fill between[of=us_top and us_bot];	
			
		\end{axis}
	\end{tikzpicture}
	\hfill
	\begin{tikzpicture}
		\begin{axis}[
			xmode=log,
			grid=major,
			xlabel={$\dt$},
			ylabel={Iterations},
			legend pos=north west,
			legend style={nodes={scale=0.6, transform shape}},
			width=.45\textwidth
			]					
			\newcommand{\method}{ADER_L2_equispaced}
			\addplot[mark=star,dashed, mark size=1.3pt,blue] table [x=dt, y=p average\method, col sep=comma] {figures/adaptive/convergence_ADER_adaptive_equispaced_\test.csv};
			\addlegendentry{ADERdu-equi};

			\renewcommand{\method}{ADER_L2_gaussLobatto}
			\addplot[mark=diamond*,solid,mark size=1.3pt,darkspringgreen] table [x=dt, y=p average\method, col sep=comma] {figures/adaptive/convergence_ADER_adaptive_gaussLobatto_\test.csv};
			\addlegendentry{ADERdu-GLB};

			\renewcommand{\method}{ADER_L2_gaussLegendre}
			\addplot[mark=square*,dotted,mark size=1.3pt,red] table [x=dt, y=p average\method, col sep=comma] {figures/adaptive/convergence_ADER_adaptive_gaussLegendre_\test.csv};
			\addlegendentry{ADERdu-GLG};

			\renewcommand{\method}{ADER_L2_equispaced}
			\addplot[name path=us_top,dashed,blue!50]   table [x=dt, y expr=\thisrow{p average\method}+0.5*\thisrow{p std\method}, col sep=comma  ]{figures/adaptive/convergence_ADER_adaptive_equispaced_\test.csv};
			\addplot[name path=us_bot,dashed,blue!50]   table [x=dt, y expr=\thisrow{p average\method}-0.5*\thisrow{p std\method}, col sep=comma  ]{figures/adaptive/convergence_ADER_adaptive_equispaced_\test.csv};	
			\addplot[blue!30,fill opacity=0.3] fill between[of=us_top and us_bot];	
			
			\renewcommand{\method}{ADER_L2_gaussLobatto}
			\addplot[mark=diamond*,solid,mark size=1.3pt,darkspringgreen] table [x=dt, y=p average\method, col sep=comma] {figures/adaptive/convergence_ADER_adaptive_gaussLobatto_\test.csv};
			\addplot[name path=us_top,solid,darkspringgreen!50]   table [x=dt, y expr=\thisrow{p average\method}+0.5*\thisrow{p std\method}, col sep=comma  ]{figures/adaptive/convergence_ADER_adaptive_gaussLobatto_\test.csv};
			\addplot[name path=us_bot,solid,darkspringgreen!50]   table [x=dt, y expr=\thisrow{p average\method}-0.5*\thisrow{p std\method}, col sep=comma  ]{figures/adaptive/convergence_ADER_adaptive_gaussLobatto_\test.csv};
			\addplot[darkspringgreen!30,fill opacity=0.3] fill between[of=us_top and us_bot];

			\renewcommand{\method}{ADER_L2_gaussLegendre}
			\addplot[mark=square*,dotted,mark size=1.3pt,red] table [x=dt, y=p average\method, col sep=comma] {figures/adaptive/convergence_ADER_adaptive_gaussLegendre_\test.csv};
			\addplot[name path=us_top,dotted,red!50]   table [x=dt, y expr=\thisrow{p average\method}+0.5*\thisrow{p std\method}, col sep=comma  ]{figures/adaptive/convergence_ADER_adaptive_gaussLegendre_\test.csv};
			\addplot[name path=us_bot,dotted,red!50]   table [x=dt, y expr=\thisrow{p average\method}-0.5*\thisrow{p std\method}, col sep=comma  ]{figures/adaptive/convergence_ADER_adaptive_gaussLegendre_\test.csv};
			\addplot[red!30,fill opacity=0.3] fill between[of=us_top and us_bot];	
			
		\end{axis}
	\end{tikzpicture}
	\caption{\testname: Average number of iterations ($\pm$ half standard deviation) of adaptive \ADERu~(left) and \ADERdu~(right) for different time steps. \label{fig:p_adaptive_linear}}
\end{figure}

\subsubsection{C5 problem}
\renewcommand{\test}{C5}
\renewcommand{\testname}{C5}
The next test we study is the C5 nonstiff problem proposed in \cite{enright1987two}. It consists of a five body problem in three dimensions and it represents the description of the five outer planets (including Pluto) around the solar system. Each body has three coordinates $y_{1j},y_{2j},y_{3j}$ and each coordinate satisfies
\begin{equation}
	y_{ij}'' = k_2 \left(\frac{-(m_0+m_j)y_{ij}}{r_j^3} + \sum_{k\neq j} m_k\left[\frac{y_{ik}-y_{ij}}{d_{ik}^3}-\frac{y_{ik}}{r_k^3}\right]\right), \quad r_j^2=\sum_{i=1}^3 y_{ij}^2,\quad d_{kj}=\sum_{i=1}^3(y_{ik}-y_{ij})^2.
\end{equation}
All mass coefficients, gravitational constant, final time and initial conditions can be found in \cite{enright1987two}. This second order systems can be rewritten in terms of coordinates and velocities into a 30 equations system.

We are interested in this problem as the number of equations guarantees that the leading computational cost of the methods is directly proportional to the number of right-hand side evaluations. 
We use, as a reference, the solution obtained with ADER GLG of order 9 with 256 timesteps. All the errors are computed with respect to this reference solution.

\begin{figure}
	\newcommand{\quadrature}{equispaced}
	\includegraphics[width=0.28\textwidth, trim={24 0 120 0}, clip]	{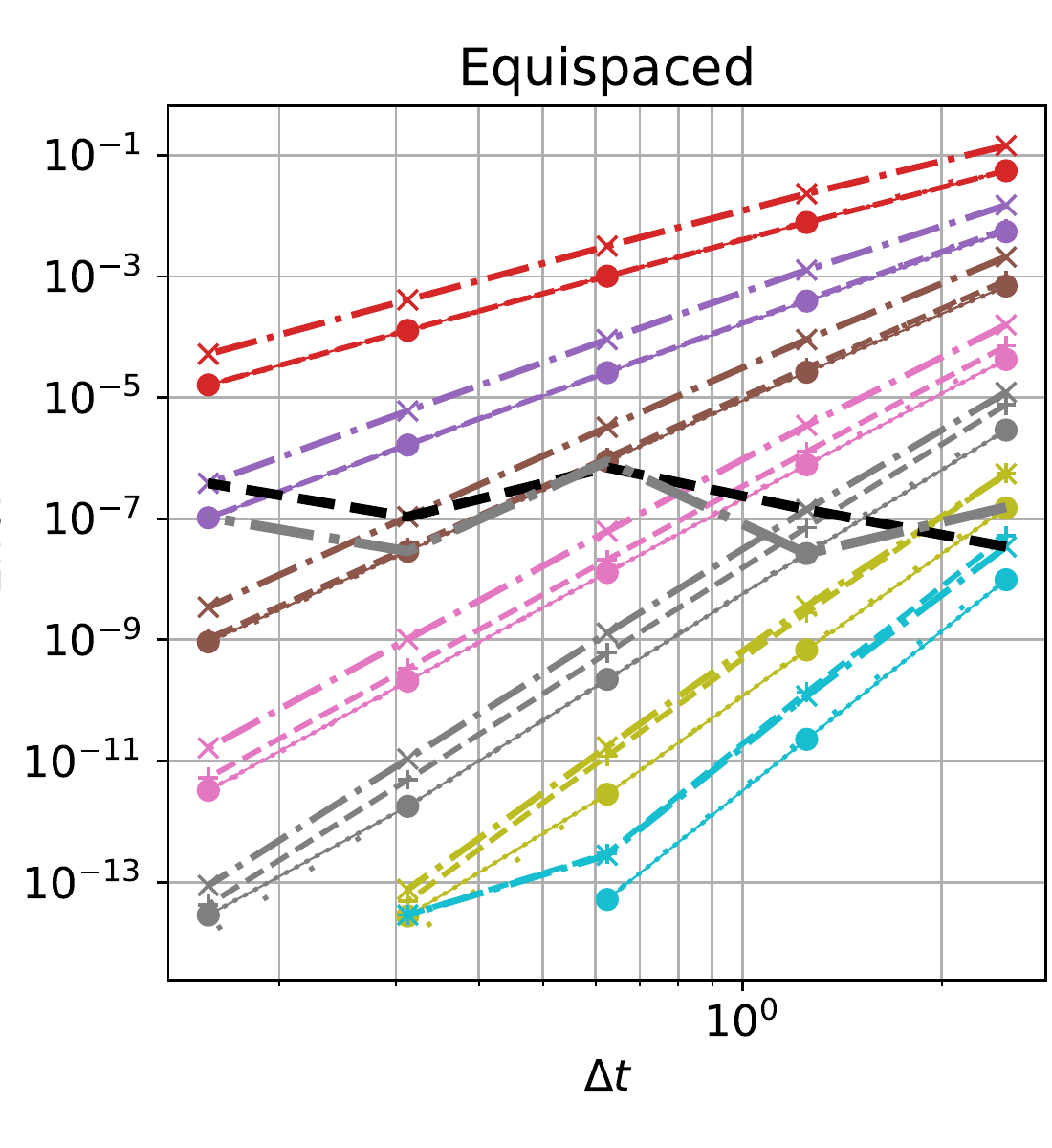}
	\renewcommand{\quadrature}{gaussLobatto}
	\includegraphics[width=0.28\textwidth, trim={24 0 120 0}, clip]{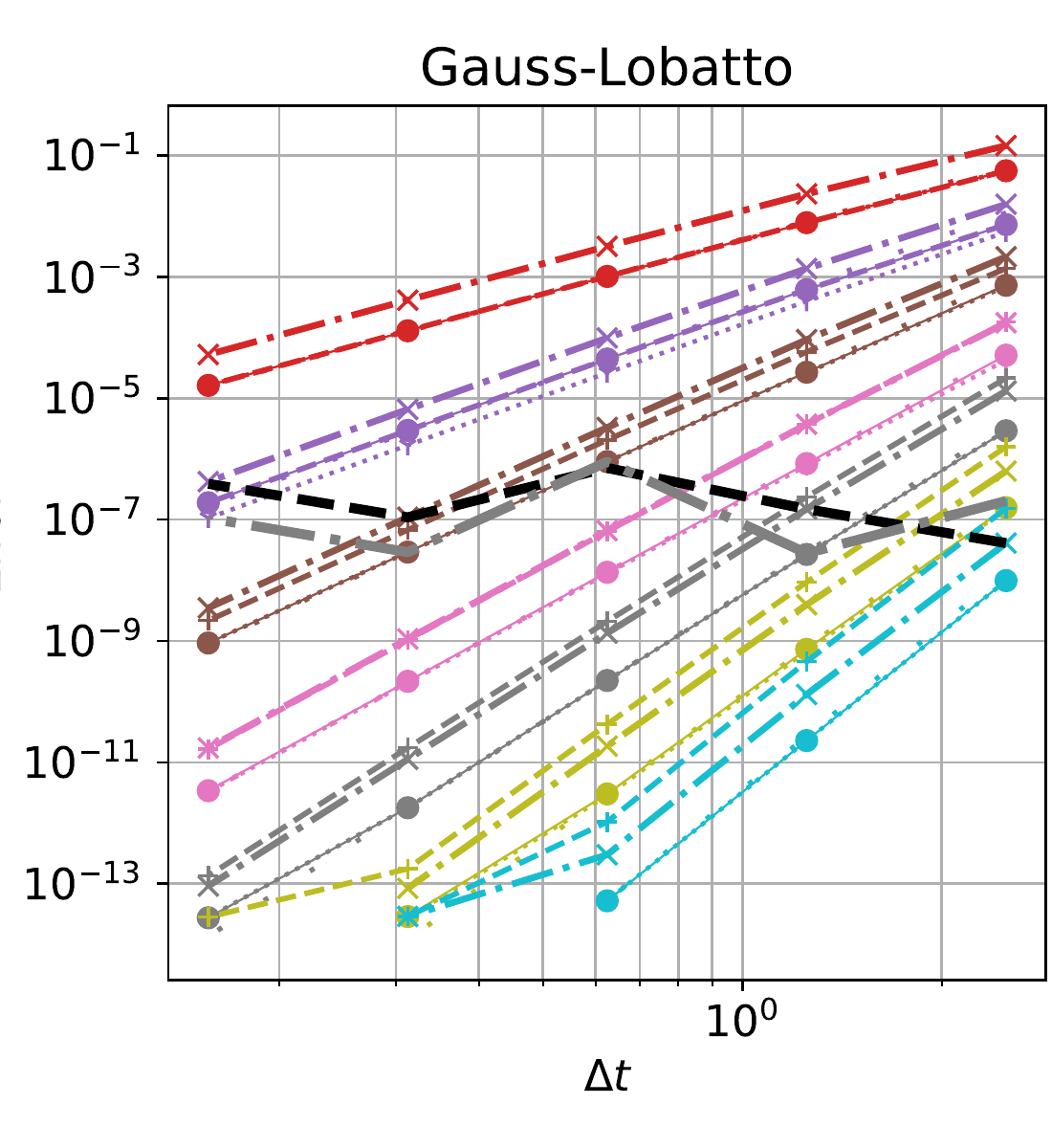}
	\renewcommand{\quadrature}{gaussLegendre}
	\includegraphics[width=0.28\textwidth, trim={24 0 120 0}, clip]{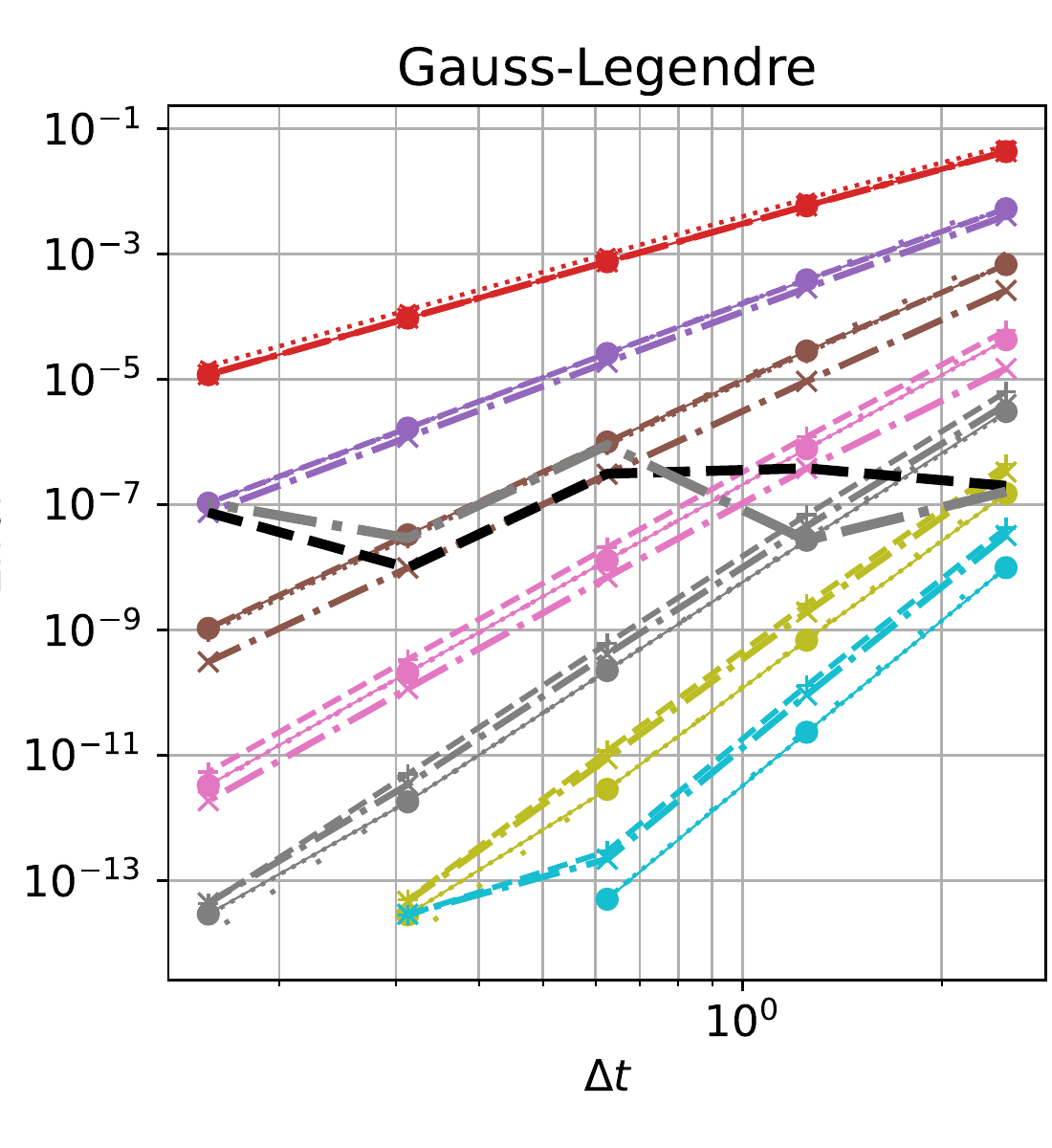}
	\includegraphics[width=0.12\textwidth, trim={340 20 0 40}, clip]{figures/adaptive/convergence_ADER_wc_adaptive_\quadrature_staggered_\test.pdf}
	\caption{\testname: Error decay for various methods and orders. The ``ref'' line is the reference for every order of accuracy
	}
	\label{fig:convergence_C5}
\end{figure}

In Figure~\ref{fig:convergence_C5}, we observe that all schemes converge with the expected order of accuracy, even if, differently from the previous linear test, they do not have the same errors. For equispaced and GLB subtimenodes, we have that the smaller errors are almost always achieved by ADER and \cADER~methods, while for GLG nodes different orders have different behaviors. For example, for order 5 and 6 \ADERdu~is the method with lowest errors. As for the previous test, the adaptive methods with a tolerance of $10^{-8}$ lead, more or less, to errors of the magnitude of the tolerance, independently of the mesh discretization for all subtimenodes.

\begin{figure}
	\newcommand{\quadrature}{equispaced}
	\includegraphics[width=0.28\textwidth, trim={24 0 120 0}, clip]	{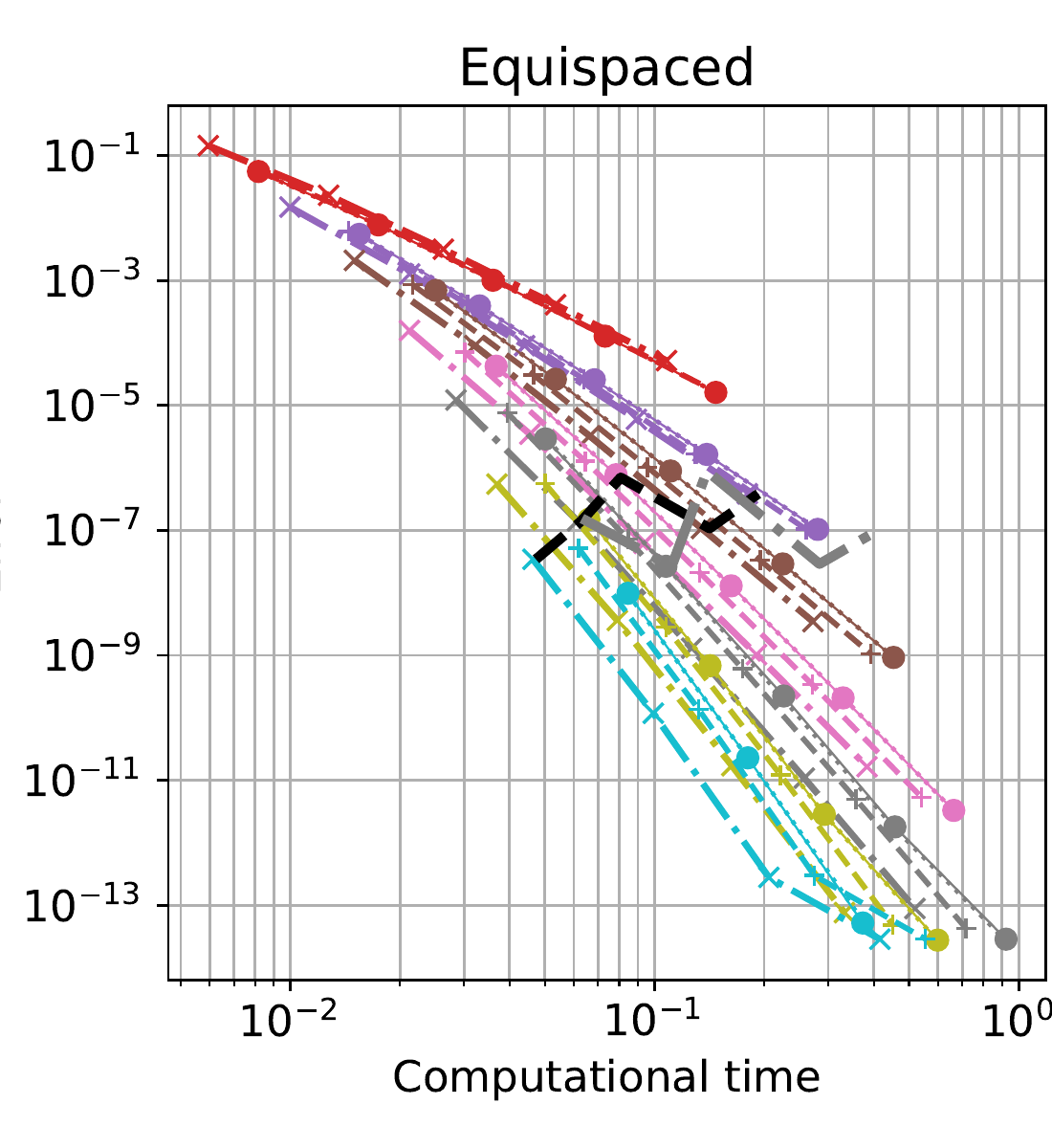}
	\renewcommand{\quadrature}{gaussLobatto}
	\includegraphics[width=0.28\textwidth, trim={24 0 120 0}, clip]{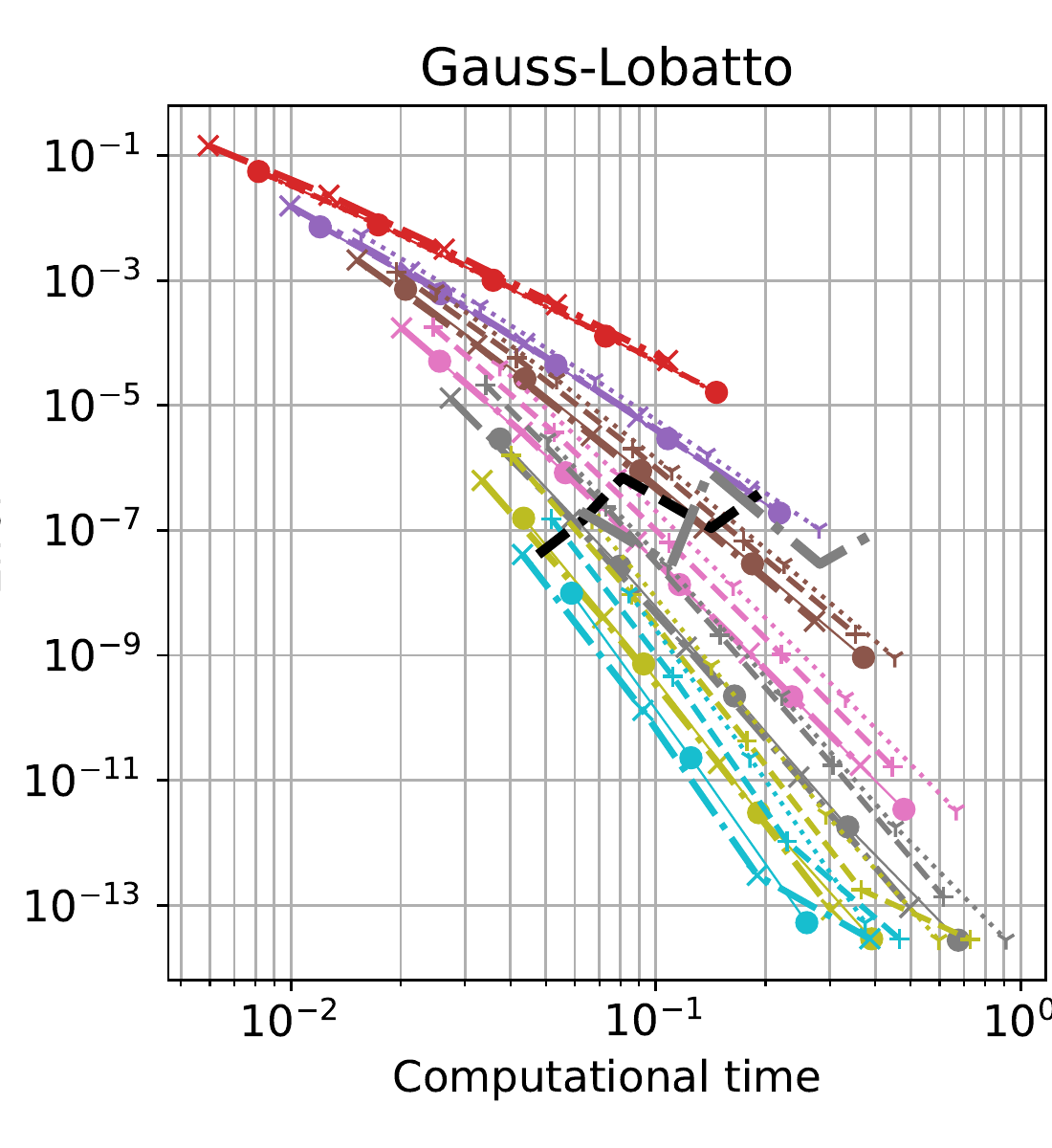}
	\renewcommand{\quadrature}{gaussLegendre}
	\includegraphics[width=0.28\textwidth, trim={24 0 120 0}, clip]{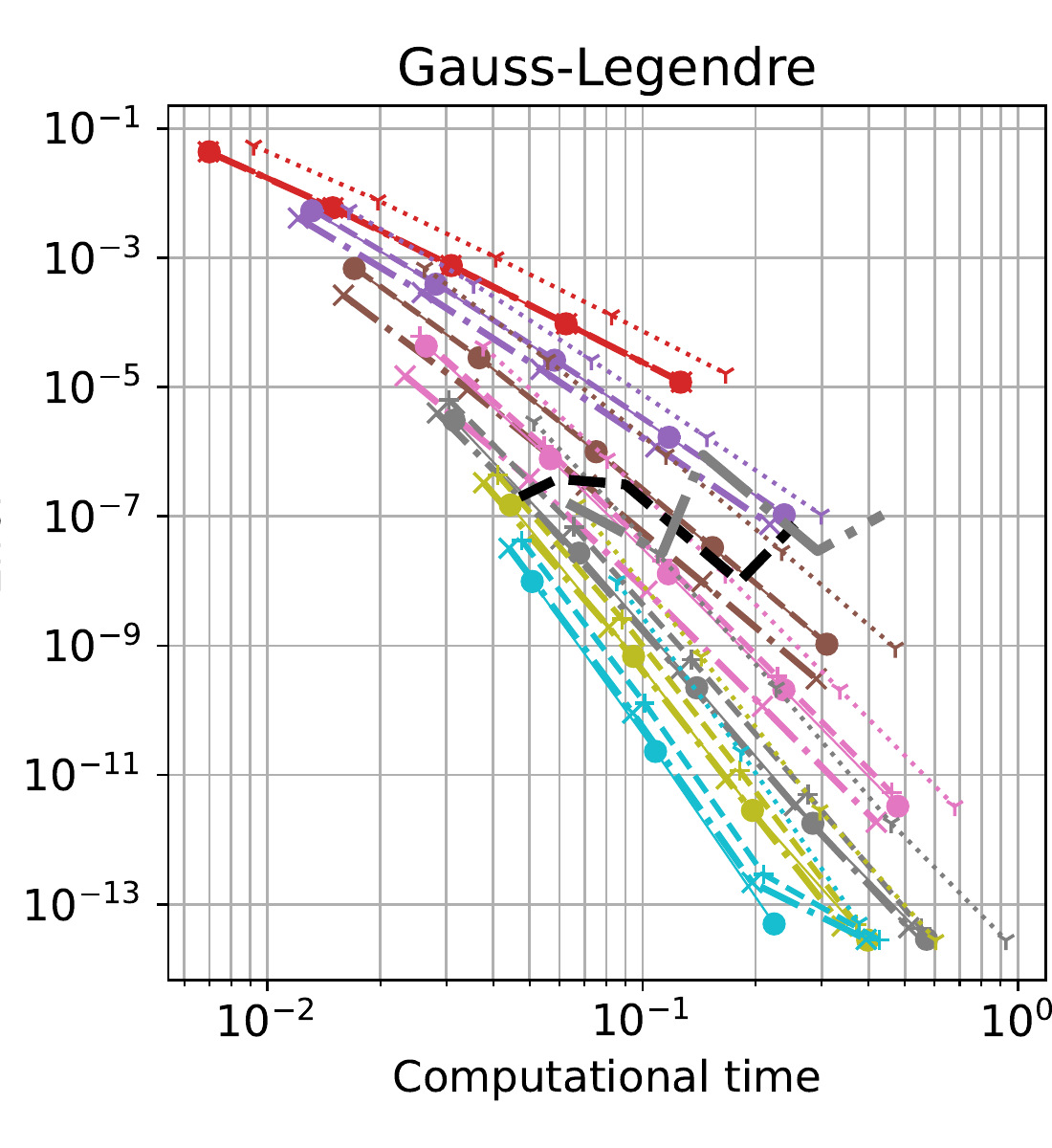}
	\includegraphics[width=0.12\textwidth, trim={342 30 0 40}, clip]	{figures/adaptive/convergence_vs_time_ADER_wc_adaptive_gaussLegendre_staggered_\test.pdf}
	\caption{\testname:  Error with respect to computational time.
	}
	\label{fig:conv_time_C5}
\end{figure}

\begin{figure}
	\centering
	\newcommand{\quadrature}{equispaced}
	\includegraphics[width=0.28\textwidth, trim={30 0 120 0}, clip]{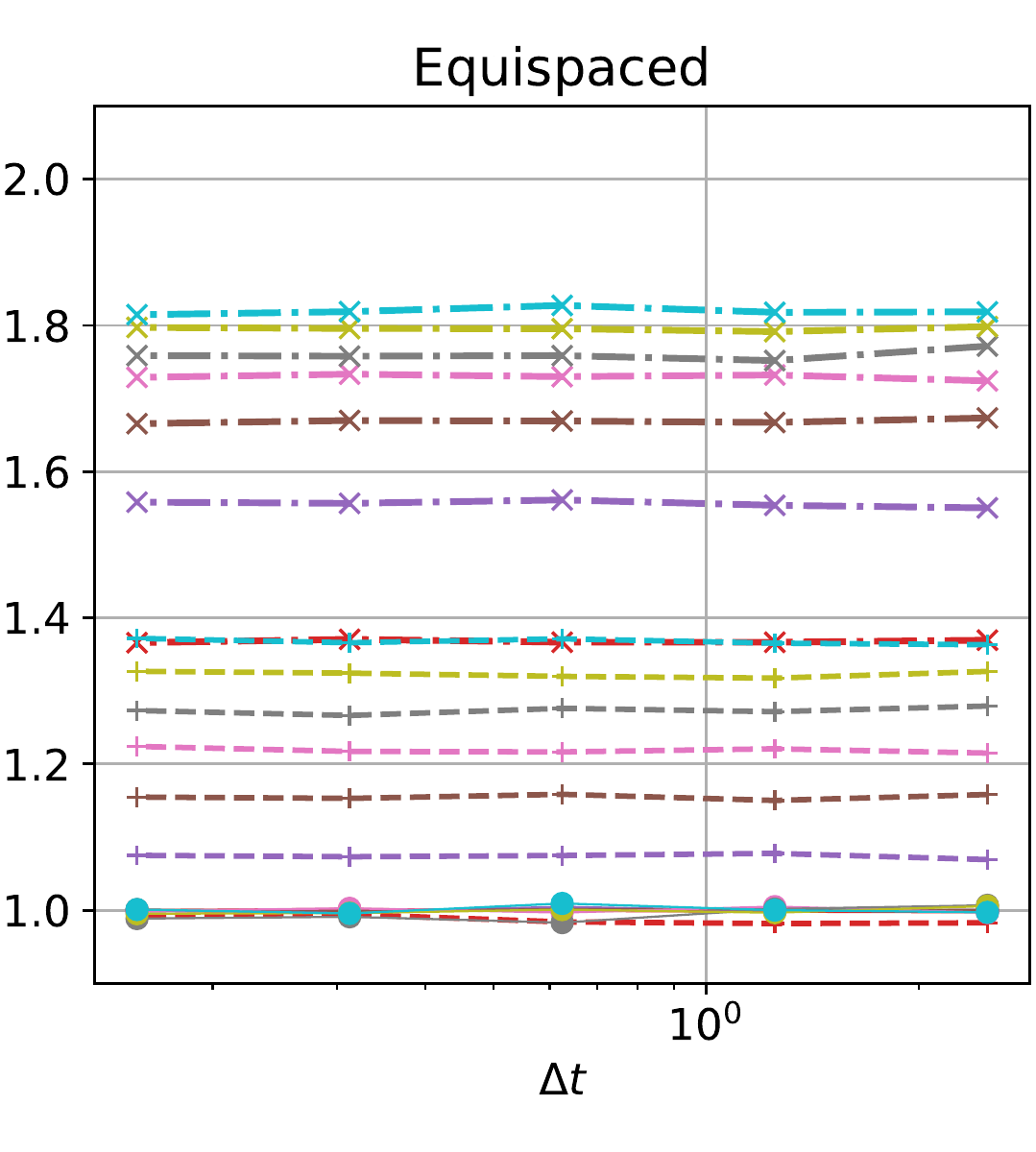}
	\renewcommand{\quadrature}{gaussLobatto}
	\includegraphics[width=0.28\textwidth, trim={30 0 120 0}, clip]{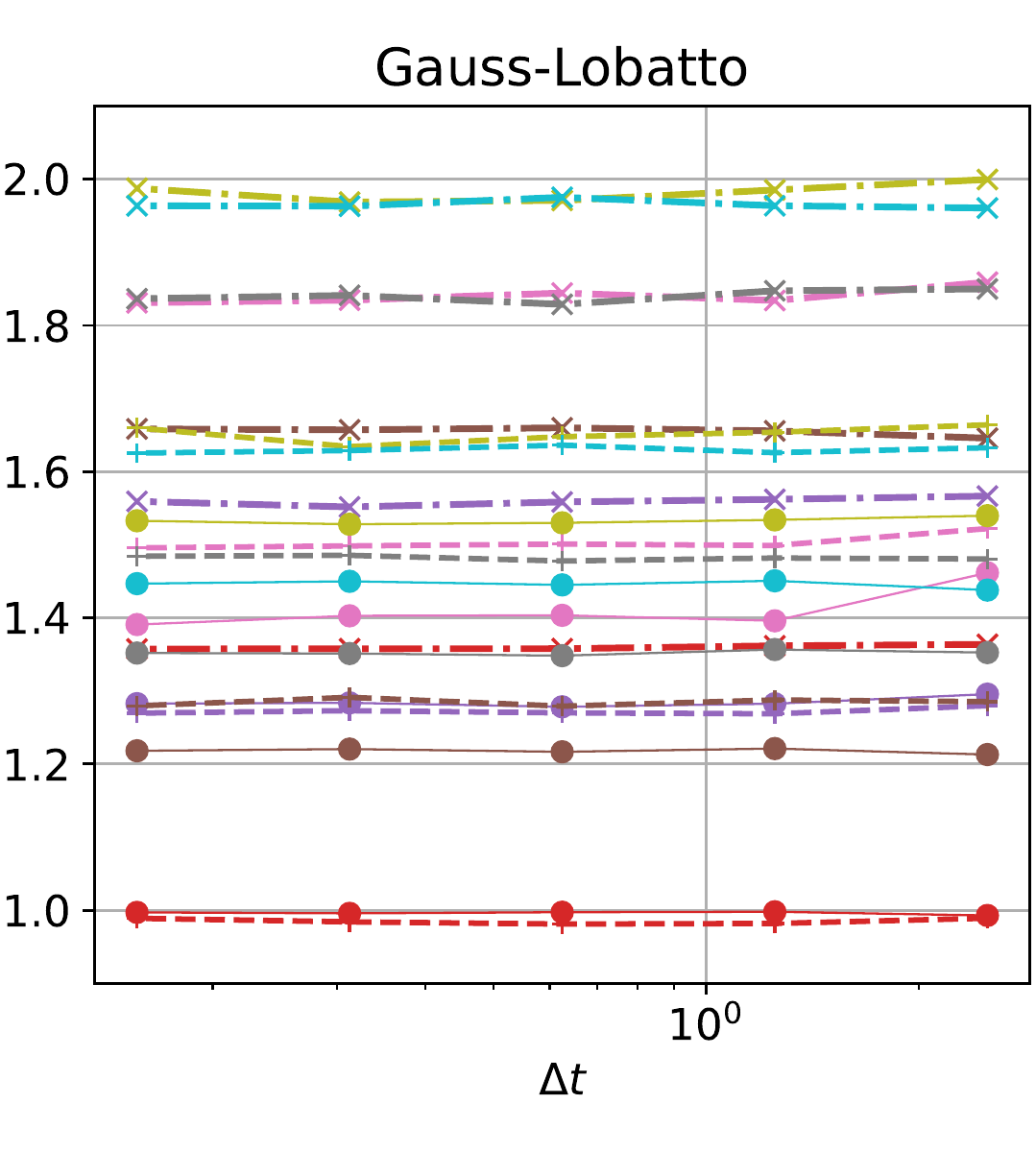}
	\renewcommand{\quadrature}{gaussLegendre}
	\includegraphics[width=0.28\textwidth, trim={30 0 120 0}, clip]{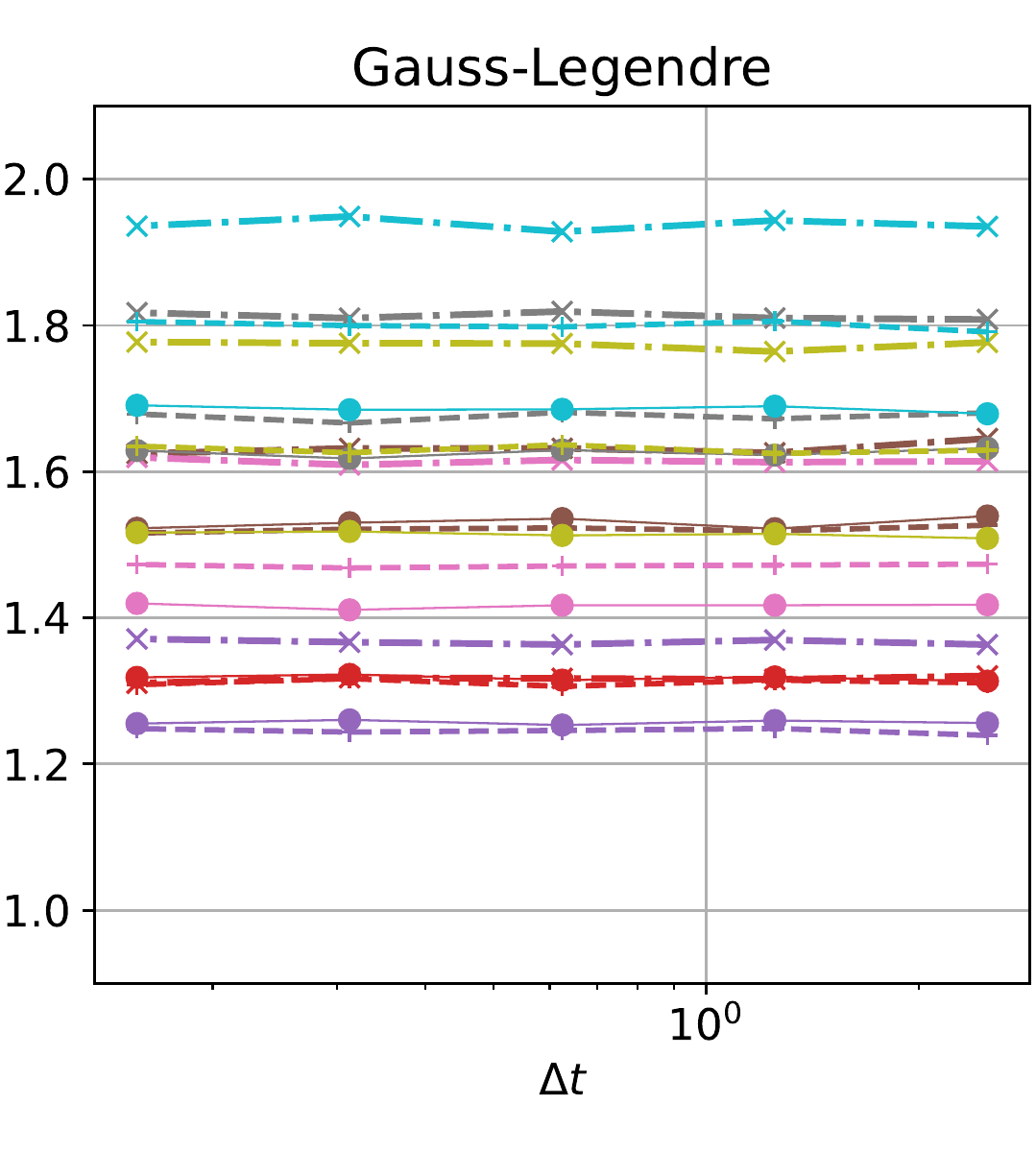}
	\includegraphics[width=0.12\textwidth, trim={343 30 0 40}, clip]{figures/adaptive/speed_up_ADER_wc_\quadrature_staggered_\test.pdf}
	\caption{{\testname} test: \textit{Numerical speed-up} 
		with respect to \cADER~method computed as the computational time of the \cADER~method over the computational time of method in consideration.
		\label{fig:speed_up_C5}}
\end{figure}

In terms of computational time, see Figure~\ref{fig:conv_time_C5}, we can observe that, most of the time, the \ADERdu~is the fastest scheme for comparable errors. 
For adaptive schemes, we observe again that the errors obtained with a fixed tolerance does not vary changing the time discretization, but the best computational times are obtained with coarse meshes. In Figure~\ref{fig:speed_up_C5}, we plot the \textit{numerical speed-up} of ADER, \ADERu~and \ADERdu~methods against \cADER. We recall that the \textit{numerical speed-up} is defined as the computational time of the reference method, in this case of \cADER, divided by the computational time of the method of interest, in our case ADER, \ADERu~and \ADERdu.
In this test problem, the simulation cost is mainly dictated by the number of right hand side evaluations, and this can be seen in Figure~\ref{fig:speed_up_C5} as the \textit{numerical speed-up} values do not vary much for different $\Delta t$. The \textit{theoretical speed-up} values can be found in Tables~\ref{tab:S_equispaced}, \ref{tab:S_GLB} and \ref{tab:S_GLG} under the \cADER-speed-up column. Comparing those values, we notice that the \textit{numerical speed-ups} are very close to the theoretical ones within a 10\% error maximum.

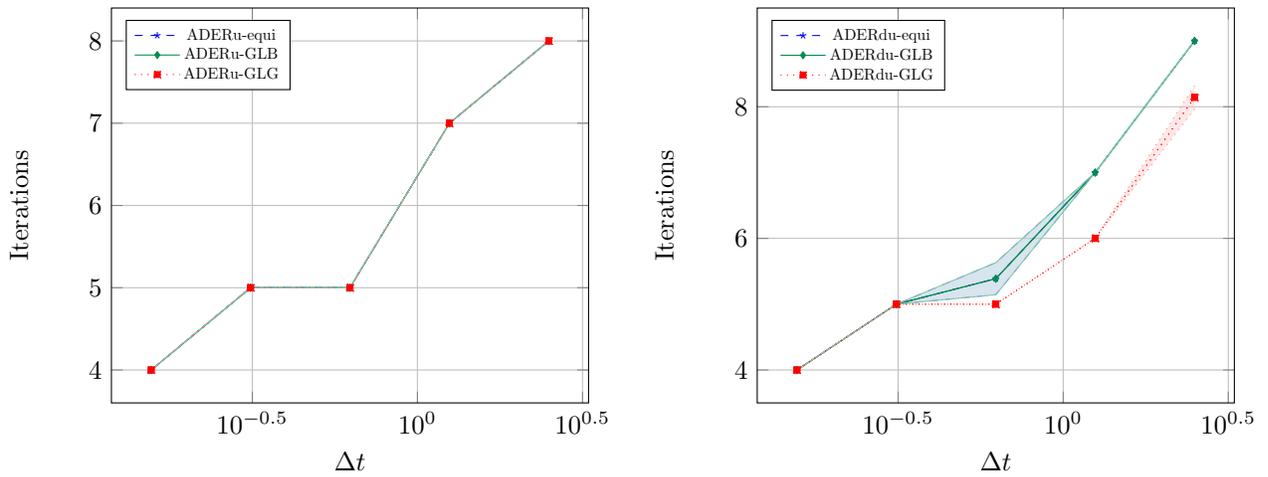
\begin{figure}
	\centering
	\begin{subfigure}{0.48\textwidth}
		\begin{tikzpicture}
			\begin{axis}[
				xmode=log,
				grid=major,
				xlabel={$\dt$},
				ylabel={Iterations},
				legend pos=north west,
				legend style={nodes={scale=0.6, transform shape}},
				width=\textwidth
				]					
				\newcommand{\method}{ADER_u_equispaced}
				\addplot[mark=star,dashed, mark size=1.3pt,blue] table [x=dt, y=p average\method, col sep=comma] {figures/adaptive/convergence_ADER_adaptive_equispaced_\test.csv};
				\addlegendentry{ADERu-equi};

				\renewcommand{\method}{ADER_u_gaussLobatto}
				\addplot[mark=diamond*,solid,mark size=1.3pt,darkspringgreen] table [x=dt, y=p average\method, col sep=comma] {figures/adaptive/convergence_ADER_adaptive_gaussLobatto_\test.csv};
				\addlegendentry{ADERu-GLB};

				\renewcommand{\method}{ADER_u_gaussLegendre}
				\addplot[mark=square*,dotted,mark size=1.3pt,red] table [x=dt, y=p average\method, col sep=comma] {figures/adaptive/convergence_ADER_adaptive_gaussLegendre_\test.csv};
				\addlegendentry{ADERu-GLG};

				\renewcommand{\method}{ADER_u_equispaced}
				\addplot[name path=us_top,dashed,blue!50]   table [x=dt, y expr=\thisrow{p average\method}+0.5*\thisrow{p std\method}, col sep=comma  ]{figures/adaptive/convergence_ADER_adaptive_equispaced_\test.csv};
				\addplot[name path=us_bot,dashed,blue!50]   table [x=dt, y expr=\thisrow{p average\method}-0.5*\thisrow{p std\method}, col sep=comma  ]{figures/adaptive/convergence_ADER_adaptive_equispaced_\test.csv};	
				\addplot[blue!30,fill opacity=0.3] fill between[of=us_top and us_bot];	
				
				\renewcommand{\method}{ADER_u_gaussLobatto}
				\addplot[mark=diamond*,solid,mark size=1.3pt,darkspringgreen] table [x=dt, y=p average\method, col sep=comma] {figures/adaptive/convergence_ADER_adaptive_gaussLobatto_\test.csv};
				\addplot[name path=us_top,solid,darkspringgreen!50]   table [x=dt, y expr=\thisrow{p average\method}+0.5*\thisrow{p std\method}, col sep=comma  ]{figures/adaptive/convergence_ADER_adaptive_gaussLobatto_\test.csv};
				\addplot[name path=us_bot,solid,darkspringgreen!50]   table [x=dt, y expr=\thisrow{p average\method}-0.5*\thisrow{p std\method}, col sep=comma  ]{figures/adaptive/convergence_ADER_adaptive_gaussLobatto_\test.csv};
				\addplot[darkspringgreen!30,fill opacity=0.3] fill between[of=us_top and us_bot];

				\renewcommand{\method}{ADER_u_gaussLegendre}
				\addplot[mark=square*,dotted,mark size=1.3pt,red] table [x=dt, y=p average\method, col sep=comma] {figures/adaptive/convergence_ADER_adaptive_gaussLegendre_\test.csv};
				\addplot[name path=us_top,dotted,red!50]   table [x=dt, y expr=\thisrow{p average\method}+0.5*\thisrow{p std\method}, col sep=comma  ]{figures/adaptive/convergence_ADER_adaptive_gaussLegendre_\test.csv};
				\addplot[name path=us_bot,dotted,red!50]   table [x=dt, y expr=\thisrow{p average\method}-0.5*\thisrow{p std\method}, col sep=comma  ]{figures/adaptive/convergence_ADER_adaptive_gaussLegendre_\test.csv};
				\addplot[red!30,fill opacity=0.3] fill between[of=us_top and us_bot];	
				
			\end{axis}
		\end{tikzpicture}
	\end{subfigure}
	\hfill
	\begin{subfigure}{0.48\textwidth}
		\begin{tikzpicture}
			\begin{axis}[
				xmode=log,
				grid=major,
				xlabel={$\dt$},
				ylabel={Iterations},
				legend pos=north west,
				legend style={nodes={scale=0.6, transform shape}},
				width=\textwidth
				]					
				\newcommand{\method}{ADER_L2_equispaced}
				\addplot[mark=star,dashed, mark size=1.3pt,blue] table [x=dt, y=p average\method, col sep=comma] {figures/adaptive/convergence_ADER_adaptive_equispaced_\test.csv};
				\addlegendentry{ADERdu-equi};

				\renewcommand{\method}{ADER_L2_gaussLobatto}
				\addplot[mark=diamond*,solid,mark size=1.3pt,darkspringgreen] table [x=dt, y=p average\method, col sep=comma] {figures/adaptive/convergence_ADER_adaptive_gaussLobatto_\test.csv};
				\addlegendentry{ADERdu-GLB};

				\renewcommand{\method}{ADER_L2_gaussLegendre}
				\addplot[mark=square*,dotted,mark size=1.3pt,red] table [x=dt, y=p average\method, col sep=comma] {figures/adaptive/convergence_ADER_adaptive_gaussLegendre_\test.csv};
				\addlegendentry{ADERdu-GLG};

				\renewcommand{\method}{ADER_L2_equispaced}
				\addplot[name path=us_top,dashed,blue!50]   table [x=dt, y expr=\thisrow{p average\method}+0.5*\thisrow{p std\method}, col sep=comma  ]{figures/adaptive/convergence_ADER_adaptive_equispaced_\test.csv};
				\addplot[name path=us_bot,dashed,blue!50]   table [x=dt, y expr=\thisrow{p average\method}-0.5*\thisrow{p std\method}, col sep=comma  ]{figures/adaptive/convergence_ADER_adaptive_equispaced_\test.csv};	
				\addplot[blue!30,fill opacity=0.3] fill between[of=us_top and us_bot];	
				
				\renewcommand{\method}{ADER_L2_gaussLobatto}
				\addplot[mark=diamond*,solid,mark size=1.3pt,darkspringgreen] table [x=dt, y=p average\method, col sep=comma] {figures/adaptive/convergence_ADER_adaptive_gaussLobatto_\test.csv};
				\addplot[name path=us_top,solid,darkspringgreen!50]   table [x=dt, y expr=\thisrow{p average\method}+0.5*\thisrow{p std\method}, col sep=comma  ]{figures/adaptive/convergence_ADER_adaptive_gaussLobatto_\test.csv};
				\addplot[name path=us_bot,solid,darkspringgreen!50]   table [x=dt, y expr=\thisrow{p average\method}-0.5*\thisrow{p std\method}, col sep=comma  ]{figures/adaptive/convergence_ADER_adaptive_gaussLobatto_\test.csv};
				\addplot[darkspringgreen!30,fill opacity=0.3] fill between[of=us_top and us_bot];

				\renewcommand{\method}{ADER_L2_gaussLegendre}
				\addplot[mark=square*,dotted,mark size=1.3pt,red] table [x=dt, y=p average\method, col sep=comma] {figures/adaptive/convergence_ADER_adaptive_gaussLegendre_\test.csv};
				\addplot[name path=us_top,dotted,red!50]   table [x=dt, y expr=\thisrow{p average\method}+0.5*\thisrow{p std\method}, col sep=comma  ]{figures/adaptive/convergence_ADER_adaptive_gaussLegendre_\test.csv};
				\addplot[name path=us_bot,dotted,red!50]   table [x=dt, y expr=\thisrow{p average\method}-0.5*\thisrow{p std\method}, col sep=comma  ]{figures/adaptive/convergence_ADER_adaptive_gaussLegendre_\test.csv};
				\addplot[red!30,fill opacity=0.3] fill between[of=us_top and us_bot];	
				
			\end{axis}
		\end{tikzpicture}
	\end{subfigure}
	\caption{{\testname} test: Average number of iterations ($\pm$ half standard deviation) of adaptive \ADERu~(left) and \ADERdu~(right) for different time steps. \label{fig:p_adaptive_C5}}
\end{figure}
Finally, for the adaptive methods we observe in Figure~\ref{fig:p_adaptive_C5} that the chosen order for a given simulation is quite stable along the simulations, as we see almost no variance in all plots, and it scales very well changing the time discretization scale.

\begin{table}
	\centering
	\footnotesize
	\begin{tabular}{|c|c||c|c|c||c|c|c|}
		\hline
		\multicolumn{2}{|c||}{Speed-up}& \multicolumn{3}{c||}{ADERu vs ADER}& \multicolumn{3}{c|}{ADERdu vs ADER}\\ \hline \hline 
		$P$ & $M$ & Th. & LS & C5 & Th. & LS & C5\\\hline 
		3 & 2 & 1 & 0.956 & 0.993 & 1.5 & 1.069 & 1.366 \\ \hline 
		4 & 3 & 1.091 & 0.975 & 1.076 & 1.714 & 1.132 & 1.561 \\ \hline 
		5 & 4 & 1.176 & 1.013 & 1.154 & 1.818 & 1.185 & 1.664 \\ \hline 
		6 & 5 & 1.25 & 1.049 & 1.225 & 1.875 & 1.247 & 1.731 \\ \hline 
		7 & 6 & 1.312 & 1.088 & 1.288 & 1.909 & 1.279 & 1.779 \\ \hline 
		8 & 7 & 1.366 & 1.124 & 1.333 & 1.931 & 1.321 & 1.806 \\ \hline 
		9 & 8 & 1.431 & 1.163 & 1.371 & 1.973 & 1.359 & 1.813 \\ \hline 
	\end{tabular}
	
	\caption{Comparison between the \textit{theoretical speed-ups} and the numerical ones obtained for linear system (LS) test and the C5 problem. Equispaced subtimenodes}	\label{tab:speed_up_equi}
\end{table}

\begin{table}
	\centering
	\resizebox{\columnwidth}{!}{
		\begin{tabular}{|c|c||c|c|c||c|c|c||c|c|c||c|c|c||c|c|c|}
			\hline
			\multicolumn{2}{|c||}{Speed-up}& \multicolumn{3}{c||}{ADER vs \cADER}& \multicolumn{3}{c||}{ADERu vs \cADER}& \multicolumn{3}{c||}{ADERdu vs \cADER}& \multicolumn{3}{c||}{ADERu vs ADER}& \multicolumn{3}{c|}{ADERdu vs ADER}\\ \hline \hline 
			$P$ & $M$ & Th. & LS & C5 & Th. & LS & C5& Th. & LS & C5 & Th. & LS & C5& Th. & LS & C5\\\hline 
			3 & 2 & 1 & 0.997 & 0.997 & 1 & 0.951 & 0.989 & 1.5 & 1.066 & 1.357 & 1 & 0.954 & 0.992 & 1.5 & 1.070 & 1.361 \\ \hline 
			4 & 2 & 1.333 & 1.141 & 1.283 & 1.333 & 1.052 & 1.270 & 1.714 & 1.133 & 1.559 & 1 & 0.922 & 0.990 & 1.286 & 0.994 & 1.216 \\ \hline 
			5 & 3 & 1.25 & 1.130 & 1.218 & 1.333 & 1.076 & 1.280 & 1.818 & 1.192 & 1.659 & 1.067 & 0.952 & 1.050 & 1.455 & 1.055 & 1.362 \\ \hline 
			6 & 3 & 1.5 & 1.250 & 1.391 & 1.579 & 1.177 & 1.496 & 2 & 1.281 & 1.831 & 1.053 & 0.942 & 1.075 & 1.333 & 1.024 & 1.317 \\ \hline 
			7 & 4 & 1.433 & 1.220 & 1.352 & 1.593 & 1.187 & 1.484 & 2.048 & 1.300 & 1.837 & 1.111 & 0.973 & 1.098 & 1.429 & 1.066 & 1.359 \\ \hline 
			8 & 4 & 1.6 & 1.328 & 1.533 & 1.75 & 1.277 & 1.660 & 2.154 & 1.398 & 1.988 & 1.094 & 0.962 & 1.083 & 1.346 & 1.053 & 1.297 \\ \hline 
			9 & 5 & 1.521 & 1.294 & 1.447 & 1.738 & 1.290 & 1.625 & 2.147 & 1.426 & 1.964 & 1.143 & 0.998 & 1.123 & 1.412 & 1.102 & 1.357 \\ \hline 
		\end{tabular}
	}
	\caption{Comparison between the \textit{theoretical speed-ups} and the numerical ones obtained for linear system (LS) test and the C5 problem. GLB subtimenodes}	\label{tab:speed_up_GLB}
\end{table}

\begin{table}
	\centering
	\resizebox{\columnwidth}{!}{
		\begin{tabular}{|c|c||c|c|c||c|c|c||c|c|c||c|c|c||c|c|c|}
			\hline
			\multicolumn{2}{|c||}{Speed-up}& \multicolumn{3}{c||}{ADER vs \cADER}& \multicolumn{3}{c||}{ADERu vs \cADER}& \multicolumn{3}{c||}{ADERdu vs \cADER}& \multicolumn{3}{c||}{ADERu vs ADER}& \multicolumn{3}{c|}{ADERdu vs ADER}\\ \hline \hline 
			$P$ & $M$ & Th. & LS & C5 & Th. & LS & C5& Th. & LS & C5 & Th. & LS & C5& Th. & LS & C5\\\hline 
			3 & 1 & 1.4 & 1.146 & 1.319 & 1.4 & 1.056 & 1.309 & 1.4 & 1.104 & 1.311 & 1 & 0.922 & 0.992 & 1 & 0.963 & 0.994 \\ \hline 
			4 & 2 & 1.3 & 1.146 & 1.255 & 1.3 & 1.053 & 1.248 & 1.444 & 1.131 & 1.371 & 1 & 0.919 & 0.995 & 1.111 & 0.987 & 1.092 \\ \hline 
			5 & 2 & 1.615 & 1.280 & 1.523 & 1.615 & 1.164 & 1.516 & 1.75 & 1.222 & 1.625 & 1 & 0.910 & 0.996 & 1.083 & 0.955 & 1.067 \\ \hline 
			6 & 3 & 1.476 & 1.238 & 1.420 & 1.55 & 1.158 & 1.473 & 1.722 & 1.240 & 1.620 & 1.05 & 0.936 & 1.038 & 1.167 & 1.002 & 1.141 \\ \hline 
			7 & 3 & 1.72 & 1.372 & 1.629 & 1.792 & 1.275 & 1.679 & 1.955 & 1.333 & 1.818 & 1.042 & 0.929 & 1.031 & 1.136 & 0.971 & 1.116 \\ \hline 
			8 & 4 & 1.583 & 1.330 & 1.517 & 1.727 & 1.266 & 1.635 & 1.9 & 1.350 & 1.777 & 1.091 & 0.952 & 1.078 & 1.2 & 1.016 & 1.172 \\ \hline 
			9 & 4 & 1.78 & 1.432 & 1.691 & 1.921 & 1.360 & 1.805 & 2.086 & 1.433 & 1.936 & 1.079 & 0.949 & 1.068 & 1.171 & 1.000 & 1.145 \\ \hline 
		\end{tabular}
	}
	\caption{Comparison between the \textit{theoretical speed-ups} and the numerical ones obtained for linear system (LS) test and the C5 problem. GLG subtimenodes}\label{tab:speed_up_GLG}	
\end{table}

In Tables~\ref{tab:speed_up_equi}, \ref{tab:speed_up_GLB} and \ref{tab:speed_up_GLG}, we can observe the \textit{numerical speed-ups} compared with the theoretical ones for equispaced, GLB and GLG subtimenodes respectively. We recall that the \textit{theoretical speed-up} is the ratio between the number of RK stages of the reference method over the number of RK stages of the method of interest. On the other hand, the \textit{numerical speed-up} is computed as the ratio between the computational times of such methods.
The two measures are quite comparable as, in particular for complicated problems, the cost of the evaluation of $\uvec{G}$ is the most expensive operation of the method, modulo parallelization and vectorization.
For the ODE tests, our Python implementation, based on \texttt{numpy}, performs very similarly to the expected theoretical predictions.
Looking at the tables, we observe that in the context of problem~C5, for which the complexity of the function $\uvec{G}$ is higher than for the linear system, the \textit{numerical speed-ups} coincide with the theoretical ones within an error which is at most 10\%. 
On the other hand, for the linear system, lower \textit{numerical speed-ups} are obtained, as the ODE is so simple that the cost of basic operations is not negligible with respect to the cost of the computation of $\uvec{G}$. As already remarked, the linear system is not a suitable problem for verifying the computational advantages of the new methods and its only purpose is verifying some analytical results, yet computational advantages are registered also in this case.

\subsection{PDE tests}
In this section, we will apply the novel ADER methods to hyperbolic PDEs through SD spatial semidiscretization as described in Section~\ref{sec:SD}. In particular, the accuracy of the ADER methods will always be chosen equal to the spatial accuracy.  The timings reported in this section are CPU process times.

\subsubsection{Linear advection}

We consider the linear advection equation
\begin{equation}
	\label{eq:advection}
	\partial_t u + \partial_x (a u) = 0, 
\end{equation}
where $u(x,t)$ is a scalar wave,  advected by a nonzero constant $a \in \mathbb{R}$.  We consider the spatial domain $x\in[0,1]$, $a = 1$, initial condition $u_0(x) = \sin(2\pi x)$ with periodic boundary conditions. The analytical solution is given by $u(x,t)=u_0(x-at)$.  The CFL is chosen according to \eqref{eq:sd-cfl}, with $C=0.8$, and the final time is set to be $T=1$. 

\begin{figure}
	\includegraphics[width=1.0\textwidth]	{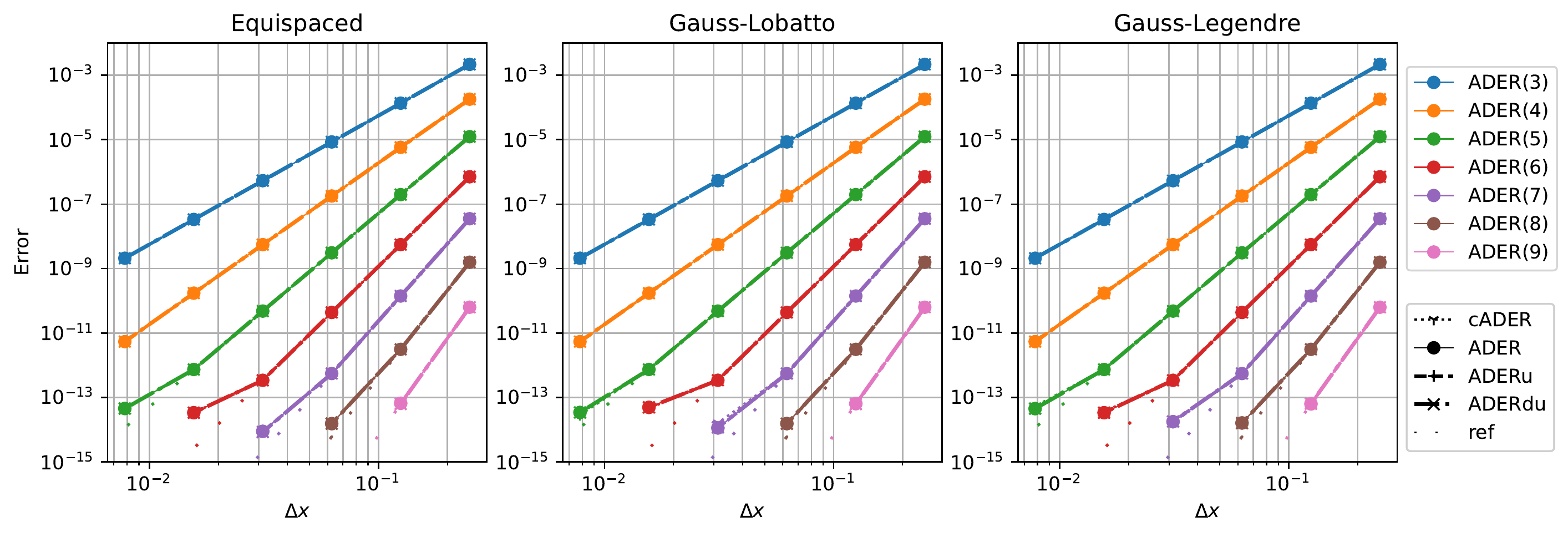}
	\caption{Linear advection: Error decay for various methods and orders. The ``ref'' line is the reference for every order of accuracy}
	\label{fig:error_res_advection}
\end{figure}

\begin{figure}
	\includegraphics[width=1.0\textwidth]	{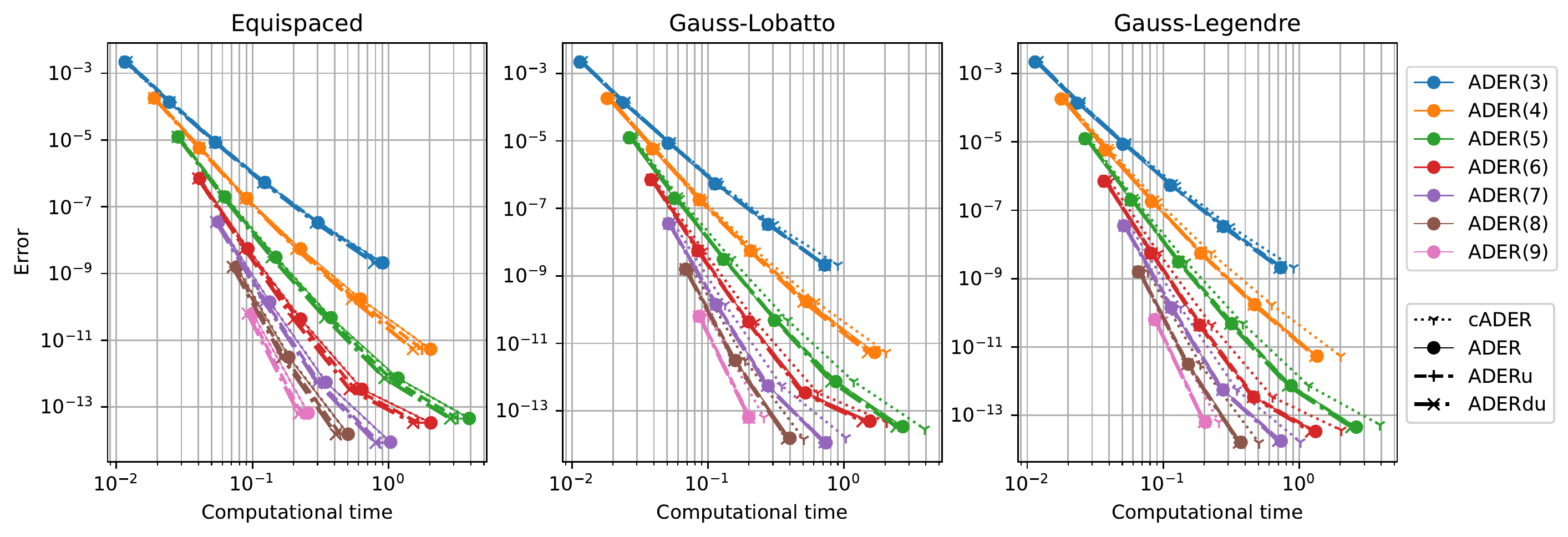}
	\caption{Linear advection: Error with respect to computational time}	
	\label{fig:error_time_advection}
\end{figure}
In Figure~\ref{fig:error_res_advection},  we show the convergence plots and observe that the theoretical convergence rate is attained. 
In Figure~\ref{fig:error_time_advection},  the error as a function of the computational time of the different ADER methods is shown. 
We can observe that the proposed modifications succeed in reducing the computational costs as $\Delta x$ decreases, being \ADERdu~and \ADERu~the best performing schemes, followed by ADER with optimal number of subtimenodes and, finally, by \cADER.
\begin{figure}
	\includegraphics[width=1.0\textwidth]	{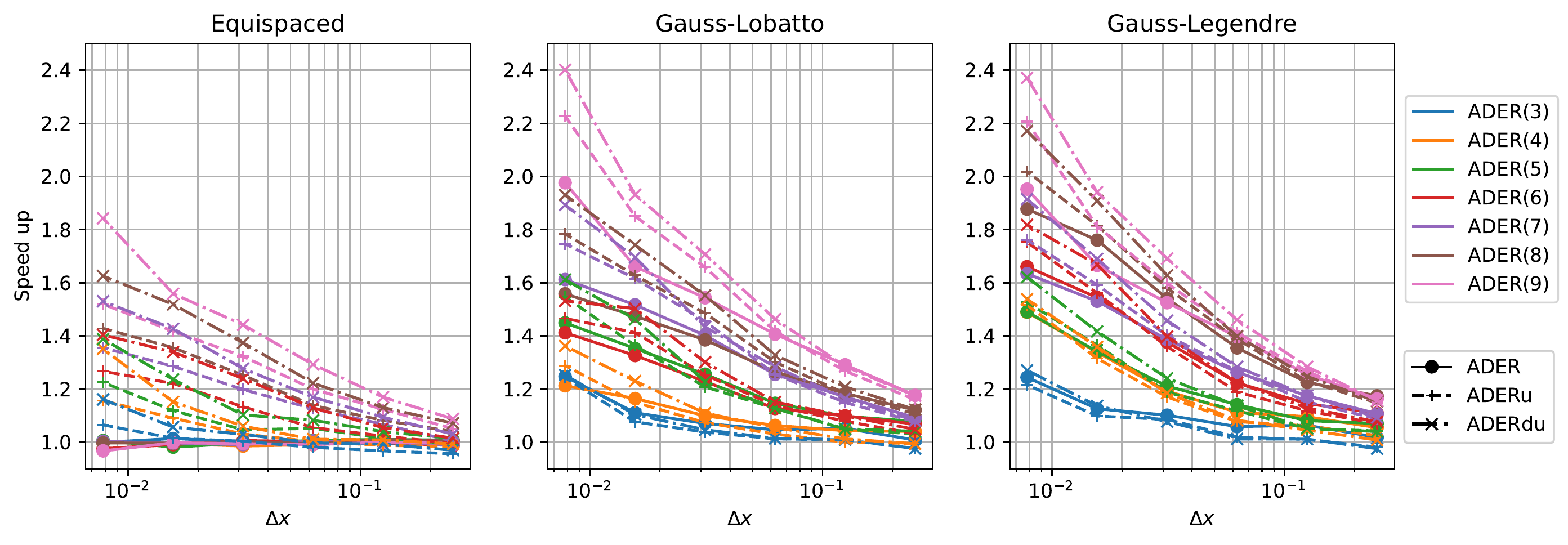}
	\caption{Linear advection: \textit{Numerical speed-up} factor with respect to cADER method computed as the computational time of the cADER method over the computational time of method in consideration}
	\label{fig:speed_up_advection}
\end{figure}
Lastly, in Figure~\ref{fig:speed_up_advection},  we show the \textit{numerical speed-ups} attained when using the reformulated ADER,  \ADERu~and \ADERdu~with respect to \cADER. In particular, we empirically observe \textit{numerical speed-ups} that vary a lot depending on the spatial discretization and that they increase as the spatial grid is refined and the polynomial order increases. 
The authors are not completely sure about the reason of this behavior, but they suspect that the implementation through Python packages and functions as \texttt{einsum} of \texttt{numpy} could be the source of this outcome. The phenomenon is currently under investigation.

However, despite the less homogeneous behavior of the \textit{numerical speed-ups} with respect to the ODE tests, the results obtained for small values of $\Delta x$ (in principle more reliable) are in good agreement with the theoretical predictions.
In general, we remark again that the \ADERdu~is the fastest method for not too coarse meshes.

\subsubsection{Euler equations}
We now consider the 1D Euler equations
\begin{equation}
	\partial_t \begin{pmatrix}
		\rho \\
		\rho v \\
		E
	\end{pmatrix} + \partial_x 
	\begin{pmatrix}
		\rho v \\
		\rho v^2 + p \\
		(E + p)v
	\end{pmatrix} = 0 ,
\end{equation}
where $\rho$ is the mass density, $v$ the velocity, $E = e + \frac{1}{2}\rho v^2$ the total energy, equal to the sum of the internal energy density $e$ and the kinetic energy density. The system is closed with the equation of state of an ideal gas $p = (\gamma-1)e$, with $\gamma$ the constant adiabatic index,  set at $\gamma=1.4$.

We consider first a nonlinear sound wave test case,  as in \cite{velasco2022spectral},  that consists of a nonlinear acoustic perturbation over a uniform equilibrium state
\begin{equation}
	\label{eq:sound_wave_ic}
	\begin{split}
		\rho_0(x) &= 1 + A\sin(kx)/c_{s,0}, \\
		v_0(x) &= A\sin(kx),    \\
		p_0(x) &= 1 + \gamma p_0 A\sin(kx) /c_{s,0},
	\end{split}
\end{equation}
in the spatial computational domain $[0,1]$, with $c_{s,0}=\sqrt{\gamma}$,  $k= 20\pi$ and $A=10^{-6}$.

The corresponding time-dependent solution for the velocity field has the following analytical solution, valid up to second-order (in a perturbative sense),
\begin{equation*}
	v(x,t) = A\sin(kx-\omega t) + A^2\frac{\gamma+1}{4} \frac{\omega t}{c_{s,0}}\cos{[2(kx-\omega t)]},
\end{equation*}
for $\omega = k\gamma$.

We compute the error of the numerical solution with respect to this analytical solution at the final time $T=1/(c_{s,0}k)$,  corresponding to the time needed by the sound wave to perform one complete orbit over the periodic domain.  The CFL is chosen according to \eqref{eq:sd-cfl}, with $C=0.4$. 

First,  we verify in Figure~\ref{fig:error_euler_wave} the convergence rates for increasing degree of the polynomial basis.  We observe essentially no difference between the different versions of ADER.  
As the errors are identical for all types of schemes, we report directly the \textit{numerical speed-up} as a measure of the computational advantage. In Figure~\ref{fig:speed_up_euler_wave}, \ADERdu~outperforms all other methods reaching \textit{numerical speed-up} factors of the order of the theoretical ones in Section~\ref{sec:RK}. Also in this case, there is more variance in these results with respect to the ODEs ones and the authors believe this may be due to the implementation of the ADER-SD method using Python packages and functions as \texttt{einsum} of \texttt{numpy}, where, for instance, the role of the memory layout plays a big role in the computational costs.
Nevertheless, just like in the previous test, for small values of $\Delta x$ the \textit{numerical speed-ups} get closer to the analytical ones.

\begin{figure}
	\includegraphics[width=1.0\textwidth]	{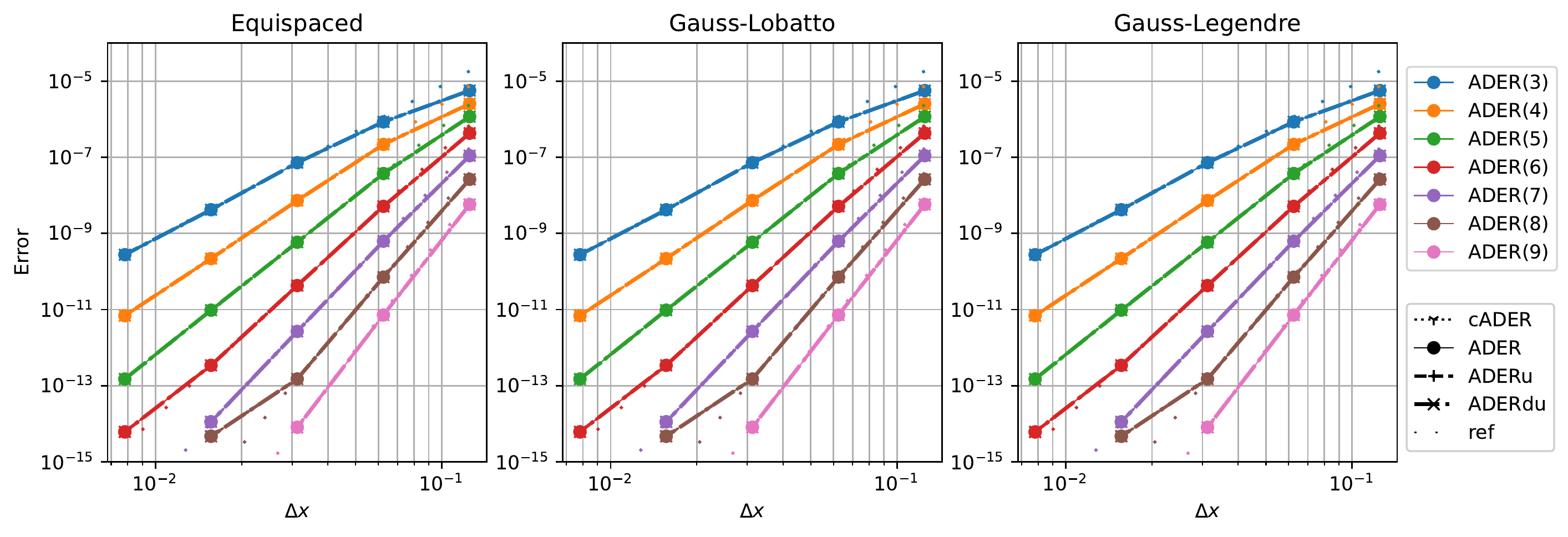}
	\caption{1D Euler equations with sound wave initial condition \eqref{eq:sound_wave_ic}: Error decay for various methods and orders. The ``ref'' line is the reference for every order of accuracy}
	\label{fig:error_euler_wave}
\end{figure}

\begin{figure}
	\includegraphics[width=1.0\textwidth]	{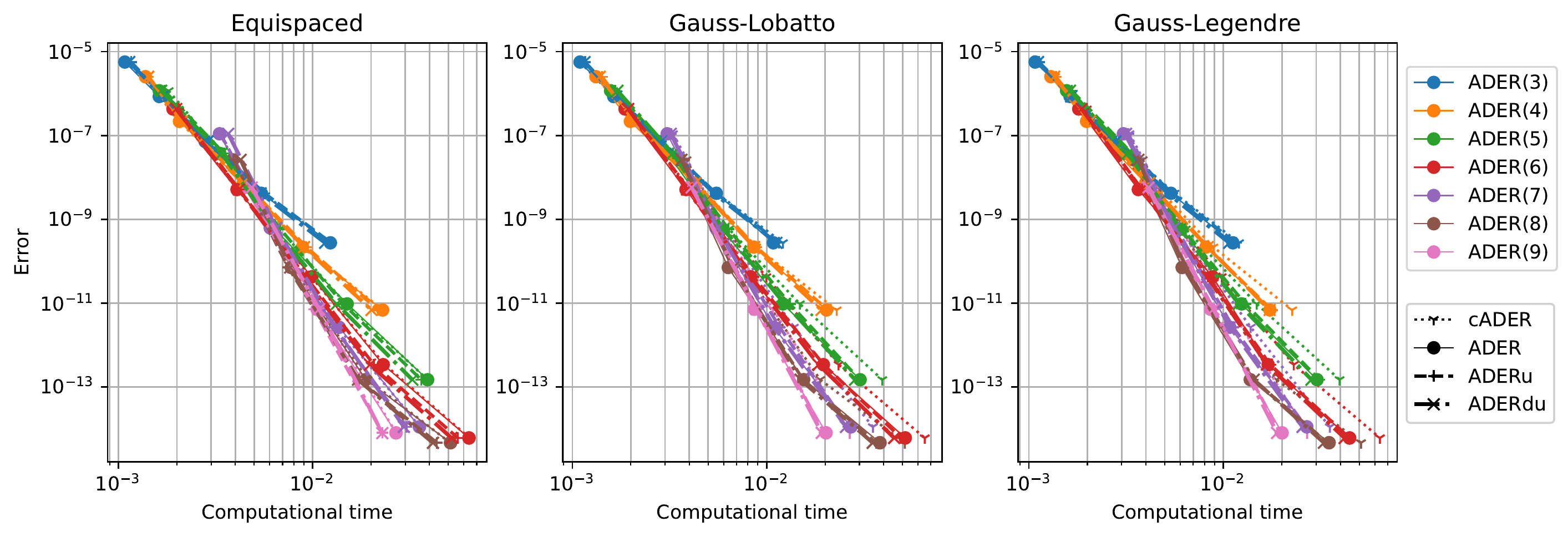}
	\caption{1D Euler equations with sound wave initial condition \eqref{eq:sound_wave_ic}: Error with respect to computational time}	
	\label{fig:error_time_euler_wave}
\end{figure}

\begin{figure}
	\includegraphics[width=1.0\textwidth]	{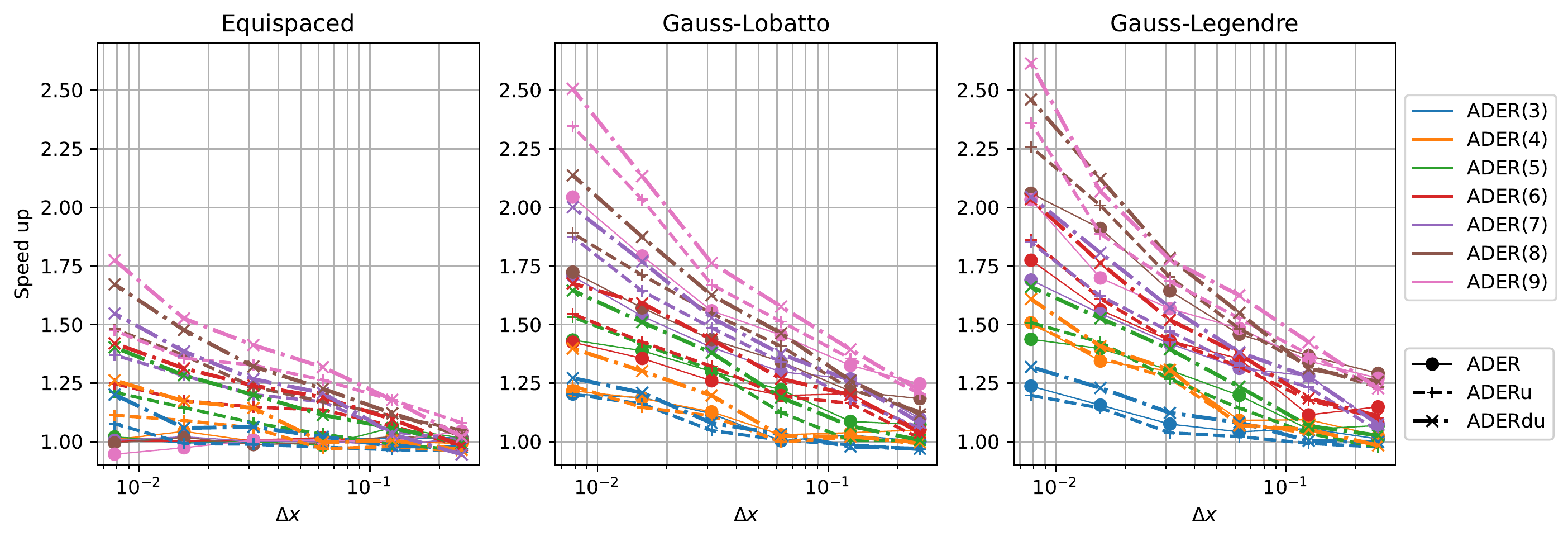}
	\caption{1D Euler equations  with sound wave initial condition \eqref{eq:sound_wave_ic}: \textit{Numerical speed-up} factor with respect to \cADER~method computed as the computational time of the \cADER~over the computational time of the method in consideration}
	\label{fig:speed_up_euler_wave}
\end{figure}

The next test is the well known Sod shock tube problem \cite{sod1978}, characterized by the following initial condition
\begin{align}
	\label{eq:sod}
	\rho_0(x)= \left\{
	\begin{array}{ll}
		1, & x < 0.5, \\
		0.1, & x\geq 0.5,
	\end{array} 
	\right.  \quad  v_0(x) = 1, \quad p_0(x)= \left\{
	\begin{array}{ll}
		1, & x < 0.5, \\
		0.125, & x\geq 0.5,
	\end{array} 
	\right. 
\end{align}
in the spatial computational domain $[0,1]$ and zero gradient boundary conditions.  
We let the numerical solution evolve until the final time $T = 0.2$.  In order to tackle the discontinuities occurring in the solution, we adopt an \textit{a posteriori} limiting strategy similar to the one presented in \cite{velasco2022spectral} and described in detail in \ref{app:limiter}.

In this case, we do not perform convergence analysis as the convergence order would be at most 1, due to the mentioned discontinuities. 
However,  we can notice, from Figure \ref{fig:solution_euler_sod}, that the quality of the solution improves as the (formal) order of the discretization increases and that there is no noticeable difference between the quality of the solution for the different versions of the ADER scheme.
In Figure~\ref{fig:speed_up_euler_sod},  we report the numerical speed-ups obtained for this particular test.
As expected, we observe that \ADERu, \ADERdu~and ADER outperform \cADER. Further, \ADERu~and \ADERdu~outperform ADER. 
Again, as in the previous PDE tests, we observe a dependency on the grid size $\Delta x$, with larger gains when the mesh is more refined.
Let us notice that the computational complexity of the limiter is directly related to the number of subtimenodes and this implies further advantages in ADER, \ADERu~and \ADERdu~with respect to \cADER.

\begin{figure}
	\includegraphics[width=1.0\textwidth]	{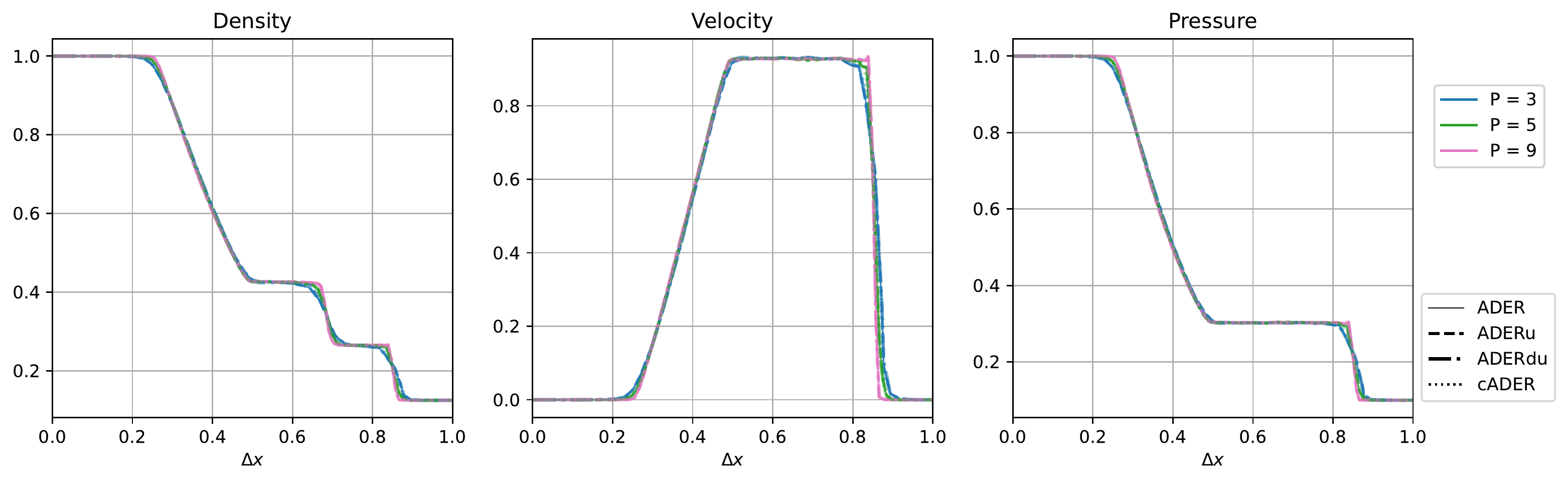}
	\caption{1D Euler equations  with Sod shock tube initial condition \eqref{eq:sod}: Numerical solutions computed for different orders $P$ for a fixed resolution characterized by $16$ elements}
	\label{fig:solution_euler_sod}
\end{figure}

\begin{figure}
	\includegraphics[width=1.0\textwidth]	{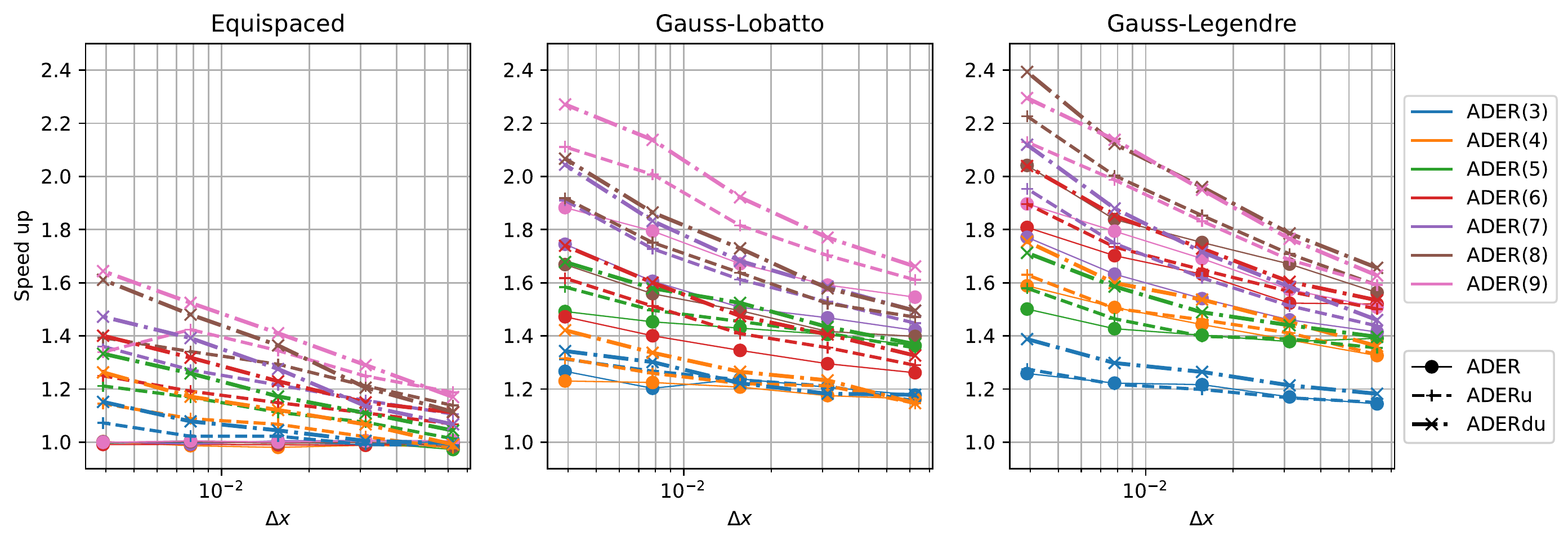}
	\caption{1D Euler equations with Sod shock tube initial condition \eqref{eq:sod}: \textit{Numerical speed-up} factor with respect to \cADER~method computed as the computational time of the \cADER~method over the computational time of the method in consideration}
	\label{fig:speed_up_euler_sod}
\end{figure}

\section{Conclusions and further developments}
\label{sec:Conclusions}
Summarizing, in this work we have showed different techniques to save computational times when adopting ADER methods.  In the first part of the paper, we have shown how the weak formulations can be optimally discretized using a minimal reconstruction degree to maximize the order of accuracy. 
Then, we have designed efficient methods, based on increasing the polynomial degree of the reconstructed numerical solution at each iteration, saving computational time especially in early stages. 
This allowed us to easily set up $p$-adaptive versions of the new schemes, where the iterations can be stopped on-the-fly when certain criteria are met, e.g., the matching of an error tolerance.
The whole presentation is accompanied by theoretical and numerical analysis that validate the proposed methods, showing strong improvements in the computational times without degradation of the accuracy nor stability of the methods.

We believe these strategies can have a strong impact in the community that uses these algorithms, as they allow to speed-up the simulations, save computational resources and they give new hints for adaptive methods. 
From this work, new research directions arise, both in the ODE and hyperbolic PDE framework, namely, for hp-adaptive methods, implicit and structure preserving schemes.

\section*{Acknowledgments}
L. Micalizzi has been funded by the SNF grant 200020\_204917 ``Structure preserving and fast methods for hyperbolic systems of conservation laws'' and by the Forschungskredit grant FK-21-098.
D. Torlo has been funded by a SISSA Mathematical Fellowship.  M.  Han Veiga has been funded by the Michigan Institute for Data Science (MIDAS) and the Van Loo Postdoctoral fellowship at the University of Michigan.
\section*{Declarations of interest}
None.

\appendix

\section{Proofs}\label{app:proofs}
In this section, we report the proofs of several ``minor'' propositions and theorems that have been presented throughout the paper. 

\subsection{Proof of Proposition \ref{prop:r}}\label{app:r}
\begin{proof}
	\begin{subequations}
		We can equivalently prove that
		\begin{equation}
			\uvec{r}=B\undu_n.
			\label{eq:rBun}
		\end{equation}
		By a direct computation of the $\ell$-th component of both sides of the previous equation, thanks to the fact that the Lagrange functions are such that $\sum_{m=0}^M \psi^m\equiv 1$, we have
		\begin{align}
			\begin{split}
				\sum_{m=0}^M B_{\ell,m} \uvec{u}_n &=\sum_{m=0}^M \Bigg[ \psi^\ell(t_{n+1})\psi^m(t_{n+1})- \int_{t_{n}}^{t_{n+1}} \left(\frac{d}{dt}\psi^\ell(t)\right)\psi^m(t)dt \Bigg]\uvec{u}_n \\
				&= \psi^\ell(t_{n+1}) \uvec{u}_n- \int_{t_{n}}^{t_{n+1}} \left(\frac{d}{dt}\psi^\ell(t) \right)dt \uvec{u}_n\\
				&= [\psi^\ell(t_{n+1})-\psi^\ell(t_{n+1})+\psi^\ell(t_{n})]\uvec{u}_n=\psi^\ell(t_{n})\uvec{u}_n =\uvec{r}_{\ell},
			\end{split}
		\end{align}
		which implies \eqref{eq:rBun}.
	\end{subequations}
\end{proof}

\subsection{Proof of Proposition \ref{prop:iterative_procedure}: Convergence of the iterative procedure}\label{app:iterative_procedure}
\begin{proof}
	\begin{subequations}
		The proof relies on the definition of the map $\mathcal{J}:\mathbb{R}^{(M+1)\times Q} \rightarrow \mathbb{R}^{(M+1)\times Q}$, defined by
		\begin{equation}
			\mathcal{J}(\undu):=\undu_n+\Delta t B^{-1}\Lambda  \underline{\uvec{G}}(\undu), \quad \forall \undu \in \mathbb{R}^{(M+1)\times Q}.
		\end{equation} 
		We will now show that it is a contraction. We take two general vectors $\undv, \undw \in \mathbb{R}^{(M+1)\times Q}$ and, from the definition of $\underline{\uvec{G}}$ and the Lipschitz-continuity of $\uvec{G}$ with respect to $\uvec{u}$ uniformly with respect to $t$, basic computations give
		\begin{align}
			\begin{split}
				\norm{\mathcal{J}(\undv)-\mathcal{J}(\undw)}_\infty=\norm{ \Delta t B^{-1}\Lambda \left[ \underline{\uvec{G}}(\undv)-\underline{\uvec{G}}(\undw) \right]}_\infty
				\leq \Delta t \norm{B^{-1}\Lambda }_\infty C_{Lip} \norm{\undv-\undw }_\infty.
			\end{split}
		\end{align}
		The entries of $B$ and $\Lambda$, as well as $C_{Lip}$, are constants independent of $\Delta t$, therefore, for $\Delta t<\frac{1}{\widetilde{C}_{Lip}}$ with $ \widetilde{C}_{Lip}:=\norm{B^{-1}\Lambda }_\infty C_{Lip}$ we have that $\mathcal{J}$ is a contraction over $\mathbb{R}^{(M+1)\times Q}$ with respect to the infinity norm. 
		Thus, thanks to the Banach fixed-point theorem, we have that for $\Delta t<\frac{1}{\widetilde{C}_{Lip}}$ the map $\mathcal{J}$ has a unique fixed point, which can be obtained as the limit of the iteration $\undu^{(p)}:=\mathcal{J}(\undu^{(p-1)})$, independently of the choice of $\undu^{(0)}$. This fixed-point iteration is equivalent to the iterative procedure \eqref{eq:ADER_Picard}.
		Observing that a fixed point of $\mathcal{J}$ is also a solution of the nonlinear system \eqref{eq:ADER_system_final} and vice versa, we get the desired result.
	\end{subequations}
\end{proof}

\subsection{Proof of Theorem \ref{th:invB}: Invertibility of $B$}\label{app:invertB} 
\begin{proof}
	As shown in Theorem~\ref{th:link} and more precisely in Equation \eqref{eq:equivalence_B_matrices2}, the matrix $B$ exactly integrated, for any generic basis $\left\lbrace\hphi(\xi)\right\rbrace_{m=0,\dots,M}$, is equivalent to the matrix $B_{GL*}$ of the ADER method with $GL*$ subtimenodes up to the multiplication by the change of basis matrix $\mathcal{H}$ and its transpose, i.e., $B=\mathcal{H}^T B_{GL*}\mathcal{H}$. 
	
	Since the matrix $\mathcal{H}$ is invertible, we can prove the invertibility of $B$ for a particular basis and obtain the desired result for all the other bases. 
	Let us consider the modal basis functions $\hphi^m (\xi) := \xi^m$ for $m=0,\dots, M$. A direct computation shows that, in such a case, the matrix $B$ is defined as
	\begin{align}
		B_{\ell,m} = \hphi^\ell(1)\hphi^m(1) - \int_{0}^1 \left(\frac{d}{d\xi}\hphi^{\ell}(\xi)\right) \hphi^m (\xi)\, d\xi = 
		\begin{cases}
			1, \quad &\text{if}~\ell=0,\\
			\frac{m}{\ell+m}, \quad &\text{if}~\ell=1,\dots,M.\\
		\end{cases}
	\end{align}
	The first row of the matrix $B$ is given by $B_{0,m}=1$ for all $m=0,\dots,M$, while, the first column by $B_{\ell,0}=0$ for $\ell=1,\dots,M$. 
	Hence, the determinant of the matrix is equal to the determinant of the submatrix $B_{1:,1:}$ and we can focus on it. In particular, we will show that it is nonsingular and, therefore, that its determinant is nonzero.
	
	From basic linear algebra, a matrix is nonsingular if and only if its columns are linearly independent.
	Furthermore, the columns of a matrix are linearly independent if and only if multiplied by nonzero factors they are still linearly independent.
	Hence, let us consider the matrix $\widetilde{B}$ with generic entry given by $\widetilde{B}_{\ell,m}:=\frac{1}{m}B_{\ell,m}=\frac{1}{\ell+m}$ for $\ell,m=1,\dots,M$. 
	Such matrix is the so-called ``Hilbert'' matrix, which is invertible \cite{choi1983tricks}, and so also $B_{1:,1:}$ is invertible.
	Hence, $B$ has nonzero determinant and is invertible.
\end{proof}

\subsection{Proof of Theorem \ref{th:l2isRK}: The \ADERIWF~is an implicit RK}\label{app:l2isRK} 

\begin{proof}
	\begin{subequations}
		We have already observed that the nonlinear systems \eqref{eq:ADER_system_final} and \eqref{eq:weakproblemdiscrete} are equivalent and, identifying $\uvec{u}^m$ with RK stage values $\uvec{y}^s$, it is clear that \eqref{eq:ADER_system_final} is 
		the nonlinear system of a RK method, expressed by the first equation in \eqref{eq:RK_ACCURACY_NODAL}, characterized by $A=B^{-1}\Lambda$ and $\uvec{c}=\vecbeta$.
		We are left to check that the reconstruction formula \eqref{eq:uh} can be written as the RK final update, expressed by the second equation in \eqref{eq:RK_ACCURACY_NODAL}, with coefficients $b_m=w_m$.
		Manipulating the nonlinear system \eqref{eq:weakproblemdiscrete}
		we get 	\begin{align}
			\begin{split}
				\sum_{m=0}^M \Bigg[ \psi^{\ell}(t_{n+1})\psi^m(t_{n+1}) & -\int_{t_n}^{t_{n+1}} \left(\frac{d}{dt}\psi^\ell(t)\right)\psi^m(t)dt   \Bigg]\uvec{u}^{m}=\psi^\ell(t_n)\uvec{u}_n \\
				&+ \sum_{m=0}^M \left( \int_{t_n}^{t_{n+1}} \psi^\ell(t)\psi^m(t) dt \right) \uvec{G}(t^m,\uvec{u}^{m}) , \quad \ell=0,\dots,M,
			\end{split}
			\label{eq:weakproblemdiscrete2}
		\end{align}
		and, summing over $\ell$ and recalling the fact that $\sum_{\ell=0}^M \psi^{\ell} \equiv 1$, we obtain
		\begin{align}
			\begin{split}
				\uvec{u}_h(t_{n+1})=\sum_{m=0}^M  \psi^m(t_{n+1}) \uvec{u}^{m} =\uvec{u}_n + \sum_{m=0}^M \left( \int_{t_n}^{t_{n+1}} \psi^m(t) dt \right) \uvec{G}(t^m,\uvec{u}^{m}) =  \uvec{u}_n + \Delta t\sum_{m=0}^M w_m \uvec{G}(t^m,\uvec{u}^{m}).
			\end{split}
		\end{align}	
		
	\end{subequations}
\end{proof}

\subsection{Proof of Proposition \ref{prop:HORK_condition}}\label{app:HORK_condition} 
\begin{proof}
	\begin{subequations}
		Proving \eqref{eq:semirequiredforRK} is equivalent to prove that $A\uvec{1}=\uvec{c}$, where $\uvec{1}$ is a vector with all entries equal to $1$. Knowing that $A=B^{-1}\Lambda$ and $\uvec{c}=\vecbeta$, we can prove \eqref{eq:semirequiredforRK} by showing that $\Lambda\uvec{1}=B\vecbeta.$
		
		By a direct computation of the $\ell$-th component of both the sides, since $\sum_{m=0}^M \widehat{\psi}^m\equiv 1$, we have
		\begin{align}
			\left(\Lambda\uvec{1}\right)_{\ell}&=\sum_{m=0}^M \int_{0}^{1} \widehat{\psi}^\ell(\xi)\widehat{\psi}^m(\xi)d\xi=\int_{0}^{1} \widehat{\psi}^\ell(\xi)d\xi,\\
			\left(B\vecbeta\right)_{\ell}&=\sum_{m=0}^M \Bigg[ \widehat{\psi}^\ell(1)\widehat{\psi}^m(1)- \int_{0}^{1}\left(\frac{d}{d\xi}\widehat{\psi}^\ell(\xi)\right) \widehat{\psi}^m(\xi)d\xi \Bigg]\xi^m \nonumber\\
			&=\widehat{\psi}^\ell(1)\left( \sum_{m=0}^M  \xi^m \widehat{\psi}^m(1) \right)- \int_{0}^{1}\left(\frac{d}{d\xi}\widehat{\psi}^\ell(\xi)\right)\left( \sum_{m=0}^M \xi^m\widehat{\psi}^m(\xi) \right)d\xi. 
		\end{align}		
		Let us focus on $\left(B\vecbeta\right)_{\ell}$. Since $ \sum_{m=0}^M  \xi^m \widehat{\psi}^m$ is nothing but the interpolation of the linear function $\xi$, which is exact when interpolated in at least $2\leq M+1$ nodes, we can write
		\begin{align}
			\left(B\vecbeta\right)_{\ell}=  \widehat{\psi}^\ell(1)\cdot 1- \int_{0}^{1}\left(\frac{d}{d\xi}\widehat{\psi}^\ell(\xi)\right)\xi d\xi,
		\end{align}
		and integrating by parts we obtain
		\begin{align}
			\left(B\vecbeta\right)_{\ell}= \widehat{\psi}^\ell(1)\cdot 1- \left[\widehat{\psi}^\ell(1)\cdot 1-\widehat{\psi}^\ell(0)\cdot 0 \right]+ \int_{0}^{1}\widehat{\psi}^\ell(\xi)d\xi=\int_{0}^{1}\widehat{\psi}^\ell(\xi)d\xi,
		\end{align}
		and, thus, the desired result. 
	\end{subequations}
\end{proof}

\subsection{Proof of Proposition \ref{prop:ADERRK_Mp1}}\label{app:ADERRK_Mp1} 
\begin{proof}
	\begin{subequations}
		By direct computation, we will estimate the order of magnitude of the local truncation error.
		The properties of interpolation with $M+1$ subtimenodes guarantee that
		\begin{align}
			\uvec{u}(t)=\sum_{m=0}^M\uvec{u}(t^{m}) \psi^m(t)+O(\Delta t ^{M+1}),
			\qquad \uvec{G}(t,\uvec{u}(t))=\sum_{m=0}^M\uvec{G}(t^m,\uvec{u}(t^{m})) \psi^m(t)+O(\Delta t ^{M+1}).
		\end{align}				
		Thus, if we insert the exact solution of the ODEs system \eqref{eq:ODE} in the \ADERIWF~\eqref{eq:weakproblemdiscrete}, we get
		\begin{align}
			\begin{split}
				\sum_{m=0}^M \Bigg[ \psi^{\ell}(t_{n+1})\psi^m(t_{n+1}) & -\int_{t_n}^{t_{n+1}} \left(\frac{d}{dt}\psi^\ell(t)\right) \psi^m(t) dt   \Bigg]\uvec{u}(t^{m})-\psi^\ell(t_n)\uvec{u}_n \\
				&- \sum_{m=0}^M \left( \int_{t_n}^{t_{n+1}} \psi^\ell(t)\psi^m(t) dt \right) \uvec{G}(t^m,\uvec{u}(t^{m}))\\
				=\psi^{\ell}(t_{n+1})\uvec{u}(t_{n+1}) &-\int_{t_n}^{t_{n+1}} \left(\frac{d}{dt}\psi^\ell(t)\right)\uvec{u}(t)dt   -\psi^\ell(t_n)\uvec{u}_n \\
				&-   \int_{t_n}^{t_{n+1}}\psi^\ell(t)\uvec{G}(t,\uvec{u}(t))  dt+O(\Delta t^{M+1})\\
				=\int_{t_n}^{t_{n+1}}\psi^\ell(t)\Bigg[ \frac{d}{dt}\uvec{u}(t) &-\uvec{G}(t,\uvec{u}(t)) \Bigg] dt+O(\Delta t^{M+1}),\quad \ell=0,\dots,M.
			\end{split}
			\label{eq:local_truncation}
		\end{align}
		Hence, the coefficients $\uvec{u}^{m}$ are $O(\Delta t^{M+1})$ accurate with respect to the exact values $\uvec{u}(t^m)$.
		In the proof of Theorem \ref{th:l2isRK}, we showed that the final interpolation step to get $\uvec{u}_{n+1}$ is equivalent to the integration
		\begin{align}
			\begin{split}
				\uvec{u}_{n+1}=\uvec{u}_n + \sum_{m=0}^M \left( \int_{t_n}^{t_{n+1}} \psi^m(t) dt \right) \uvec{G}(t^m,\uvec{u}^{m})=\uvec{u}_n +   \int_{t_n}^{t_{n+1}}\left(\sum_{m=0}^M \uvec{G}(t^m,\uvec{u}^{m}) \psi^m(t)\right) dt  .
			\end{split}
		\end{align}
		So, recalling that we are assuming $\uvec{u}_{n}=\uvec{u}(t_n)$, we have
		\begin{align}
			\begin{split}
				\uvec{u}_{n+1}&=\uvec{u}_n +  \int_{t_n}^{t_{n+1}} \left[\uvec{G}(t,\uvec{u}(t))+O(\Delta t^{M+1}) \right] dt  
				=\uvec{u}(t_n) +  \int_{t_n}^{t_{n+1}}\uvec{G}(t,\uvec{u}(t)) dt+O(\Delta t^{M+2})\\
				&= \uvec{u}(t_{n+1})+O(\Delta t^{M+2}),
			\end{split}
		\end{align}
		which concludes the proof.		
	\end{subequations}
\end{proof}		

\section{Equivalence of \ADERRK-GLB and Lobatto IIIC}\label{app:GLB_Lobatto_IIIC}
In this section, we show the equivalence between  \ADERRK-GLB and Lobatto IIIC schemes. First of all, let us recall how Lobatto IIIC are defined.
\begin{theorem}[Chipman 1971 \cite{chipman1971stable}]\label{th:uniquenessLobatto}
	The Lobatto IIIC scheme with $S$ stages is uniquely defined by $S$ GLB quadrature points and associated weights, respectively as RK coefficients $\uvec{c}$ and $\uvec{b}$, and coefficients $a_{i,j}$ determined imposing the conditions 
	\begin{equation}\label{eq:conditionLobattoIIIC}
		a_{i,0}=b_0, \qquad \text{for } i=0,\dots,S-1
	\end{equation}
	and $\Cp(S-1)$, with definition given in \eqref{eq:condRKC}. 
	
\end{theorem}

Then, we can show that also \ADERRK-GLB satisfies the same conditions.
\begin{lemma}\label{lem:conditionLobattoIIIC}
	\ADERRK-GLB satisfies condition \eqref{eq:conditionLobattoIIIC}.
\end{lemma}
\begin{proof}
	\begin{subequations}
		Let us recall the RK structures of the \ADERRK~methods, given in Theorem \ref{th:l2isRK} 
		\begin{align}
			A:=B^{-1}\Lambda, \quad \uvec{c}:=\vecbeta, \quad  b_m:=\frac{1}{\Delta t}\int_{t_n}^{t_{n+1}}\psi^{m}(t)dt=\int_{0}^{1}\widehat{\psi}^{m}(\xi)d\xi=w_m.
		\end{align}	
		
		Assuming GLB subtimenodes and quadrature points (exact for polynomials of degree $2M-1$), we have that the matrix $B$ is computed exactly, while $\Lambda$ is under-integrated (as its terms involve integrals of polynomials of degree $2M$), leading to a diagonal $\Lambda$
		\begin{align}
			\begin{cases}
				B_{\ell,m}:=\widehat{\psi}^\ell(1)\widehat{\psi}^m(1)- \int_{0}^{1} \left(\frac{d}{d\xi}\widehat{\psi}^\ell(\xi)\right)\widehat{\psi}^m(\xi)d\xi,\\
				\Lambda_{\ell,m} : = w_\ell \delta_{\ell,m}= \sum_{k=0}^M w_k \widehat{\psi}^\ell(\xi_k)\widehat{\psi}^m(\xi_k) \approx \int_0^1 \widehat{\psi}^\ell(\xi)\widehat{\psi}^m(\xi) d \xi.
			\end{cases}
		\end{align}

		To prove that $a_{i,0}=b_0$ for all $i=0,\dots,S-1$, we have to prove that
		\begin{align*}
			(B^{-1}\Lambda)_{i,0} \stackrel{!}{=} w_0 \quad
			\Longleftrightarrow\quad (B^{-1}\Lambda)_{:,0} \stackrel{!}{=} 
			w_0 \underline{1}\quad
			\Longleftrightarrow\quad B^{-1}\Lambda_{:,0} \stackrel{!}{=}
			w_0 \underline{1}
			\quad
			\Longleftrightarrow  \quad  \Lambda_{:,0} \stackrel{!}{=}
			w_0 B \underline{1} ,
		\end{align*}
		where the symbol colon ``:'' is used to select all the elements of a given row or column.	
		
		This amounts to show that $\sum_{m=0}^{M} B_{\ell,m} w_0 = \delta_{\ell, 0} w_0$,
		which can be easily proven exploiting the definition of $B$ and the fact that $\sum_{m=0}^{M} \widehat{\psi}^{m} \equiv 1$:
		\begin{align}
			\sum_{m=0}^{M} B_{\ell,m} w_0 &=\widehat{\psi}^\ell(1) \cdot w_0- \int_{0}^{1} \left(\frac{d}{d\xi}\widehat{\psi}^\ell(\xi)\right)d\xi w_0=[\widehat{\psi}^\ell(1)-\widehat{\psi}^\ell(1)+\widehat{\psi}^\ell(0)]w_0 =  \widehat{\psi}^\ell(0) w_0 .
		\end{align}
		Since only the first GLB basis function, for $\ell=0$, is $1$ in $\xi=0$, while the other ones have value $0$ there, \eqref{eq:conditionLobattoIIIC} holds.
	\end{subequations}
\end{proof}

\begin{theorem}
	\ADERRK-GLB methods are Lobatto IIIC methods.
\end{theorem}
\begin{proof}
	\ADERRK-GLB methods are such that: $\uvec{c}$ and $\uvec{b}$ coincide with the GLB quadrature points and weights, $\Cp(S-1)$ holds for for Lemma~\ref{lem:conditionsRKC_ADER}, while Lemma \ref{lem:conditionLobattoIIIC} guarantees condition \eqref{eq:conditionLobattoIIIC}.
	Hence, \ADERRK-GLB methods satisfy the properties of Theorem~\ref{th:uniquenessLobatto}, which uniquely characterize the Lobatto IIIC methods. Thus, the two methods coincide.
\end{proof}

\section{\textit{A posteriori} limiting strategy}\label{app:limiter}
In this appendix, we describe the shock capturing strategy adopted to perform the last numerical test for PDEs characterized by shocks. 
The high-order space discretization is achieved via the SD scheme,  presented in Section~\ref{sec:SD}.  We adopt an \textit{a posteriori} limiting strategy similar to the one described in \cite{velasco2022spectral} and references therein.

The general methodology of the \textit{a posteriori} limiting consists in correcting the high-order fluxes,  that are responsible for oscillatory behavior,  after the computations of one time step have been completed. 
This is in contrast with \textit{a priori} limiting,  which modifies the high-order flux during the time step computation.  
In order to identify which fluxes must be corrected, we consider two physical criteria 
\begin{enumerate}
	\item positiveness of density and pressure;
	\item discrete maximum principle for the density.
\end{enumerate} 

As shown in \cite{velasco2022spectral}, the SD method is equivalent to a Finite Volume explicit Euler method using high-order fluxes $\widehat{\uvec{f}}^{HO,m}_i$ and a nonuniform (space-time) subcell grid, such that the update \eqref{eq:semidiscretizazion_SD} for the final iteration after ADER discretization can be equivalently written as
\begin{align}
	\bar{\uvec{u}}^{-1}_{i} &= \uvec{u}_{n,i},\\
	\tilde{\uvec{u}}^{m}_i &= \bar{\uvec{u}}^{m-1}_{i} - (t^m - t^{m-1})   \frac{\widehat{\uvec{f}}^{HO,m-1}_{i+1/2}-\widehat{\uvec{f}}^{HO,m-1}_{i-1/2}}{\Delta x_i},\qquad  \text{ for }m=0,\dots,M,\label{eq:update_bar_subcell}\\
	\uvec{u}_{n+1,i} &= \uvec{u}_{n,i} - \Delta t \sum_{m=0}^M w_m \partial_x \uvec{F}^{(P)}_h(x_i^s,t^m) =  \\
	&=\bar{\uvec{u}}^{M}_{i} - (t_{n+1} - t^{M})   \frac{\widehat{\uvec{f}}^{HO,M}_{i+1/2}-\widehat{\uvec{f}}^{HO,M}_{i-1/2}}{\Delta x_i} \label{eq:subcellHO_final_update}\\
	&=\uvec{u}_{n,i} - \sum_{m=0}^M (t^m - t^{m-1})   \frac{\widehat{\uvec{f}}^{HO,m-1}_{i+1/2}-\widehat{\uvec{f}}^{HO,m-1}_{i-1/2}}{\Delta x_i},
\end{align}
with the convention of $t^{-1}=t_n$.

Hence,  the limiting strategy is equivalent to replacing the high-order fluxes $\widehat{\uvec{f}}^{m}_{i+1/2}$ in the subcells that trigger the aforementioned criteria with a low-order flux $\uvec{f}^{LO,m}_{i+1/2}$. In this work, we use a simple first order Finite Volume scheme on the (space-time) subcell grid, referred as \textit{parachute scheme}.

The space-time update of the SD-ADER method with limiting is shown in Algorithm~\ref{algo:update}.  For further details, the interested reader is referred to \cite{velasco2022spectral}.
Finally, we point out that the cost of the \textit{a posteriori} limiter is proportional to the number of subtimenodes.

\begin{algorithm}
	\caption{Timestep evolution of SD-ADER with the \textit{a posteriori} limiting}
	\label{algo:update}
	\SetAlgoLined
	\textbf{Input:} Numerical solution $\uvec{u}_n$ at $t_n$,  $\Delta t$ \\
	\textbf{Output:} Numerical solution $\uvec{u}_{n+1}$ at $t_{n+1}$\\
	Initialize solution in the subtimenodes $\uvec{u}^{m,(0)} = \uvec{u}_n$ for $m=0,...,M$\\
	Compute the high order update of the ADER-SD scheme up to the final iteration $\uvec{u}^{m,(P)}$ for all $m=0,\dots,M$\\
	Convert the high-order fluxes into the subcell version $\widehat{\uvec{f}}_{i+1/2}^{HO,m}= \widehat{\uvec{f}}_{i+1/2}^{HO,m}(\uvec{u}^{m,(P)})$ for $m=0,...,M$, $\forall i$,\\
	\For{$m = 0, ... ,M$}{
		Compute low-order fluxes using solution $\bar{\uvec{u}}^m$ with parachute scheme\\
		Compute high-order candidate solution (shown $i$-th subcell) as in \eqref{eq:update_bar_subcell}\\
		Perform physical criteria check\\
		\If{$\tilde{\uvec{u}}^{m+1}_{i}$ {\rm is~\emph{troubled}}}
		{
			Replace $\widehat{\uvec{f}}^{HO,m}_{i\pm1/2}$ with fluxes of parachute scheme fluxes $\widehat{\uvec{f}}^{LO,m}_{i\pm1/2}$ in  \eqref{eq:update_bar_subcell}\\
			Replace surrounding fluxes $\widehat{\uvec{f}}^{HO,m}_{i\pm3/2}$ using parachute scheme fluxes $\widehat{\uvec{f}}^{LO,m}_{i\pm3/2}$ in  \eqref{eq:update_bar_subcell}\\
		}
	}
	Compute $\uvec{u}_{n+1,i}$ using \eqref{eq:subcellHO_final_update} with HO or LO fluxes according to the previous criteria.\\
\end{algorithm}

\section{Table of notation adopted in the paper}\label{app:notation}
In Table~\ref{tab:symbols}, we summarize the notation adopted for the main structures in the context of the ADER methods throughout the whole paper.

\begin{table}
	\begin{tabularx}{1.\textwidth}{|l|>{\centering\arraybackslash}X|}
		\hline
		$t_n$ & timenode \\
		\hline
		$\uvec{u}_n$ &  approximation of the solution of \eqref{eq:ODE} in $t_n$ \\
		\hline
		$\Delta t$ & timestep \\
		\hline
		$N$ & accuracy of the \ADERIWF~\eqref{eq:weakproblemdiscrete}\\ \hline
		$P$ & number of ADER iterations \\
		\hline
		$M+1$ & number of subtimenodes \\
		\hline
		$t^m, \, m=0,\dots,M$ & subtimenode in the interval $[t_n,t_{n+1}]$ \\
		\hline
		$\uvec{u}^m$ & reconstruction coefficients of solution of \eqref{eq:ODE} \\
		\hline
		$\psi^m$ & Lagrange basis function associated to $t^m$ \\
		\hline
		$\widehat{\psi}^m$ & Lagrange basis function associated to $\xi^m$ onto $[0,1]$ \\
		\hline
		$\xi^m$ & subtimenode $t^m$ remapped into the reference interval $[0,1]$ \\
		\hline
		$\omega_m$ & quadrature weight associated to $\widehat{\psi}^m$ \\
		\hline
		$\phi^m$ & generic basis function in the interval $[t_{n},t_{n+1}]$ \\
		\hline
		$\widehat{\phi}^m$ & generic basis function remapped into the reference interval $[0,1]$ \\
		\hline
		$\underline{\beta}$ & vector of the subtimenodes $\xi^m$\\
		\hline
		$\uvec{u}_h$ & interpolation of the numerical solution in the interval $[t_n,t_{n+1}]$ \\
		\hline
		$B,\Lambda$ & ADER matrices  \\
		\hline
		$\undu,\undr, \undG(\undu),\undu_{n}, \undv, $ & ADER vectors  \\
		\hline
		$\undu^{(p)}, \, p=0,\dots,P$ & vector of the approximations coefficients at the iteration $p$ \\
		\hline
		$\mathcal{H}$ & change of basis matrix \\
		\hline
		$S$ & number of RK stages \\
		\hline 
		$\uvec{y}^s, \, s=0,\dots,S-1$ & RK stage values \\
		\hline
		$a_{s,r}, b_{r}, c_{r}$ & RK coefficients \\
		\hline
		$A, \uvec{b}, \uvec{c}$ & RK structures \\
		\hline
		$\lopdt^1,\lopdt^2$  & DeC operators  \\
		\hline
		$\usol$  & solution of the $\lopdt^2$ operator  \\
		\hline
		$\lopdt^{1,(p)},\lopdt^{2,(p)},\mathcal{E}^{(p)},\Pi^{(p)}$  & structures of efficient DeC methods  \\
		\hline
		$\psi^{m,(p)}\quad 0,\dots,M^{(p)}$  & basis functions of ADERu, ADERdu, ADER-$L^2$  \\
		\hline
		$B^{(p)},\Lambda^{(p)},H^{(p)},\undu_n^{(p)},\undu^{*(p)},$  & structures of ADERu, ADERdu, ADER-$L^2$ \\
		$\undG^{*(p)},\Lambda^{(p,p-1)},\undr^{(p)},\underline{\beta}^{(p)}$  &   \\
		\hline
	\end{tabularx}
	\caption{Table of symbols \label{tab:symbols}}
\end{table}

	\bibliographystyle{abbrv}
	\bibliography{sn-bibliography}
	
\end{document}